\documentclass[a4paper]{amsart}
\usepackage[utf8]{inputenc}
\usepackage{amsfonts,amsmath,amsthm,graphicx,mathrsfs}
\usepackage{verbatim}
\usepackage{hyperref}
\usepackage{array,float}
\usepackage[toc,page]{appendix}
\usepackage{caption, subcaption,tikz}
\usepackage{rotating}
\usepackage{todonotes}
\usepackage[a4paper, total={6.5in,8in}]{geometry}
\numberwithin{equation}{section}

\hypersetup{
    colorlinks,
    linkcolor={blue!75},
    citecolor={blue!75},
    urlcolor={blue!75}
}

\allowdisplaybreaks 

\newcommand{\ddt}[1]{\frac{\partial #1}{\partial t}}
\newcommand{\matdev}{\partial^{\bullet}}
\newcommand{\grad}{\nabla}
\newcommand{\Div}{\nabla \cdot}

\newcommand{\esssup}[1]{\underset{#1}{\operatorname{ess \ sup \ }}}

\newenvironment{rcases}
{\left.\begin{aligned}}
	{\end{aligned}\right\rbrace}

\theoremstyle{plain}
\newtheorem{theorem}{Theorem}[section]

\newtheorem{lemma}[theorem]{Lemma}
\newtheorem{definition}[theorem]{Definition}

\newtheorem{remark}[theorem]{Remark}

\newcommand{\mbf}[1]{\mathbf{#1}}
\newcommand{\mbb}[1]{\mathbb{#1}}
\newcommand{\tr}[1]{\operatorname{tr}\left(#1\right)}

\newcommand{\utwo}{\boldsymbol{\mathfrak{u}}^2}
\newcommand{\ptwo}{\boldsymbol{\mathfrak{p}}^2}

\title[An iterative approach to a fluid-rigid body interaction problem]{An iterative approach to a fluid-rigid body interaction problem}
\author[C.~M.~Elliott]{Charles~M.~Elliott$^\star$}
\author[T.~Sales]{Thomas~Sales$^\dagger$}
\address[$\star$]{Mathematics Institute, Zeeman Building, University of Warwick, Coventry CV4 7AL, United Kingdom}
\address[$\dagger$]{Department of Mathematics, University of Sussex, Brighton  BN1 9RF, United Kingdom}
\email{\href{mailto:c.m.elliott@warwick.ac.uk}{c.m.elliott@warwick.ac.uk}($\star$), \href{mailto:t.p.sales@sussex.ac.uk}{t.p.sales@sussex.ac.uk}($\dagger$)}


\begin{document}
\begin{abstract}
	We study a novel approach for the existence of solutions to an incompressible fluid-rigid body interaction problem in three dimensions.
    Our approach introduces an iteration based on a sequence of related problems posed on domains with prescribed evolution.
	In particular we prove the short-time existence of strong solutions to a system coupling the incompressible Navier--Stokes equations to the ordinary differential equations governing the motion of a rigid body, with no slip boundary conditions on the boundary of the rigid body, provided that the relative density, $\frac{\rho}{\rho_B}$, is sufficiently small.
    We also discuss the use of our iterative approach in numerical methods for the moving boundary problem, and complement this with some numerical experiments in two dimensions which demonstrate the necessity of the smallness assumption on $\frac{\rho}{\rho_B}$.
\end{abstract}

\maketitle

\section{Introduction}
\label{section: FSI Intro}
This paper considers a novel iterative approach to construct solutions to a \emph{moving boundary problem} where the incompressible Navier--Stokes equations are coupled to the equations of motion for a rigid body.
The movement of the boundary comes from the motion of the rigid body, and so one must simultaneously solve for both the solutions to the system of differential equations and the (time-dependent) fluid domain.
In particular we study the short time existence of functions $(\mbf{q}, \boldsymbol{\omega}, \mbf{u}, p)$ which solve, in a suitable weak sense,
\begin{subequations}
\label{eqn: FSI fluid}
\begin{gather}
	\rho \left( \ddt{\mbf{u}} + (\mbf{u} \cdot \grad) \mbf{u} \right) = - \grad p + \mu \Delta \mbf{u},\label{FSI1}\\
	\Div \mbf{u} = 0,\label{FSI2}
\end{gather}
\end{subequations}
on a non-cylindrical space-time domain, $\bigcup_{t \in [0,T]} \Omega(t) \times \{ t \} \subset \mbb{R}^3 \times [0,T]$, which is not known a priori.
Here $\rho, \mu$ are the constant density and (dynamic) viscosity of the fluid respectively.
At each time the fluid domain is given by $\Omega(t) := (\Omega \setminus B(t))^\circ$ for $B(t)$ a rigid body with motion determined by its the centre of mass,~$\mbf{q}$, and angular momentum,~$\boldsymbol{\omega}$, solving
\begin{subequations}
\label{eqn: FSI rigid body}
\begin{gather}
	m \frac{\mathrm{d}^2 \mbf{q}}{\mathrm{d}t^2} = - \int_{\partial B(t)} \mbb{T} \boldsymbol{\nu},\label{FSI3}\\
	\frac{\mathrm{d}}{\mathrm{d}t} \left( \mbb{J} \boldsymbol{\omega} \right) = - \int_{\partial B(t)} (\mbf{x} - \mbf{q}(t)) \times \mbb{T} \boldsymbol{\nu}, \label{FSI4}
\end{gather}
\end{subequations}
subject to appropriate initial conditions.
We will often omit the measures in integrals when it is clear from context.
Here $\boldsymbol{\nu}$ denotes the outward unit normal vector, $\mbb{T}$ denotes the viscous stress tensor, and $\mbb{J}$ the inertial tensor,
\begin{gather*}
	\mbb{T} := -p \mbb{I} + \mu \left( \grad \mbf{u} + (\grad \mbf{u})^T \right),\\
	\mbb{J} := \int_{B(t)} \rho_B \left( |\mbf{x}-\mbf{q}(t)|^2 \mbb{I} - (\mbf{x} - \mbf{q}(t)) \otimes (\mbf{x} - \mbf{q}(t)) \right),
\end{gather*}
where $\mbb{I}$ is the identity matrix, and $\rho_B$ is the constant density of the rigid body.
We note that one considers the unit normal vector, $\boldsymbol{\nu}$, to be pointing \emph{out} of the fluid domain $\Omega(t)$, so that on $\partial B(t)$ the normal vector is pointing \emph{into} $B(t)$.
Although the pressure, $p$, is only unique up to a constant this does not change how we understand the boundary integrals in \eqref{eqn: FSI rigid body} since the divergence theorem yields
\[ \int_{\partial B(t)} \boldsymbol{\nu} = 0 = \int_{\partial B(t)} (\mbf{x} - \mbf{q}(t)) \times \boldsymbol{\nu}. \]
Thus, without loss of generality, we shall consider $p$ such that $\int_{\Omega(t)} p = 0$ at all times.
Here \eqref{eqn: FSI fluid} are the usual incompressible Navier--Stokes equations, and \eqref{eqn: FSI rigid body} are the equations of motion for a rigid body (in three dimensions).

To avoid confusing notation we define the boundary of $\Omega(t)$ as a union of disjoint sets $\partial \Omega(t) = \Gamma \cup \partial B(t)$, where $\Gamma = \partial \Omega$ is the (stationary) outer portion of the boundary and $\partial B(t)$ is the boundary of the rigid body.
\begin{figure}[ht]
    \centering
    \scalebox{1}{
    \begin{tikzpicture}
		\draw [blue!30, thick, -]  (10,1) -- (2.5,1);
		\filldraw[color=black, fill = blue!30, thick] (2.5,1) to[out=85, in= 75] (10,1);
		\filldraw[color=black, fill = blue!30, thick] (10,1) to[out=-110,in=-95] (2.5,1);
		\filldraw[color=black, fill = white, thick](7,2) circle (0.6);
		\node at (7,2) {$B(t)$};
		\node[scale=1.25] at (5,-0.7) {$\Omega(t)$};
        \node[scale =0.75] at (6.7,-1.3) {$\mbf{u} = 0$}; 
		\draw [black, thick, ->]  (6.65,1.5) -- (6.32,1);
		\node[scale = 0.75] at (6.2,0.8) {$\mbf{u} = \boldsymbol{q}'(t) + \boldsymbol{\omega}(t) \times (\boldsymbol{x} - \boldsymbol{q}(t))$};
		\node[scale=1.25] at (9.7,-0.1) {$\Gamma$};
	\end{tikzpicture}
    }
    \caption{A diagram showing a fluid-rigid body interaction problem.}
    \label{fig: fsi diagram}
\end{figure}
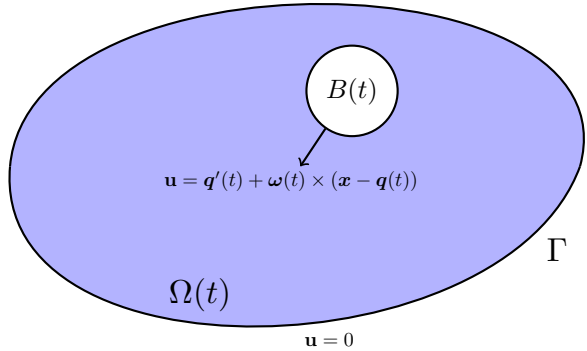
The system \eqref{eqn: FSI fluid}, \eqref{eqn: FSI rigid body} is complemented with no-slip boundary conditions
\begin{subequations}\label{eqn: FSI BCs}
    \begin{gather}
	\mbf{u} = 0, \text{ on } \Gamma,\label{eqn: FSI BC1}\\
	\mbf{u} = \mbf{q}'(t) + \boldsymbol{\omega}(t) \times (\mbf{x} - \mbf{q}(t)), \text{ on } \partial B(t), \label{eqn: FSI BC2}
\end{gather}
\end{subequations}
and initial conditions for $\mbf{u}$, $\mbf{q}$, and $\boldsymbol{\omega}$
\begin{align}
	\mbf{u}(0) = \mbf{u}_0, \quad \mbf{q}(0) = \mbf{q}_0, \quad \frac{\mathrm{d}\mbf{q}}{\mathrm{d}t}(0) = \mbf{v}_0, \quad \boldsymbol{\omega}(0) = \boldsymbol{\omega}_0, \label{FSI IC}
\end{align}
where $\Div \mbf{u}_0 = 0$.
We shall assume throughout that the initial data are compatible in the sense that
\[ \mbf{u}_0 |_{\partial B(0)} = \mbf{v}_0 + \boldsymbol{\omega}_0 \times (\mbf{x} - \mbf{q}_0), \text{ on } \partial B(0). \]
We shall consider $\mbf{u}_0 \in \mbf{H}_0^1(\Omega)$ (where we are using boldface letters to denote vector-valued Sobolev spaces) so that this equality is understood in the trace sense.

We assume throughout that $\Gamma$ and $\partial B(0)$ are both $C^{2,1}$ surfaces, so that $\Omega(0)$ is a $C^{2,1}$ domain.
We also assume that the rigid body and the outer boundary, $\Gamma$, are not initially in contact --- that is to say there exists some $\delta > 0$ such that
\[ \mathrm{dist}(\Gamma, B(0)) > \delta > 0. \]

This system models the motion of a rigid body immersed in a viscous, incompressible fluid, with no slip on the boundary of the rigid body.
The no-slip boundary conditions \eqref{eqn: FSI BCs} are  typically considered in the literature, but one could alternatively consider Navier-slip boundary conditions on the boundary of the rigid body, that is we replace \eqref{eqn: FSI BC2} with
\begin{subequations}
\label{eqn: slip BCs}
\begin{gather}
	\mbf{u} \cdot \boldsymbol{\nu} = \left(\mbf{q}'(t) + \boldsymbol{\omega}(t) \times (\mbf{x} - \mbf{q}_0)\right) \cdot \boldsymbol{\nu}, \label{eqn: slip BC1}\\
	\left(\mu \left( \grad \mbf{u} + (\grad \mbf{u})^T \right) \boldsymbol{\nu}\right) \times \boldsymbol{\nu} = -\beta \left( \mbf{u} - \mbf{q}'(t) + \boldsymbol{\omega}(t) \times (\mbf{x} - \mbf{q}_0) \right) \times \boldsymbol{\nu}, \label{eqn: slip BC2}
\end{gather}
\end{subequations}
on $\partial B(t)$ for some slip length $\beta > 0$.
With these boundary conditions the system has been studied in \cite{al2018strong,chemetov2017motion,gerard2014existence}.

Fluid-structure interaction problems naturally find application in the biomedical sciences, namely to the study of blood flow around the heart \cite{astorino2009fluid,vcanic2014fluid,van2006fluid} and the design of stents \cite{vcanic2021next}.
More generally, free/moving boundary problems are often motivated by applications in applied sciences and industry, as discussed in \cite{EllOckBook82}.
For further details on fluid-structure interactions and their applications we refer the reader to \cite{vcanic2021moving,galdi2002motion} and the references therein.

Before stating our main result we firstly we define what our notion of a strong solution of \eqref{eqn: FSI fluid}, \eqref{eqn: FSI rigid body} entails.

\begin{definition}
    \label{defn: strong soln fsi}
    A strong solution to \eqref{eqn: FSI fluid}, \eqref{eqn: FSI rigid body} is a quadruple of functions $(\mbf{q}, \boldsymbol{\omega}, \mbf{u}, p)$, defined on a time interval $[0,T]$, such that the following hold:
    \begin{enumerate}
        \item $\mbf{q} \in \mbf{H}^2([0,T])$, and $\boldsymbol{\omega} \in \mbf{H}^1([0,T])$ define a time-dependent domain $\Omega(t)$, and a parametrisation $\Phi(t): \Omega(0) \rightarrow \Omega(t)$.
        The rigid body evolves via $B(t) = \Phi(B(0),t)$ and is such that 
        \[\mathrm{dist}(\Gamma, B(t)) > \frac{\delta}{2} > 0,\]
        for all $t \in [0,T]$.
        \item The pairs $(\mbf{W}^{k,q}(\Omega(t)),\Phi(t))_{t \in [0,T]}$ are compatible, as defined in \S\ref{subsection: function spaces}, for all $k = 0,1,2$ and $q \in [0,\infty]$.
        \item The fluid quantities are such that $\mbf{u} \in H^1_{\mbf{L}^2}(\Phi)\cap L^2_{\mbf{H}^2}(\Phi)$, $p \in L^2_{H^1 \cap L^2_0}(\Phi)$.
        Moreover $(\mbf{u}, p)$ solve \eqref{eqn: FSI fluid} pointwise almost everywhere on $\bigcup_{t \in [0,T]} \Omega(t) \times \{t\}$, along with the boundary conditions \eqref{eqn: FSI BCs}, and initial condition $\mbf{u}(0) = \mbf{u}_0$ on $\Omega(0)$.
        \item $(\mbf{q}, \boldsymbol{\omega})$ solve \eqref{eqn: FSI rigid body} almost everywhere on $[0,T]$, and satisfy the initial conditions \eqref{FSI IC}.
    \end{enumerate}
    
\end{definition}

For a similar notation of a strong solution to this problem we refer the reader to \cite[Definition 2.1]{takahashi2003analysis}.
We defer further details on function spaces and notation to Section \ref{section: FSI prelim}.
With this definition of a strong solution, our main result is the following short time existence theorem.
\begin{theorem}
	\label{thm: FSI existence}
    Let the relative density, $\frac{\rho}{\rho_B}$, be sufficiently small\footnote{The smallness of $\frac{\rho}{\rho_B}$ is determined by the geometry of $\Omega(0)$ and $\mathrm{diam}(B(0))$ (see the proofs of Lemma \ref{lemma: non degeneracy2} and Lemma \ref{lemma: fsi contraction}).}.
    Then there exists some time $T$, depending only on the initial data and the geometry of $\Omega(0)$, such that there exists a unique quadruple
	\[(\mbf{q}, \boldsymbol{\omega}, \mbf{u}, p) \in \mbf{H}^2([0,T]) \times \mbf{H}^1([0,T]) \times L^2_{\mbf{H}^2}(\Phi)\cap H^1_{\mbf{L}^2}(\Phi) \times L^2_{H^1}(\Phi)\]
	of strong solutions solving \eqref{eqn: FSI fluid}, \eqref{eqn: FSI rigid body} in the sense of Definition \ref{defn: strong soln fsi}.
\end{theorem}
\begin{remark}
    Some basic numerical experiments (see Section~\ref{section: FSI numerics}) indicate that the smallness condition on the relative density, $\frac{\rho}{\rho_B}$, is necessary for our iterative approach to be convergent.
\end{remark}
The novelty of this present work is not in establishing the existence of strong solutions, as indeed this has been known since the early 2000s (cf.~\cite{grandmont2000existence, takahashi2003analysis}), but rather in the iterative method we use to prove this result.
We defer further discussion of our approach until Section~\ref{section:FSI iteration}.
In some sense our main result could equivalently be stated as a continuity-type result for solution maps related to the Navier--Stokes equations on domains with prescribed evolution, and one obtains Theorem~\ref{thm: FSI existence} as a corollary of this result.

Solutions to \eqref{eqn: FSI fluid}, \eqref{eqn: FSI rigid body} are often understood in a \emph{weak Leray--Hopf sense} --- see for instance the notions of weak solutions defined in \cite{al2018strong,vcanic2021moving,chemetov2017motion,chemetov2019weak}.
A solution in this sense is understood as a pair $(\mbf{u}, \mathscr{B})$, where $\mathscr{B}(\cdot,t)$ is an orientation preserving isometry defining the motion of the rigid body $B(t) := \mathscr{B}(t,B(0))$ which is compatible with the boundary conditions \eqref{eqn: FSI BCs}, and $\mbf{u}$ is a divergence-free velocity field solving \eqref{eqn: FSI fluid} in a Leray--Hopf sense.
We do not provide a proper definition, and instead refer to the above references, as the method we establish in this paper shall require strong solutions to the Navier--Stokes equations.
The existence of weak solutions has been previously established in \cite{desjardins1999existence} where the authors use DiPerna--Lions transport theory \cite{diperna1989ordinary} to show the existence of Eulerian quantities (defined on $\Omega$ rather than $\Omega(t)$).
Likewise strong solutions to the system \eqref{eqn: FSI fluid}, \eqref{eqn: FSI rigid body} have previously been established in \cite{geissert2013Lp,grandmont2000existence, takahashi2003analysis}.
We also note that an analogous system with a compressible fluid has also been studied in \cite{desjardins2000weak,feireisl2003motion,Hieber2015Compressible,Necasova2022Compressible}. 

Our approach to the existence of solutions is based on an iteration of solutions of the Navier--Stokes equations posed on a sequence of evolving domains with prescribed evolution.
The crux of this approach is quite intuitive: if we make a guess for the motion of the rigid body, we obtain a guess for the motion of the fluid, which in turn gives a new guess for the motion of the rigid body.
By iterating accordingly we hope to obtain a fixed point which, by construction, will solve the moving boundary problem.
For the sake of clarity we only consider a single rigid body, although our approach extends to multiple rigid bodies in a straightforward way, and likewise we do not consider a body force acting on the fluid.

The outline of this paper is as follows.
In Section \ref{section: FSI prelim} we discuss some preliminary material, such as function spaces and some useful inequalities.
Then in Section \ref{section: evolving NS} we recall known results on the Navier--Stokes equations posed on an evolving domain with \emph{prescribed evolution}, and prove energy estimates to be used throughout our analysis.
In Section \ref{section:FSI iteration} we introduce our iteration to construct solutions to \eqref{eqn: FSI fluid}, \eqref{eqn: FSI rigid body}, culminating in the proof of Theorem \ref{thm: FSI existence}.
Section \ref{section: FSI numerics} discusses a finite element implementation of this iterative method, and provides some numerical experiments using this numerical method.
Lastly we include some further details on a calculation used in our energy estimates in Appendix \ref{appendix: bihari use}, and some results on the Stokes problem on a prescribed evolving domain in Appendix \ref{appendix:stokes}.

\section{Preliminaries}
\label{section: FSI prelim}
\subsection{Function spaces}
\label{subsection: function spaces}
As our approach is based on a sequence of solutions to the Navier--Stokes equations posed on a sequence of evolving domains it is necessary for us to discuss the appropriate function spaces in which solutions shall live.
As such we shall use the theory of evolving function spaces, initiated by Alphonse, Elliott \& Stinner in~\cite{alphonse2015abstract,Alphonse2015SomeLinear}, which has been useful in the analysis of PDEs on evolving domains (see for instance~\cite{alphonse2023function}).
To this end we now briefly recall some of the theory of \cite{alphonse2023function,alphonse2015abstract}.
\begin{definition}\label{defn: compatible family}
Let $(X(t))_{t \in [0,T]}$ be a family of real Banach spaces indexed over $[0,T]$, such that for each $t \in [0,T]$ there exists a bounded, invertible, linear map
\[  \Phi_t :X(0) \rightarrow X(t),\]
with inverse
\[ \Phi_{-t}: X(t) \rightarrow X(0). \]
We say that the indexed pair $(X(t),\Phi_t)_{t \in [0,T]}$ is compatible if:
\begin{enumerate}
	\item $\Phi_0$ is the identity map on $X(0)$.
	\item There exists a constant, $C$, independent of $t \in [0,T]$ such that
	\begin{alignat*}{3}
		\| \Phi_t u\|_{X(t)} \leq C \|u\|_{X(0)},&& \quad &&\forall u \in X(0),\\
		\| \Phi_{-t} v\|_{X(0)} \leq C \|v\|_{X(t)},&& \quad &&\forall v \in X(t).
	\end{alignat*}
	\item For all $u \in X(0)$ the map $t \mapsto \|\Phi_t u\|_{X(t)}$ is measurable.
\end{enumerate}
\end{definition}

For a compatible family of time-dependent Banach spaces,~$(X(t),\Phi_t)_{t \in [0,T]}$, one defines time-dependent Bochner spaces as follows.

\begin{definition}
    \label{defn: evolving Bochner spaces}
	Let~$(X(t),\Phi_t)_{t \in [0,T]}$ be a compatible family of time-dependent Banach spaces.
	The space $L^p_X$ consists of (equivalence classes of) functions
	\begin{gather*}
		u:[0,T] \rightarrow \bigcup_{t \in [0,T]} X(t) \times \{ t \},\\
		t \mapsto (\overline{u}(t), t),
	\end{gather*}
	such that $ \Phi_{-(\cdot)} \overline{u}(\cdot) \in L^p(0,T;X(0))$.
	We identify $u$ with $\overline{u}$ throughout.
	This space is equipped with norm
	\[\|u\|_{L^p_X} := \begin{cases}
		\left( \int_0^T \|u(t)\|^p_{X(t)} \right)^{\frac{1}{p}}, & p \in [1,\infty),\\
		\esssup{t \in [0,T]}\|u(t)\|_{X(t)}, & p = \infty.
	\end{cases}\]
\end{definition}
It is known (see \cite{alphonse2023function,alphonse2015abstract}) that these are Banach spaces, and for $p=2$ and $X(t)$ a family of Hilbert spaces these are Hilbert spaces with inner product
\[(u,v)_{L^2_X} = \int_0^T (u(t),v(t))_{X(t)},\]
for $u,v \in L^2_X$.
In this case, as justified in \cite{alphonse2015abstract}, we make the identification of dual spaces $(L^2_{X})' \cong L^2_{X'}$.
To distinguish between two function spaces parametrised over the same initial space we shall include the parametrisation as an argument, i.e. $L^p_X(\Phi)$, which will help distinguish which evolving domain (and parametrisation) we are considering.

\subsection{Some useful results}

We now recall some results to be used throughout.
We begin with a crucial observation --- since the geometry of the boundary does not change (as we consider rigid body motion) one may obtain constants from elliptic regularity theory/Sobolev embeddings which are independent of $t \in [0,T]$.
This follows by observing that estimates obtained by ``straightening the boundary'', as in \cite[\S 6.2]{gilbarg1977elliptic}, are unchanged by translations/rotations of the rigid body.
Likewise the constants appearing in Sobolev embeddings are unchanged (see for instance \cite[\S 7.7]{gilbarg1977elliptic}).
This observation will be hidden throughout our analysis, and indeed any extension of our method to moving boundary problems with boundary deformation would require more careful usage of tools such as elliptic regularity theory.

Firstly we define the following divergence-free subspace of $\mbf{L}^2(\Omega(t))$,
\[ \mbf{H}_\sigma(\Omega(t)) := \left\{ \boldsymbol{\phi} \in \mbf{C}^\infty(\Omega(t)) \mid \Div \boldsymbol{\phi} = 0 \right\}^{\|\cdot\|_{\mbf{L}^2(\Omega(t))}}.  \]
Here the superscript denotes the closure in the $\mbf{L}^2(\Omega(t))$ norm.
We now state some well known results.

\begin{lemma}[Helmholtz decomposition]
	Let $\Omega(t) \subset \mbb{R}^3$ be a bounded, evolving domain with Lipschitz boundary.
	Then for all $\mbf{u} \in \mbf{L}^2(\Omega(t))$ there exist unique $\mbf{v} \in \mbf{H}_\sigma(\Omega(t)), \phi \in H_0^1(\Omega(t))$ such that
	\[ \mbf{u} = \mbf{v} + \grad \phi.\]
\end{lemma}
\begin{proof}
    See \cite[Theorem 2.16]{robinson2016three}.
\end{proof}

\begin{definition}
	For $\mbf{u} \in \mbf{L}^2(\Omega(t))$ we define the Leray projection of $\mbf{u}$ to be $P_\sigma \mbf{u} \in \mbf{H}_\sigma(\Omega(t))$
	\[ P_\sigma \mbf{u} := \mbf{v}, \]
	for $\mbf{v}$ as in the Helmholtz decomposition.
\end{definition}
\begin{lemma}  
    For $\mbf{u} \in \mbf{L}^2(\Omega(t))$ one has that
    \begin{align}
        \|P_\sigma \mbf{u}\|_{\mbf{L}^2(\Omega(t))} \leq \|\mbf{u}\|_{\mbf{L}^2(\Omega(t))}, \label{eqn: leray L2 bound}
    \end{align}
    and if one has further regularity, $\mbf{u} \in \mbf{H}^1(\Omega(t))$, then
    \begin{align}
        \|P_\sigma \mbf{u}\|_{\mbf{H}^1(\Omega(t))} \leq C\|\mbf{u}\|_{\mbf{H}^1(\Omega(t))}, \label{eqn: leray H1 bound}
    \end{align}
    for constants independent of $t \in [0,T]$.
\end{lemma}
\begin{proof}
    The first result follows from the fact that
    \[ \int_{\Omega(t)} \mbf{u} \cdot P_\sigma \mbf{u} = \int_{\Omega(t)} |P_\sigma \mbf{u}|^2. \]
    The second result follows from standard elliptic regularity theory, and our above observation that the constants obtained depend only on the geometry of $\Omega(0)$ due to the rigid body motion.
\end{proof}
The Leray projection is sometimes also stated as a pseudodifferential operator,
\[ P_\sigma \mbf{u} = \mbf{u} - \grad \Delta^{-1}(\Div \mbf{u}), \]
for an appropriate notion of the inverse Laplacian, $\Delta^{-1}$.
We refer the reader to \cite{robinson2016three,temam2001navier} for further details.
We note that for an evolving domain the Leray projection now varies in time.
\begin{lemma}[Agmon's inequality]
	Let $\Omega(t) \subset \mbb{R}^3$ be a bounded, evolving domain with Lipschitz boundary.
	Then for $\mbf{u} \in \mbf{H}^2(\Omega(t))$ one has
	\begin{align}
		\|\mbf{u}\|_{\mbf{L}^\infty(\Omega(t))} \leq C \|\mbf{u}\|_{\mbf{H}^1(\Omega(t))}^{\frac{1}{2}}\|\mbf{u}\|_{\mbf{H}^2(\Omega(t))}^{\frac{1}{2}},\label{agmon inequality}
	\end{align}
	for a constant, $C$, independent of $t \in [0,T]$.
\end{lemma}

\begin{proof}
	We refer the reader for to \cite{agmon2010lectures} for the proof of this on a stationary domain --- moreover the constant one obtains in this inequality depends only on the geometry of $\Omega(t)$ which is does not change as the boundaries do not deform.
\end{proof}

\begin{lemma}[Ladyzhenskaya's inequality]
	Let $\Omega(t) \subset \mbb{R}^3$ be a bounded, evolving domain with Lipschitz boundary.
	Then for $\mbf{u} \in \mbf{H}_0^1(\Omega(t))$ one has
	\begin{align}
		\|\mbf{u}\|_{\mbf{L}^4(\Omega(t))} \leq C \|\mbf{u}\|_{\mbf{H}^1(\Omega(t))}^{\frac{3}{4}}\|\mbf{u}\|_{\mbf{L}^2(\Omega(t))}^{\frac{1}{4}}, \label{ladyzhenskaya inequality}
	\end{align}
	for a constant, $C$, independent of $t \in [0,T]$.
\end{lemma}
\begin{proof}
    We refer the reader to \cite[Lemma 3.3.5]{temam2001navier} which follows generally from the Gagliardo--Nirenberg inequality.
    Moreover this proof shows the constant is independent of choice of domain.
\end{proof}

\begin{lemma}[{\cite[Theorem 2.23]{robinson2016three}}]
	Let $\Omega(t) \subset \mbb{R}^3$ be a $C^{2}$ bounded, evolving domain.
	For $\mbf{u} \in \mbf{H}_\sigma(\Omega(t))\cap \mbf{H}_0^1(\Omega(t))\cap\mbf{H}^2(\Omega(t))$ one has that
	\begin{align}
		\| \mbf{u} \|_{\mbf{H}^2(\Omega(t))} \leq C \| P_\sigma \Delta \mbf{u} \|_{\mbf{L}^2(\Omega(t))}, \label{leray laplacian inequality}
	\end{align}
	where $P_\sigma$ is the Leray projection as defined above, and $C$ is a constant independent of $t$.
\end{lemma}
\begin{proof}
	We refer the reader to \cite[Theorem 2.23]{robinson2016three}, and note that as discussed in the proof the constant appears from regularity results for the Stokes problem, which in turn depend on the geometry of $\Omega(0)$.
\end{proof}

We also recall a nonlinear generalisation of the Gr\"onwall inequality, which will be used to establish energy estimates.

\begin{lemma}[Bihari--LaSalle inequality, \cite{bihari1956generalization}]
	\label{biharilasalle}
	Let $X, K : [0,T] \rightarrow \mbb{R}$, be non-negative continuous
	functions, $\lambda: \mbb{R}^+ \rightarrow \mbb{R}^+$ be a non-decreasing continuous function, and $k \geq 0$.
	Then if
	\[X(t) \leq k + \int_0^t K(s) \lambda(X(s)) \, \mathrm{d}s ,\]
	holds for $t \in [0,T]$, and one can choose $y_0 > 0$ such that
	\[\Lambda(k) + \int_0^T K(s) \, \mathrm{d}s \in \mathrm{dom}(\Lambda^{-1}), \quad \text{where} \quad \Lambda(y) := \int_{y_0}^y \frac{1}{\lambda(s)} \, \mathrm{d}s,\]
	then for $t \in [0,T]$
	\begin{align}
		X(t) \leq \Lambda^{-1}\left( \Lambda(k) + \int_0^t K(s) \, \mathrm{d}s \right). \label{biharilasalle2}
	\end{align}
	
\end{lemma}

Lastly we recall the Reynolds transport theorem, and prove a version involving the gradients of functions on a time-dependent domain.
\begin{lemma}[Reynolds Transport Theorem]
\label{lemma: reynolds transport theorem}
Let $\mbf{u}, \mbf{w}$ be two sufficiently smooth functions defined on an evolving domain $\Omega(t)$ moving with velocity $\mbf{V}$.
Then
\begin{gather}
    \begin{split}\frac{\mathrm{d}}{\mathrm{d}t} \int_{\Omega(t)} \mbf{u} \cdot \mbf{w} &=\int_{\Omega(t)} \matdev_\mbf{V}{\mbf{u}} \cdot \mbf{w} + \int_{\Omega(t)} \mbf{u} \cdot \matdev_\mbf{V}{\mbf{w}} + \int_{\Omega(t)} (\mbf{u} \cdot \mbf{w}) (\Div \mbf{V})\\
    &= \int_{\Omega(t)} \ddt{\mbf{u}} \cdot \mbf{w} + \int_{\Omega(t)} \mbf{u} \cdot \ddt{\mbf{w}} + \int_{\partial\Omega(t)} (\mbf{u} \cdot \mbf{w}) (\mbf{V} \cdot \boldsymbol{\nu}),\\
    \end{split}\label{eqn: reynolds1}\\
    \frac{\mathrm{d}}{\mathrm{d}t} \int_{\Omega(t)} \grad\mbf{u} : \grad\mbf{w} = \int_{\Omega(t)} \grad\matdev_{\mbf{V}}{\mbf{u}} : \grad \mbf{w} + \int_{\Omega(t)} \grad \mbf{u} : \grad \matdev_{\mbf{V}}\mbf{w} + \int_{\Omega(t)} \mbb{B}(\mbf{V})\grad\mbf{u} : \grad\mbf{w}, \label{eqn: reynolds2}
\end{gather}
where we have used $\matdev_{\mbf{V}}$ as shorthand notation for the material derivative following the flow of $\mbf{V}$,
\[ \matdev_{\mbf{V}} \mbf{u} := \ddt{\mbf{u}} + \mbf{V} \cdot (\grad \mbf{u}), \]
where~$\mbb{B}(\mbf{V})$ denotes the deformation tensor
\[ \mbb{B}(\mbf{V}) := (\Div \mbf{V}) \mbb{I} - (\grad \mbf{V} + (\grad \mbf{V})^T). \]
\end{lemma}
\begin{proof}
    We only prove the second equality as the first is variation of the usual Reynolds transport theorem (see for instance \cite{gurtin1982introduction}), and indeed this proof is essentially the same as that of~\cite[Lemma 5.2]{Dziuk2013SurfaceFEM}.
    For the second equality we note that by pulling back onto $\Omega(0)$ we find that
    \[ \int_{\Omega(t)} \grad\mbf{u} : \grad\mbf{w} = \int_{\Omega(0)}J (\grad \Phi)^{-T}\grad\widehat{\mbf{u}} : (\grad \Phi)^{-T}\grad\widehat{\mbf{w}}, \]
    where $\Phi(t) : \Omega(0) \rightarrow \Omega(t)$ is the flow generated by $\mbf{V}$, and $J = \det(\grad \Phi)$, and hats denote the pullback by $\Phi$, i.e. $\widehat{\mbf{u}} = \mbf{u}\circ \Phi(t)$.
    This is discussed later in Lemma \ref{pullback lemma}.
    Differentiating the above in time yields
    \begin{align*}
        \frac{\mathrm{d}}{\mathrm{d}t}\int_{\Omega(t)} \grad\mbf{u} : \grad\mbf{w} &= \int_{\Omega(0)}\ddt{J} (\grad \Phi)^{-T}\grad\widehat{\mbf{u}} : (\grad \Phi)^{-T}\grad\widehat{\mbf{w}} + \int_{\Omega(0)}J \ddt{}(\grad \Phi)^{-T}\grad\widehat{\mbf{u}} : (\grad \Phi)^{-T}\grad\widehat{\mbf{w}}\\
        &+ \int_{\Omega(0)}J (\grad \Phi)^{-T}\grad\ddt{\widehat{\mbf{u}}} : (\grad \Phi)^{-T}\grad\widehat{\mbf{w}} + \int_{\Omega(0)}J (\grad \Phi)^{-T}\grad\widehat{\mbf{u}} : \ddt{}(\grad \Phi)^{-T}\grad\widehat{\mbf{w}}\\
        &+ \int_{\Omega(0)}J (\grad \Phi)^{-T}\grad\widehat{\mbf{u}} : (\grad \Phi)^{-T}\grad\ddt{\widehat{\mbf{w}}}.
    \end{align*}
    We then recall that 
    \begin{alignat*}{3}
        \ddt{J} = J \widehat{\Div \mbf{V}} && \quad \text{and} \quad && \ddt{} (\grad \Phi)^{-T} = - (\grad \Phi)^{-T}(\grad \widehat{\mbf{V}})^T (\grad \Phi)^{-T},
    \end{alignat*}
    from which one finds that
    \begin{align*}
        \frac{\mathrm{d}}{\mathrm{d}t}\int_{\Omega(t)} \grad\mbf{u} : \grad\mbf{w} &= \int_{\Omega(t)} \grad\matdev_{\mbf{V}}{\mbf{u}} : \grad \mbf{w} + \int_{\Omega(t)} \grad \mbf{u} : \grad \matdev_{\mbf{V}}\mbf{w} + \int_{\Omega(t)} \grad\mbf{u} : \grad\mbf{w} (\Div \mbf{V})\\
        &- \int_{\Omega(t)} (\grad \mbf{V})^T \grad\mbf{u} : \grad\mbf{w} - \int_{\Omega(t)} \grad\mbf{u} : (\grad \mbf{V})^T\grad\mbf{w},
    \end{align*}
    which one can readily rewrite as in \eqref{eqn: reynolds2}.
\end{proof}

\section{Navier--Stokes equations on a domain with prescribed evolution}
\label{section: evolving NS}
\subsection{Statement of the problem}
In this section we recall some results on the Navier--Stokes equations posed on a bounded evolving domain, $\Omega(t) \subset \mbb{R}^3$, with prescribed evolution.
Motivated by the moving boundary problem we consider spatial domains of the form $\Omega(t) = (\Omega \setminus B(t))^\circ$, where $B(t)$ is a rigid body moving with a prescribed velocity (i.e.~$\mbf{q}$ and~$\boldsymbol{\omega}$ are known functions).
We are interested in \emph{strong solutions} of the problem
\begin{align}
	\begin{split}
		\rho \left( \ddt{\mbf{u}} + (\mbf{u} \cdot \grad) \mbf{u} \right) &= - \grad p + \mu \Delta \mbf{u},\\
		\Div \mbf{u} &= 0,
	\end{split}
	\label{eqn: evolving NS}
\end{align}
on the non-cylindrical space-time domain $\bigcup_{t \in [0,T]} \Omega(t) \times \{ t \}$, with inhomogeneous boundary conditions
\begin{gather*}
	\mbf{u} = 0, \text{ on } \Gamma,\\
	\mbf{u} = \mbf{q}'(t) + \boldsymbol{\omega}(t) \times (\mbf{x} - \mbf{q}(t)), \text{ on } \partial B(t), 
\end{gather*}
and such that $\mbf{u}(0) = \mbf{u}_0$.
Here we crucially assume that the functions $\mbf{q}, \boldsymbol{\omega}$ are \emph{known a priori}, sufficiently regular (for our purposes $\mbf{q} \in \mbf{H}^2([0,T])$ and $\boldsymbol{\omega} \in \mbf{H}^1([0,T])$), and define the motion of $B(t)$.
In particular, this means that the space-time domain $\bigcup_{t \in [0,T]} \Omega(t) \times \{ t \}$ is prescribed, and evolves with a velocity field $\mbf{V}$ (which is an appropriate extension of the above boundary data for $\mbf{u}$).
This problem, and similar problems, have been studied previous in \cite{Bock77NavierStokes,djurdjevac2023evolving,farwig2022fujita,fujita1970existence,ladyzhenskaya1970initial,saal2006maximal,salvi1988navier} and so our presentation is largely formal.
We note that \cite{djurdjevac2023evolving} uses the same evolving function space framework as us, although it considers the Oseen equations rather than the Navier--Stokes equations.
We assume compatibility, in the sense of Definition~\ref{defn: compatible family}, of the pairs $(\mbf{W}^{k,q}(\Omega(t)), \Phi_t)$ for $k=0,1,2$ and $q \in [1,\infty]$ throughout this section.

We note that the boundary conditions are compatible with the divergence-free condition.
This follows by noting that
\[ \grad \left( \mbf{q}'(t) + \boldsymbol{\omega}(t) \times (\mbf{x} - \mbf{q}(t)) \right) = \begin{pmatrix}
	0 & -\omega_3 & \omega_2\\
	\omega_3 & 0 & -\omega_1\\
	-\omega_2 & \omega_1 & 0
\end{pmatrix}, \]
where $\omega_i$ are the components of $\boldsymbol{\omega}$, and so $\Div \left( \mbf{q}'(t) + \boldsymbol{\omega}(t) \times (\mbf{x} - \mbf{q}(t))\right) = 0$. 

\subsubsection{Transformation to homogeneous boundary conditions}
As this problem has inhomogeneous boundary data is it natural to study a transformed problem with homogeneous boundary data.
For each $t \in [0,T]$ we define $\widetilde{\mbf{u}}(t)$ to be the unique solution to the steady state Stokes problem on $\Omega(t)$:
\begin{equation}
    \begin{aligned}
        -\mu \Delta \widetilde{\mbf{u}} + \grad \widetilde{p} &= 0\\
    \Div \widetilde{\mbf{u}} &= 0
    \end{aligned} \quad \text{on } \Omega(t),
    \label{eqn: stokes eqn}
\end{equation}
such that
\[ \widetilde{\mbf{u}} = 0 \text{ on } \Gamma, \quad \widetilde{\mbf{u}} = \mbf{q}'(t) + \boldsymbol{\omega}(t)\times(\mbf{x} - \mbf{q}(t)) \text{ on } \partial B(t). \]
Existence of a weak solution follows from standard results on the Stokes equation, and we have a regularity result \eqref{stokes regularity1}.
We will also require that the time derivative of $\widetilde{\mbf{u}}$ is sufficiently smooth, namely an element of $L^2_{\mbf{H}^1}(\Phi)$.
Further discussion of the Stokes equation on an evolving domain is discussed in Appendix \ref{appendix:stokes}.

Given these $\widetilde{\mbf{u}}$ and $\widetilde{p}$ we now define $\overline{\mbf{u}} = \mbf{u} - \widetilde{\mbf{u}}$ and $\overline{p} = p - \widetilde{p}$ which formally solve
\begin{gather}
	\begin{gathered}
		\rho \left( \ddt{\overline{\mbf{u}}} + (\overline{\mbf{u}} \cdot \grad)\overline{\mbf{u}}\right) = - \grad \overline{p}+ \mu \Delta \overline{\mbf{u}} - \rho \left(\ddt{\widetilde{\mbf{u}}} + (\widetilde{\mbf{u}} \cdot \grad) \overline{\mbf{u}} + (\overline{\mbf{u}} \cdot \grad)\widetilde{\mbf{u}} + (\widetilde{\mbf{u}} \cdot \grad) \widetilde{\mbf{u}}\right),\\
		\Div \overline{\mbf{u}} = 0,
	\end{gathered}
	\label{eqn: evolving NS2}
\end{gather}
such that $\overline{\mbf{u}} = 0$ on $\partial \Omega(t)$.
Motivated by this, we study the weak formulation of the Navier--Stokes equation: find $\overline{\mbf{u}} \in L^2_{\mbf{H}^2}(\Phi) \cap H^1_{\mbf{H}_0^{1}}(\Phi)$ and $\overline{p} \in L^2_{H^1}(\Phi)$ solving
\begin{gather}
	\begin{aligned}
    &\rho \int_{\Omega(t)}\ddt{\overline{\mbf{u}}} \cdot  \boldsymbol{\eta} + \rho \int_{\Omega(t)} \left(\overline{\mbf{u}}\cdot \grad \overline{\mbf{u}}\right) \boldsymbol{\eta} + \rho \int_{\Omega(t)} \left((\widetilde{\mbf{u}} \cdot \grad \overline{\mbf{u}}) + (\overline{\mbf{u}} \cdot \grad\widetilde{\mbf{u}}) \right)\boldsymbol{\eta} + \mu \int_{\Omega(t)} \grad\overline{\mbf{u}} : \grad \boldsymbol{\eta}\\
    &- \int_{\Omega(t)} \overline{p} \Div \boldsymbol{\eta} = -\rho \int_{\Omega(t)} \left(\ddt{\widetilde{\mbf{u}}} \cdot \boldsymbol{\eta} + (\widetilde{\mbf{u}} \cdot \grad \widetilde{\mbf{u}}) \boldsymbol{\eta}\right),
    \end{aligned}\label{eqn: weakNS}\\
    \int_{\Omega(t)} q \Div \overline{\mbf{u}} = 0, \label{eqn: weakNS2}
\end{gather}
for all $\boldsymbol{\eta} \in \mbf{H}_0^1(\Omega(t)), q \in L^2(\Omega(t))$, and almost all $t \in [0,T]$, such that $\mbf{u}(0) = \mbf{u}_0$.
We note that here we do not concern ourselves with Leray--Hopf weak type solutions, since our approach will require strong solutions.

In the interest of brevity we introduce shorthand notation,
\begin{equation}
    \mbf{B}:= \rho \left(\frac{\partial \widetilde{\mbf{u}}}{\partial t} + (\grad \widetilde{\mbf{u}})^T \widetilde{\mbf{u}}\right), \label{eqn: body force defn}
\end{equation}
which can be interpreted as a body force due to the motion of the rigid body --- indeed if the rigid body were stationary then term vanishes.
Notice that
\begin{align}
    \|\mbf{B}\|_{\mbf{H}^1(\Omega(t))} &\leq \rho\left\|\frac{\partial \widetilde{\mbf{u}}}{\partial t}\right\|_{\mbf{H}^1(\Omega(t))} + \rho\|(\grad \widetilde{\mbf{u}})^T \widetilde{\mbf{u}}\|_{\mbf{H}^1(\Omega(t))},\notag \\
    &\leq C \left( \left\|\frac{\partial \widetilde{\mbf{u}}}{\partial t}\right\|_{\mbf{H}^1(\Omega(t))} + \|\widetilde{\mbf{u}}\|_{\mbf{W}^{1,4}(\Omega(t))} \|\widetilde{\mbf{u}}\|_{\mbf{W}^{2,4}(\Omega(t))}\right), \label{body force bound1}
\end{align}
and hence using the results from Appendix~\ref{appendix:stokes} we see that $\mbf{B} \in L^2_{\mbf{H}^1}(\Phi)$.

\subsection{Energy estimates}
\label{subsection: energy estimates}
We present some formal calculations establishing energy estimates for a solution of \eqref{eqn: weakNS}, \eqref{eqn: weakNS2}.
Although these calculations are formal, they can be rigorously justified by means of Faedo--Galerkin approximation.

\begin{lemma}
\label{lemma: fsi energy estimates}
	A weak solution of \eqref{eqn: weakNS} satisfies 
	\begin{align}
		\frac{\rho}{2} \|\overline{\mbf{u}}\|_{\mbf{L}^2(\Omega(t))}^2 + \mu \int_0^t\|\grad\overline{\mbf{u}}\|_{\mbf{L}^2(\Omega(s))}^2 \leq e^{\frac{T}{\rho\kappa}}\left( \frac{\rho}{2}\|\mbf{u}_0\|_{\mbf{L}^2(\Omega(0))}^2 + \frac{\kappa}{2}\int_0^T \|\mbf{B}\|_{\mbf{L}^2(\Omega(t))}^2 \right), \label{fsi energyestimate1}
	\end{align}
	for almost all $t \in [0,T]$.
	Moreover, for $t \in [0,\widetilde{T})$ for some $\widetilde{T}=\widetilde{T}(\mbf{V},\mbf{u}_0)$, the solution satisfies
	\begin{align}
		\frac{\rho}{2} \sup_{t \in [0,\widetilde{T}]} \int_{\Omega(t)} |\grad \overline{\mbf{u}}|^2 + \frac{\mu}{2} \int_0^{\widetilde{T}}\int_{\Omega(t)} |A_{\sigma} \overline{\mbf{u}}|^2 \leq  \mathcal{D}(\widetilde{T}), \label{fsi energyestimate2}
	\end{align}
	for
    \begin{gather*}
        \mathcal{D}(T) := \left(\frac{ \left(\rho^2 \|\grad \mbf{u}_0\|_{\mbf{L}^2(\Omega(0))}^4 + \left(\kappa \int_0^{T} \|\mbf{B}\|_{\mbf{H}^1(\Omega(t))}^2\right)^2\right) \mathcal{E}(T)}{4 - \left(\rho^2 \|\grad \mbf{u}_0\|_{\mbf{L}^2(\Omega(0))}^4 + \left( \kappa\int_0^{T} \|\mbf{B}\|_{\mbf{H}^1(\Omega(t))}^2\right)^2\right)(\mathcal{E}(T)-1)} \right)^\frac{1}{2},\\
        \mathcal{E}(T) := \exp\left(C\int_0^{T}\left(\frac{1}{\kappa^2} + \|\Div\mbf{V}\|_{L^\infty(\Omega(t))} + \|\mbf{V}\|_{\mbf{L}^\infty(\Omega(t))}^2 +\|\widetilde{\mbf{u}}\|_{\mbf{L}^4(\Omega(t))}^8 +\|\widetilde{\mbf{u}}\|_{\mbf{W}^{1,4}(\Omega(t))}^2\right)\right).
    \end{gather*}
    Here $\kappa \in (0,1)$ denotes an arbitrary constant to be fixed later, and $C$ denotes a constant depending on $\Omega(0)$, $\rho$ and $\mu$, and $A_{\sigma} := -P_{\sigma} \Delta$.
\end{lemma}
\begin{proof}
	To show \eqref{fsi energyestimate1} we test \eqref{eqn: weakNS} with $\overline{\mbf{u}}$ and \eqref{eqn: weakNS2} with $\overline{p}$ so that
    \begin{multline*}
        \rho \int_{\Omega(t)} \ddt{\overline{\mbf{u}}} \cdot \overline{\mbf{u}} + \rho \int_{\Omega(t)} \left(\overline{\mbf{u}}\cdot \grad \overline{\mbf{u}}\right) \overline{\mbf{u}} + \mu \int_{\Omega(t)} |\grad \overline{\mbf{u}}|^2 + \rho \int_{\Omega(t)} \left((\widetilde{\mbf{u}} \cdot \grad) \overline{\mbf{u}} + (\overline{\mbf{u}} \cdot \grad)\widetilde{\mbf{u}} \right)\overline{\mbf{u}}\\
        + \int_{\Omega(t)} \mbf{B} \cdot \overline{\mbf{u}} = 0.
    \end{multline*}
	From Lemma~\ref{lemma: reynolds transport theorem} one finds that
	\[ \rho \int_{\Omega(t)} \ddt{\overline{\mbf{u}}} \cdot \overline{\mbf{u}} = \frac{\rho}{2} \frac{\mathrm{d}}{\mathrm{d}t} \int_{\Omega(t)} |\overline{\mbf{u}}|^2 - \frac{\rho}{2} \int_{\partial B(t)} |\overline{\mbf{u}}|^2 (\mbf{V} \cdot \boldsymbol{\nu}) = \frac{\rho}{2} \frac{\mathrm{d}}{\mathrm{d}t} \int_{\Omega(t)} |\overline{\mbf{u}}|^2, \]
    where we have used the fact that $\overline{\mbf{u}}$ vanishes at the boundary.
    Similarly using the divergence theorem one finds, in a standard calculation for the Navier--Stokes equations,
	\[ \rho \int_{\Omega(t)} \left(\overline{\mbf{u}} \cdot \grad \overline{\mbf{u}}\right) \overline{\mbf{u}} = 0,\]
	where we have used the homogeneous boundary condition on $\partial B(t)$.
	Hence one finds that
	\begin{align*}
		\frac{\rho}{2} \frac{\mathrm{d}}{\mathrm{d}t} \int_{\Omega(t)} |\overline{\mbf{u}}|^2 + \mu \int_{\Omega(t)} |\grad \overline{\mbf{u}}|^2 = - \rho \int_{\Omega(t)} \left((\widetilde{\mbf{u}} \cdot \grad \overline{\mbf{u}}) + (\overline{\mbf{u}} \cdot \grad\widetilde{\mbf{u}}) \right)\overline{\mbf{u}} -  \int_{\Omega(t)} \mbf{B} \cdot \overline{\mbf{u}}.
	\end{align*}
    For this final term we use Young's inequality for
    \[ \left| \int_{\Omega(t)} \mbf{B} \cdot \overline{\mbf{u}} \right| \leq \frac{\kappa}{2}\| \mbf{B}\|_{\mbf{L}^2(\Omega(t))}^2 + \frac{1}{2\kappa}\| \overline{\mbf{u}}\|_{\mbf{L}^2(\Omega(t))}^2,\]
    where $\kappa > 0$ is an arbitrary constant that we will fix later.
        We then calculate that, as $\widetilde{\mbf{u}}, \overline{\mbf{u}}$ are divergence-free
        \[ \int_{\Omega(t)} (\widetilde{\mbf{u}} \cdot \grad \overline{\mbf{u}})\overline{\mbf{u}} = - \int_{\Omega(t)} (\overline{\mbf{u}} \cdot \grad\widetilde{\mbf{u}})\overline{\mbf{u}} \]
	and so
        \[\frac{\rho}{2} \frac{\mathrm{d}}{\mathrm{d}t} \int_{\Omega(t)} |\overline{\mbf{u}}|^2 + \mu \int_{\Omega(t)} |\grad \overline{\mbf{u}}|^2 \leq \frac{\kappa}{2}\| \mbf{B}\|_{\mbf{L}^2(\Omega(t))}^2 + \frac{1}{2\kappa}\| \overline{\mbf{u}}\|_{\mbf{L}^2(\Omega(t))}^2,\]
        whence the Gr\"onwall inequality now yields \eqref{fsi energyestimate1}.
	
	To show \eqref{fsi energyestimate2} the argument is more involved.
	Our presentation is similar to that of \cite{salvi1988navier}.
	One tests \eqref{eqn: weakNS} with $A_{\sigma} \overline{\mbf{u}} := -P_\sigma \Delta \overline{\mbf{u}}$ to obtain
	\begin{align*}
		\rho \int_{\Omega(t)} \ddt{\overline{\mbf{u}}} \cdot A_{\sigma}\overline{\mbf{u}} + \left(\overline{\mbf{u}} \cdot \grad \overline{\mbf{u}}\right) A_{\sigma}\overline{\mbf{u}} + \mu \int_{\Omega(t)} \grad\overline{\mbf{u}}:\grad A_{\sigma} \overline{\mbf{u}}&= -\int_{\Omega(t)} \mbf{B} \cdot A_{\sigma} \overline{\mbf{u}}\\
        &-\rho \int_{\Omega(t)} \left((\widetilde{\mbf{u}} \cdot \grad \overline{\mbf{u}}) + (\overline{\mbf{u}} \cdot \grad\widetilde{\mbf{u}}) \right)A_{\sigma} \overline{\mbf{u}}.
	\end{align*}

        The difficulty now is that since the domain is evolving one finds that, in contrast to the stationary domain setting, $\ddt{\overline{\mbf{u}}} \neq P_\sigma \ddt{\overline{\mbf{u}}}$, and likewise the spatial and temporal derivatives do not commute --- this is in contrast to a stationary domain (see for instance the proof of \cite[Theorem 6.8]{robinson2016three}).
        However, we notice that from the Helmholtz decomposition
        \[ \int_{\Omega(t)} \left( \ddt{\overline{\mbf{u}}} - P_\sigma  \ddt{\overline{\mbf{u}}}\right) \cdot \Delta \overline{\mbf{u}} = \int_{\Omega(t)} \grad \psi \cdot \Delta \overline{\mbf{u}}, \]
        for some $\psi \in H_0^1(\Omega(t))$.
        Hence we find that, for sufficiently smooth $\bar{\mbf{u}}$,
        \[ \int_{\Omega(t)} \grad \psi \cdot \Delta \overline{\mbf{u}} = \int_{\Omega(t)}  \psi \Div \Delta \overline{\mbf{u}} = \int_{\Omega(t)} \psi \Delta \Div \overline{\mbf{u}} = 0,\]
        and this can be rigorously justified by an approximation argument.
        We also refer the reader to \cite[p.166]{salvi1988navier} for a related calculation.
        One thus obtains
        \[ \int_{\Omega(t)} \ddt{\overline{\mbf{u}}} \cdot A_{\sigma}\overline{\mbf{u}} = -\int_{\Omega(t)} P_\sigma \ddt{\overline{\mbf{u}}} \cdot \Delta \overline{\mbf{u}} = -\int_{\Omega(t)} \ddt{\overline{\mbf{u}}} \cdot \Delta \overline{\mbf{u}} = \int_{\Omega(t)} \grad\ddt{\overline{\mbf{u}}}: \grad \overline{\mbf{u}}. \]
	Now by using Lemma~\ref{lemma: reynolds transport theorem} we see that
	\begin{align*}
		\rho \int_{\Omega(t)} \grad\ddt{\overline{\mbf{u}}}: \grad \overline{\mbf{u}} = \frac{\rho}{2} \frac{\mathrm{d}}{\mathrm{d}t} \int_{\Omega(t)} |\grad \overline{\mbf{u}}|^2 - \rho \int_{\Omega(t)} \grad(\mbf{V} \cdot \grad \overline{\mbf{u}}): \grad \overline{\mbf{u}} - \frac{\rho}{2} \int_{\Omega(t)} \mbb{B}(\mbf{V}) \grad \overline{\mbf{u}} : \grad \overline{\mbf{u}}.
	\end{align*}
    Likewise from integration by parts, and the definition of the Leray projection we find
    \[ \mu \int_{\Omega(t)} \grad\overline{\mbf{u}}:\grad A_{\sigma} \overline{\mbf{u}} = -\mu \int_{\Omega(t)} \Delta\overline{\mbf{u}} \cdot A_{\sigma} \overline{\mbf{u}} = \mu \int_{\Omega(t)} |A_{\sigma} \overline{\mbf{u}}|^2. \]
	Hence we find that
	\begin{multline}
        \frac{\rho}{2} \frac{\mathrm{d}}{\mathrm{d}t} \int_{\Omega(t)} |\grad \overline{\mbf{u}}|^2 + \mu \int_{\Omega(t)} |A_{\sigma} \overline{\mbf{u}}|^2 = \rho \int_{\Omega(t)} \grad(\mbf{V} \cdot \grad \overline{\mbf{u}}): \grad \overline{\mbf{u}} + \frac{\rho}{2} \int_{\Omega(t)} \mbb{B}(\mbf{V}) \grad \overline{\mbf{u}} : \grad \overline{\mbf{u}}\\
        - \rho \int_{\Omega(t)} \left(\overline{\mbf{u}} \cdot \grad \overline{\mbf{u}}\right) A_{\sigma}\overline{\mbf{u}} - \int_{\Omega(t)} \mbf{B} \cdot A_{\sigma} \overline{\mbf{u}}
        -\rho \int_{\Omega(t)} \left((\widetilde{\mbf{u}} \cdot \grad \overline{\mbf{u}}) + (\overline{\mbf{u}} \cdot \grad\widetilde{\mbf{u}}) \right)A_{\sigma} \overline{\mbf{u}},
        \label{fsi energyestimate pf}
        \end{multline}
	where it remains to estimate the terms on the right-hand side.
	Firstly one finds that 
    \[ \left| \int_{\Omega(t)} \grad(\mbf{V} \cdot \grad \overline{\mbf{u}}): \grad \overline{\mbf{u}} \right| \leq \|\grad \overline{\mbf{u}}\|_{\mbf{L}^2(\Omega(t))}^2 \|\mbf{V}\|_{\mbf{W}^{1,\infty}(\Omega(t))} + \| \grad \overline{\mbf{u}}\|_{\mbf{L}^2(\Omega(t))} \|\overline{\mbf{u}}\|_{\mbf{H}^2(\Omega(t))}\|\mbf{V}\|_{\mbf{L}^\infty(\Omega(t))},\]
	and so by recalling \eqref{leray laplacian inequality} and using Young's inequality one finds
	\[ \left| \rho \int_{\Omega(t)} \grad(\mbf{V} \cdot \grad \overline{\mbf{u}}): \grad \overline{\mbf{u}} \right| \leq C \left(\|\mbf{V}\|_{\mbf{W}^{1,\infty}(\Omega(t))} + \|\mbf{V}\|_{\mbf{L}^\infty(\Omega(t))}^2 \right)\|\grad \overline{\mbf{u}}\|_{\mbf{L}^2(\Omega(t))}^2  + \frac{\mu}{10}\|A_{\sigma}\overline{\mbf{u}}\|_{\mbf{L}^2(\Omega(t))}^2.  \]
    We similarly have that
    \[ \left| \frac{\rho}{2} \int_{\Omega(t)} \mbb{B}(\mbf{V}) \grad \overline{\mbf{u}} : \grad \overline{\mbf{u}} \right| \leq C \|\grad \overline{\mbf{u}}\|_{\mbf{L}^2(\Omega(t))}^2 \|\mbf{V}\|_{\mbf{W}^{1,\infty}(\Omega(t))}. \]
	Next we use H\"older's inequality to see that
	\[ \left| \rho \int_{\Omega(t)} \left(\overline{\mbf{u}} \cdot \grad \overline{\mbf{u}}\right) A_{\sigma}\overline{\mbf{u}} \right| \leq \rho \|\overline{\mbf{u}}\|_{\mbf{L}^\infty(\Omega(t))} \|\grad \overline{\mbf{u}}\|_{\mbf{L}^2(\Omega(t))} \|A_{\sigma} \overline{\mbf{u}}\|_{\mbf{L}^2(\Omega(t))}, \]
	and then by Agmon's inequality \eqref{agmon inequality}, Poincar\'e's inequality, \eqref{leray laplacian inequality}, and Young's inequality one has
    \begin{align*}
        \left| \rho \int_{\Omega(t)} \left(\overline{\mbf{u}} \cdot \grad \overline{\mbf{u}}\right) A_{\sigma}\overline{\mbf{u}} \right| &\leq C\|\grad \overline{\mbf{u}}\|_{\mbf{L}^2(\Omega(t))}^\frac{3}{2} \|A_{\sigma} \overline{\mbf{u}}\|_{\mbf{L}^2(\Omega(t))}^\frac{3}{2}\\
        &\leq \frac{\mu}{10} \|A_{\sigma} \overline{\mbf{u}}\|_{\mbf{L}^2(\Omega(t))}^2 + C \|\grad \overline{\mbf{u}}\|_{\mbf{L}^2(\Omega(t))}^6.
    \end{align*}
    We estimate the ``force term'' using Young's inequality again, but here the treatment is more delicate than that of the previous estimate.
    By integrating by parts one finds that
    \[\int_{\Omega(t)} \mbf{B} \cdot A_{\sigma}\overline{\mbf{u}} = - \int_{\Omega(t)} P_\sigma \mbf{B} \cdot \Delta \overline{\mbf{u}} = \int_{\Omega(t)} \grad P_\sigma \mbf{B} : \grad \overline{\mbf{u}} - \int_{\partial \Omega(t)} ((P_\sigma \mbf{B} \cdot \grad) \overline{\mbf{u}}) \cdot \boldsymbol{\nu}, \]
    where this boundary integral is non-zero.
    To deal with this boundary integral we use the trace theorem (see \cite[Chapter 7]{adams2003sobolev}) to see that
    \[ \left| \int_{\partial \Omega(t)} ((P_\sigma \mbf{B} \cdot \grad) \overline{\mbf{u}}) \cdot \boldsymbol{\nu} \right| \leq C \|P_\sigma \mbf{B}\|_{\mbf{H}^1(\Omega(t))} \|\overline{\mbf{u}} \|_{\mbf{H}^{\frac{3}{2}}(\Omega(t))}, \]
    where $\mbf{H}^{\frac{3}{2}}(\Omega(t))$ is a Sobolev--Slobodeckij space.
    Now by using standard interpolation results (again we refer to \cite[Chapter 7]{adams2003sobolev}) we see that
    \[ \|\overline{\mbf{u}} \|_{\mbf{H}^{\frac{3}{2}}(\Omega(t))} \leq C \|\overline{\mbf{u}} \|_{\mbf{H}^{1}(\Omega(t))}^{\frac{1}{2}} \|\overline{\mbf{u}} \|_{\mbf{H}^{2}(\Omega(t))}^{\frac{1}{2}}, \]
    where the constant can be shown the be independent of $t$.
    Combining this with \eqref{eqn: leray H1 bound} one finds
    \[ \left| \int_{\Omega(t)} \mbf{B} \cdot A_{\sigma}\overline{\mbf{u}} \right| \leq C \| \mbf{B} \|_{\mbf{H}^1(\Omega(t))}\|\overline{\mbf{u}} \|_{\mbf{H}^{1}(\Omega(t))} + C \| \mbf{B} \|_{\mbf{H}^1(\Omega(t))}\|\overline{\mbf{u}} \|_{\mbf{H}^{1}(\Omega(t))}^{\frac{1}{2}} \|\overline{\mbf{u}} \|_{\mbf{H}^{2}(\Omega(t))}^{\frac{1}{2}}. \]
    Now by using \eqref{leray laplacian inequality} and Young's inequality suitably we obtain
    \[ \left| \int_{\Omega(t)} \mbf{B} \cdot A_{\sigma}\overline{\mbf{u}} \right| \leq  \frac{\kappa}{2}\| \mbf{B} \|_{\mbf{H}^1(\Omega(t))}^2 + \frac{C}{\kappa^2}\|\overline{\mbf{u}} \|_{\mbf{H}^{1}(\Omega(t))}^2 + \frac{\mu}{10} \|A_{\sigma} \overline{\mbf{u}}\|_{\mbf{L}^2(\Omega(t))}^2, \]
    where we have used the fact that $\kappa \in (0,1)$.
    For the final term we notice that
    \begin{multline}
    \left| \int_{\Omega(t)} \left((\widetilde{\mbf{u}} \cdot \grad \overline{\mbf{u}}) + (\overline{\mbf{u}} \cdot \grad\widetilde{\mbf{u}}) \right)A_{\sigma} \overline{\mbf{u}} \right|\\
    \leq \left(\| \widetilde{\mbf{u}}\|_{\mbf{L}^4(\Omega(t))} \| \overline{\mbf{u}}\|_{\mbf{W}^{1,4}(\Omega(t))} + \| \overline{\mbf{u}}\|_{\mbf{L}^4(\Omega(t))} \| \widetilde{\mbf{u}}\|_{\mbf{W}^{1,4}(\Omega(t)))}\right)\| A_{\sigma}\overline{\mbf{u}} \|_{\mbf{L}^2(\Omega(t))}. \label{fsi energyestimate pf2}
    \end{multline}
    We now use the $\mbf{H}^2$ analogue of \eqref{ladyzhenskaya inequality}, i.e.
    \[\|\mbf{u}\|_{\mbf{W}^{1,4}(\Omega(t))} \leq C \|\mbf{u}\|_{\mbf{H}^2(\Omega(t))}^{\frac{3}{4}}\|\mbf{u}\|_{\mbf{H}^1(\Omega(t))}^{\frac{1}{4}},\]
    as well as \eqref{leray laplacian inequality}, and Young's inequality for
    \begin{align*}
    \| \widetilde{\mbf{u}}\|_{\mbf{L}^4(\Omega(t))} \| \overline{\mbf{u}}\|_{\mbf{W}^{1,4}(\Omega(t))}\| A_{\sigma}\overline{\mbf{u}} \|_{\mbf{L}^2(\Omega(t))} &\leq C \| \widetilde{\mbf{u}}\|_{\mbf{L}^4(\Omega(t))} \| \overline{\mbf{u}}\|_{\mbf{H}^{1}(\Omega(t))}^{\frac{1}{4}}\| A_{\sigma}\overline{\mbf{u}} \|_{\mbf{L}^2(\Omega(t))}^{\frac{7}{4}}\\
    &\leq C \| \widetilde{\mbf{u}}\|_{\mbf{L}^4(\Omega(t))}^8 \| \grad \overline{\mbf{u}}\|_{\mbf{L}^{2}(\Omega(t))}^2 + \frac{\mu}{10}\| A_{\sigma}\overline{\mbf{u}} \|_{\mbf{L}^2(\Omega(t))}^2.
    \end{align*}
    Similarly, by using the Sobolev embedding $\mbf{H}^1(\Omega(t)) \hookrightarrow \mbf{L}^4(\Omega(t))$ and Young's inequality one finds
    \begin{align*}
        \| \overline{\mbf{u}}\|_{\mbf{L}^4(\Omega(t))} \| \widetilde{\mbf{u}}\|_{\mbf{W}^{1,4}(\Omega(t))}\| A_{\sigma}\overline{\mbf{u}} \|_{\mbf{L}^2(\Omega(t))}
        &\leq C\| \overline{\mbf{u}}\|_{\mbf{H}^1(\Omega(t))} \| \widetilde{\mbf{u}}\|_{\mbf{W}^{1,4}(\Omega(t))}\| A_{\sigma}\overline{\mbf{u}} \|_{\mbf{L}^2(\Omega(t))}\\
        &\leq C\|\grad \overline{\mbf{u}}\|_{\mbf{L}^2(\Omega(t))}^2 \| \widetilde{\mbf{u}}\|_{\mbf{W}^{1,4}(\Omega(t))}^2 + \frac{\mu}{10}\| A_{\sigma}\overline{\mbf{u}} \|_{\mbf{L}^2(\Omega(t))}^2.
    \end{align*}
    
	Hence we find that
	\begin{multline*}
		\frac{\rho}{2} \frac{\mathrm{d}}{\mathrm{d}t} \int_{\Omega(t)} |\grad \overline{\mbf{u}}|^2 + \frac{\mu}{2} \int_{\Omega(t)} |A_{\sigma} \overline{\mbf{u}}|^2 \leq \frac{\kappa}{2} \|\mbf{B}\|_{\mbf{H}^1(\Omega(t))}^2 + C\|\grad \mbf{u}\|_{\mbf{L}^2(\Omega(t))}^6\\
        + C \left(\frac{1}{\kappa^2} + \|\Div\mbf{V}\|_{L^\infty(\Omega(t))} + \|\mbf{V}\|_{\mbf{L}^\infty(\Omega(t))}^2 +\|\widetilde{\mbf{u}}\|_{\mbf{L}^4(\Omega(t))}^8 +\|\widetilde{\mbf{u}}\|_{\mbf{W}^{1,4}(\Omega(t))}^2\right)\|\grad \mbf{u}\|_{\mbf{L}^2(\Omega(t))}^2,
	\end{multline*}
	and one wants to use Lemma \ref{biharilasalle} with $\lambda(s) = s + s^3$.
	For a small, fixed $y_0 > 0$ one finds that 
	\[ \Lambda(y) = \int_{y_0}^y \frac{1}{s + s^3} \, \mathrm{d}s = \frac{1}{2}\log\left(\frac{y^2(1+y_0^2)}{y_0^2(1+y^2)}\right), \]
	and so 
	\[ \text{dom}(\Lambda^{-1}) = R(\Lambda) = \left(-\infty, \frac{1}{2}\log\left(\frac{1+y_0^2}{y_0^2}\right)\right). \]
	Hence one can verify that Lemma \ref{biharilasalle} holds for small $\widetilde{T}$ such that
	\begin{multline}
		\frac{1}{2} \log \left( \frac{\rho^2 \|\grad \mbf{u}_0\|_{\mbf{L}^2(\Omega(0))}^4 + C\left(\kappa \int_0^{\widetilde{T}} \|\mbf{B}\|_{\mbf{H}^1(\Omega(t))}^2 \right)^2}{4 + \rho^2 \|\grad \mbf{u}_0\|_{\mbf{L}^2(\Omega(0))}^4 + C\left(\kappa\int_0^{\widetilde{T}} \|\mbf{B}\|_{\mbf{H}^1(\Omega(t))}^2 \right)^2} \right)\\
        +\int_0^{\widetilde{T}}C \left( \frac{1}{\kappa^2} + \|\Div\mbf{V}\|_{L^\infty(\Omega(t))} + \|\mbf{V}\|_{\mbf{L}^\infty(\Omega(t))}^2 +\|\widetilde{\mbf{u}}\|_{\mbf{L}^4(\Omega(t))}^8 +\|\widetilde{\mbf{u}}\|_{\mbf{W}^{1,4}(\Omega(t))}^2\right) < 0, \label{smalltimecondition}
	\end{multline}
	where a tedious calculation then shows that \eqref{biharilasalle2} yields \eqref{fsi energyestimate2}, given that \eqref{smalltimecondition} holds --- we refer the reader to Appendix \ref{appendix: bihari use} for details.
    Such a $\widetilde{T}$ exists by continuity after noting that 
    \[ \frac{1}{2} \log \left( \frac{\rho^2 \|\grad \mbf{u}_0\|_{\mbf{L}^2(\Omega(0))}^4}{4 + \rho^2 \|\grad \mbf{u}_0\|_{\mbf{L}^2(\Omega(0))}^4} \right) < 0. \]
\end{proof}

In later arguments we shall choose $\kappa$ as a function of $\widetilde{T}$, namely $\kappa = \widetilde{T}^{\frac{1}{3}}$.

\begin{remark}
    \label{remark: nonzero terms}
    We note that the requirement \eqref{smalltimecondition} means that the right-hand side of \eqref{fsi energyestimate2} is finite as
    \[\mathcal{E}(\widetilde{T})\left(\rho^2 \|\grad \mbf{u}_0\|_{\mbf{L}^2(\Omega(0))}^4 + \left(\kappa \int_0^{\widetilde{T}} \|\mbf{B}\|_{\mbf{H}^1(\Omega(t))}^2\right)^2\right) < 4 + \rho^2 \|\grad \mbf{u}_0\|_{\mbf{L}^2(\Omega(0))}^4 + \left( \kappa\int_0^{\widetilde{T}} \|\mbf{B}\|_{\mbf{H}^1(\Omega(t))}^2\right)^2.\]
\end{remark}

We now use the estimates of Lemma~\ref{lemma: fsi energy estimates} to show the following.

\begin{lemma}
\label{lemma: fsi energy estimates2}
	On $[0,\widetilde{T}]$ for $\widetilde{T}$ as in the previous lemma, one has $\overline{\mbf{u}} \in H^1_{\mbf{L}^2}(\Phi)$, and there exists a unique pressure $\overline{p} \in L^2_{H^1 \cap L^2_0}(\Phi)$ such that
	\begin{equation}
        \begin{aligned}
		\int_{0}^{\widetilde{T}} \left(\left\| \ddt{\overline{\mbf{u}}} \right\|_{\mbf{L}^2(\Omega(t))}^2 + \|\overline{p}\|_{H^1(\Omega(t))}^2\right) &\leq \int_0^{\widetilde{T}} \|\mbf{B}\|_{\mbf{L}^2(\Omega(t))}^2\\
        &+C\left(1 + \int_0^{\widetilde{T}} \|\widetilde{\mbf{u}}\|_{\mbf{H}^2(\Omega(t))}^2 +  \sup_{t \in [0,\widetilde{T}]} \|\grad \widetilde{\mbf{u}}\|_{\mbf{L}^2(\Omega(t))}^2\right) \mathcal{D}(\widetilde{T})^2,
        \end{aligned}
        \label{fsi energyestimate3}
	\end{equation}
	for $\mathcal{D}(\cdot)$, $\mathcal{E}(\cdot)$ and $\kappa \in (0,1)$ as in the previous lemma.
\end{lemma}
\begin{proof}
	For almost all $t \in [0,\widetilde{T}]$ one finds that
    \[\rho P_\sigma \ddt{\overline{\mbf{u}}} = -\rho(\overline{\mbf{u}} \cdot \grad) \overline{\mbf{u}} - \mu A_{\sigma} \overline{\mbf{u}} - \rho (\widetilde{\mbf{u}} \cdot \grad)\overline{\mbf{u}} - \rho(\overline{\mbf{u}} \cdot \grad)\widetilde{\mbf{u}} - P_\sigma \mbf{B},\]
    and so
    \begin{multline*}
        \int_0^{\widetilde{T}} \left\| P_\sigma \ddt{\overline{\mbf{u}}} \right\|_{\mbf{L}^2(\Omega(t))}^2\\
        \leq C\int_0^{\widetilde{T}} \left(\|(\overline{\mbf{u}} \cdot \grad) \overline{\mbf{u}}\|_{\mbf{L}^2(\Omega(t))}^2 + \|A_{\sigma} \overline{\mbf{u}}\|_{\mbf{L}^2(\Omega(t))}^2 + \| (\widetilde{\mbf{u}} \cdot \grad)\overline{\mbf{u}}\|_{\mbf{L}^2(\Omega(t))}^2 + \| (\overline{\mbf{u}} \cdot \grad)\widetilde{\mbf{u}}\|_{\mbf{L}^2(\Omega(t))}^2 + \|P_\sigma \mbf{B}\|_{\mbf{L}^2(\Omega(t))}^2 \right),
    \end{multline*}
    where we note that
	\begin{align*}
		\int_0^T \|(\overline{\mbf{u}} \cdot \grad) \overline{\mbf{u}}\|_{\mbf{L}^2(\Omega(t))}^2 &\leq C\int_{0}^T \|\overline{\mbf{u}}\|_{\mbf{L}^\infty(\Omega(t))}^2 \|\grad \overline{\mbf{u}}\|_{\mbf{L}^2(\Omega(t))}^2\\
		&\leq C \left( \sup_{t \in [0,\widetilde{T}]} \|\grad \overline{\mbf{u}}\|_{\mbf{L}^2(\Omega(t))}^2 \right) \int_0^{\widetilde{T}} \|\overline{\mbf{u}}\|_{\mbf{H}^2(\Omega(t))}^2,
	\end{align*}
    where we have used the embedding $\mbf{H}^2(\Omega(t)) \hookrightarrow \mbf{L}^\infty(\Omega(t))$.
    We have similar estimates for $\int_0^{\widetilde{T}} \| (\widetilde{\mbf{u}} \cdot \grad)\overline{\mbf{u}}\|_{\mbf{L}^2(\Omega(t))}^2$ and $\int_0^{\widetilde{T}} \| (\overline{\mbf{u}} \cdot \grad)\widetilde{\mbf{u}}\|_{\mbf{L}^2(\Omega(t))}^2$.
	Thus by using \eqref{fsi energyestimate2} and \eqref{eqn: leray L2 bound} one finds that
    \begin{equation*}
        \int_0^{\widetilde{T}} \left\| P_\sigma \ddt{\overline{\mbf{u}}} \right\|_{\mbf{L}^2(\Omega(t))}^2 \leq \int_0^{\widetilde{T}} \|\mbf{B}\|_{\mbf{L}^2(\Omega(t))}^2 +C\left(1 + \int_0^{\widetilde{T}} \|\widetilde{\mbf{u}}\|_{\mbf{H}^2(\Omega(t))}^2 +  \sup_{t \in [0,\widetilde{T}]} \|\grad \widetilde{\mbf{u}}\|_{\mbf{L}^2(\Omega(t))}^2\right)\mathcal{D}(\widetilde{T})^2.
    \end{equation*}
	Likewise, for almost all $t \in [0,\widetilde{T}]$ one has that 
     \[\rho P_\sigma \ddt{\overline{\mbf{u}}} + \rho(\overline{\mbf{u}} \cdot \grad) \overline{\mbf{u}} - \mu A_{\sigma} \overline{\mbf{u}} + \rho (\widetilde{\mbf{u}} \cdot \grad)\overline{\mbf{u}} + \rho(\overline{\mbf{u}} \cdot \grad) \widetilde{\mbf{u}} - P_\sigma \mbf{B} \in \mbf{L}^2(\Omega(t)),\]
    where we use the Helmholtz decomposition to see this equals\footnote{Here we have written the non-solenoidal component as the pressure term, $\grad p$, and a sum of contributions from the terms involving $\overline{\mbf{u}}, \mbf{B}$.
    This is due to the fact that $P_\sigma$ does not commute with $\ddt{}$ or $\Delta$.}
    \[\grad \overline{p} + \rho\left(P_\sigma \ddt{\overline{\mbf{u}}} - \ddt{\overline{\mbf{u}}}\right) + (\mbf{B} - P_\sigma\mbf{B}) - \mu(\Delta \overline{\mbf{u}} + A_{\sigma}\overline{\mbf{u}}).\]
	The $L^2_{H^1}(\Phi)$ estimate on $\overline{p}$, and $L^2_{\mbf{L}^2}(\Phi)$ estimate for the non-solenoidal component of $\ddt{\overline{\mbf{u}}}$ now follow from using Poincar\'e's inequality and similar arguments to the above.
\end{proof}

By combining the results in this section along with those in Appendix \ref{appendix:stokes} we can show suitable $H^1_{\mbf{L}^2}(\Phi) \cap L^2_{\mbf{H}^2}(\Phi)$ and $L^2_{H^1}(\Phi)$ bounds for $\mbf{u}$ and $p$ respectively, which solve \eqref{eqn: evolving NS} with inhomogeneous boundary conditions.

\begin{remark}
\label{remark: eqn: evolving NS uniqueness}
It is well known that solutions of the three dimensional Navier--Stokes equations (on a stationary domain) which satisfy a Prodi--Serrin condition \cite{prodi1959teorema, serrin1962interior} have a weak-strong uniqueness property --- see for instance \cite[Theorem 8.19]{robinson2016three}.
On a prescribed evolving domain this remains true, and a Prodi--Serrin condition of the form
\[ \mbf{u} \in L^r_{\mbf{L}^s}(\Phi), \quad \frac{3}{s} + \frac{2}{r} = 1, \, s > 3,\]
implies the uniqueness of strong solutions in the wider class of Leray--Hopf type solutions, cf.~\cite{Bock77NavierStokes,salvi1988navier}.
As such, since the strong solutions we have discussed satisfy a Prodi--Serrin condition they are unique.
We refer the reader to \cite{maity2024uniqueness,maity2025regularity,muha2021uniqueness} for recent discussion on the implications of the Prodi--Serrin condition in fluid-structure interaction problems.
\end{remark}

\subsection{Construction of the velocity field}
\label{subsection: construction}
In this subsection we now give a concrete construction of the velocity field, $\mbf{V}$, to be used in our iterative argument along with the estimates established in \S\ref{subsection: construction}.
We assume that we are given rigid body quantities,~$(\mbf{q}, \boldsymbol{\omega}) \in \mbf{H}^2([0,T]) \times \mbf{H}^1([0,T])$, which are compatible with the initial data \eqref{FSI IC} in the following sense.
\begin{definition}
    \label{defn: compatible with IC}
    We say that functions $(\mbf{q}, \boldsymbol{\omega}) \in \mbf{H}^2([0,T]) \times \mbf{H}^1([0,T])$ are compatible with the initial data if:
    \begin{enumerate}
        \item \[ \mbf{q}(0) = \mbf{q}_0, \quad \mbf{q}'(0) = \mbf{v}_0, \quad \boldsymbol{\omega}(0)= \boldsymbol{\omega}_0.  \]
        \item The centre of mass $\mbf{q}$ and angular velocity $\boldsymbol{\omega}$ define the motion of a rigid body, $B(t)$, by evolving the points in $\mbf{x}(t) \in B(t)$ via
        \begin{equation}
            \frac{\mathrm{d}\mbf{x}}{\mathrm{d}t} = \frac{\mathrm{d}\mbf{q}}{\mathrm{d}t} + \boldsymbol{\omega} \times (\mbf{x} - \mbf{q}(t)), \quad \mbf{x}(0) = \mbf{x}_0 \quad \forall \mbf{x}_0 \in B(0). \label{eqn: points ODE}
        \end{equation}
        Moreover, this rigid body is such that, for all $t \in [0,T]$,
        \[ \mathrm{dist}(B(t), \Gamma) > \frac{\delta}{2}, \]
        where $\delta > 0$ is some constant such that \(\mathrm{dist}(B(0), \Gamma) > \delta\).
        
    \end{enumerate}
    
\end{definition}
We note that $\mbf{q}$ and $\boldsymbol{\omega}$ are sufficiently smooth so that the point evaluations in this first condition are well defined.
This second condition is required so that there is no collision of the rigid body with the fixed portion of boundary $\Gamma$.
An example of functions $\mbf{q}, \boldsymbol{\omega}$ which are compatible with the initial data in the above sense are the functions
\begin{gather*}
    \mbf{q}(t) := \mbf{q}_0 + t \mbf{v}_0 \qquad \boldsymbol{\omega}(t) := \boldsymbol{\omega}_0,
\end{gather*}
on some sufficiently small time interval which is determined solely by the initial data.

Given admissible functions in the sense of Definition \ref{defn: compatible with IC} this determines $B(t)$ by evolving points via~\eqref{eqn: points ODE}.
This in turn defines the evolving domain, $\Omega(t) := (\Omega\setminus B(t))^\circ$, in which we solve the fluid equations using the evolving function space framework discussed in \S\ref{subsection: function spaces}.
To use these evolving function spaces we require a parametrisation of the domain $\Omega(t)$, which we now construct via an appropriately defined velocity field.

For a given point, $\mbf{x}_0 \in \partial B(0)$ we know that this point evolves by \eqref{eqn: points ODE}, and similarly for a point $\mbf{x}_0 \in \Gamma$ we find that it remains fixed under the evolution of the domain.
We want to extend this flow to all points in $\overline{\Omega(0)}$ to define a parametrisation of $\overline{\Omega(t)}$.
One such way to do this to do this is to consider the harmonic extension of the boundary data on $\Omega(t)$.
That is to say, for each $t \in [0,T]$, we define $\mbf{V}(t) : \Omega(t) \rightarrow \mbb{R}^3$ to be the unique solution of
\begin{gather}
	\begin{cases}
	-\Delta \mbf{V}(t) = 0, &\text{on } \Omega(t),\\
	\mbf{V}(\mbf{x},t) = \frac{\mathrm{d}\mbf{q}}{\mathrm{d}t} + \boldsymbol{\omega} \times (\mbf{x} - \mbf{q}(t)), &\text{on } \partial B(t),\\
	\mbf{V}(\mbf{x},t) = 0, &\text{on } \Gamma,
\end{cases}\label{eqn: velocity harmonic extension}
\end{gather}
where $-\Delta$ is understood componentwise.
By standard results for the Laplace equation (see for instance \cite{gilbarg1977elliptic}) one finds that
\begin{align}
	\|\mbf{V}\|_{\mbf{C}^{2,1}(\Omega(t))} \leq C \left\| \frac{\mathrm{d}\mbf{q}}{\mathrm{d}t} + \boldsymbol{\omega} \times (\mbf{x} - \mbf{q}(t)) \right\|_{\mbf{C}^{2,1}(\Omega(t))}, \label{eqn: harmonic extension bound}
\end{align}
where the constant $C$ depends only on the geometry of $\Gamma$ and $\partial B(0)$.
This choice of extension is essentially arbitrary, provided one has a bound like \eqref{eqn: harmonic extension bound}.
Perhaps a more suitable approach would be to consider an extension given by solving the Stokes problem (see Appendix \ref{appendix:stokes}), but we do not consider other choices of velocity field here.
From the identification of $\mbf{C}^{2,1}(\Omega(t)) = \mbf{W}^{3,\infty}(\Omega(t))$ we note that the third order spatial derivatives of $\mbf{V}$ exist almost everywhere and are essentially bounded.

Continuity in the time argument is harder to establish, as it is not clear what continuity means in this case since the functions $\mbf{V}(t), \mbf{V}(s)$ are defined on different domains for $t \neq s$.
To understand continuity here, we consider an arbitrary family of diffeomorphisms $(\Phi(t))_{t \in [0,T]}$ such that $\Phi(t): \Omega(0) \rightarrow \Omega(t)$ is a $C^3$ diffeomorphism, and $\Phi(0) = \mathrm{id}$.
We then understand continuity of $t \mapsto \mbf{V}(t)$ to mean that the pullback $t \mapsto \Phi_{-t} \mbf{V}(t)$ is continuous in the usual sense for all such families $(\Phi(t))_{t \in [0,T]}$.
With this in mind, one can then establish continuity of $t \mapsto \Phi_{-t} \mbf{V}(t)$ by considering the weak formulation of \eqref{eqn: velocity harmonic extension} pulled back onto $\Omega(0)$.
We omit the details of this calculation as it is tedious and yields no valuable insight --- we refer to \cite[Appendix A]{elliott2024navier} for details on a similar calculation.
In particular, from this one may deduce the measurability of $t \mapsto \mbf{V}(t)$, and thus one finds that $\mbf{V}$ is a Carath\'eodory function.

\subsubsection{The pushforward map}
Since $\mbf{V}$ is a Carath\'eodory function we may use the Carath\'eodory existence theorem (see for instance \cite{hale2009ordinary}) for ordinary differential equations to see that
\begin{gather}
	\begin{aligned}
		\frac{\mathrm{d} \Phi}{\mathrm{d}t}(\mbf{p},t) &= \mbf{V}(\Phi(\mbf{p},t),t),\\
		\Phi(\mbf{p},0) &= \mbf{p},
	\end{aligned}
	\label{eqn: parametrisation defn}
\end{gather}
admits a solution on $[0,T]$, for all $\mbf{p} \in \overline{\Omega(0)}$.
Moreover the Lipschitz continuity of $\mbf{V}(\cdot, t)$ implies the solution to this ODE is in fact unique.
By construction one finds that $\Phi(\overline{\Omega(0)},t) = \overline{\Omega(t)}$, and moreover one finds that $\Phi$ is such that $\Phi(t): \Omega(0) \rightarrow \Omega(t)$ is invertible for all $t \in [0, T]$.

Owing to the regularity of $\mbf{V}$ in space, one can show further spatial regularity of $\Phi$ (see for instance \cite[\S 32]{arnold1992ordinary}).
By differentiating the above ODE in space one finds that
\begin{gather*}
	\begin{aligned}
		\frac{\mathrm{d} \grad \Phi}{\mathrm{d}t}(\mbf{p},t) &= \grad\mbf{V}(\Phi(\mbf{p},t),t) \grad \Phi(\mbf{p},t),\\
		\grad \Phi(\mbf{p},0) &= \mbb{I},
	\end{aligned}
\end{gather*}
and a straightforward application of Gr\"onwall's inequality yields
\[ \|\grad \Phi(\mbf{p},t)\|_{\mbf{C}^{1,1}(\Omega(0))} \leq \exp \left(C \int_0^t \left\| \grad \mbf{V}(\Phi(\mbf{p},s),s) \right\|_{\mbf{C}^{1,1}(\Omega(s))} \, \mathrm{d}s \right), \]

and using \eqref{eqn: harmonic extension bound} one finds then that
\begin{align}
	\|\grad \Phi(\mbf{p},t)\|_{\mbf{C}^{1,1}(\Omega(0))} \leq \exp \left( C \int_0^t \left\| \frac{\mathrm{d}\mbf{q}}{\mathrm{d}t} + \boldsymbol{\omega} \times (\mbf{x} - \mbf{q}(s)) \right\|_{\mbf{C}^{2,1}(\Omega(s))} \, \mathrm{d}s \right). \label{grad bound}
\end{align}
Similarly, noting that the Jacobian determinant satisfies the ODE
\[\frac{\mathrm{d}}{\mathrm{d}t} \det(\grad \Phi(\mbf{p}, t)) = \det(\grad \Phi(\mbf{p}, t)) (\Div \mbf{V})(\Phi(\mbf{p}, t),t), \quad \det(\grad \Phi(\mbf{p},0)) = 1,\]
one can show that
\begin{align}
	|\det(\grad \Phi(\mbf{p}, t))| \geq \exp \left(-C \int_0^t \left\| \frac{\mathrm{d}\mbf{q}}{\mathrm{d}t} + \boldsymbol{\omega} \times (\mbf{x} - \mbf{q}(s)) \right\|_{\mbf{C}^{2,1}(\Omega(s))} \, \mathrm{d}s \right) >0, \label{grad bound2}
\end{align}
and since $\det(\grad \Phi(\mbf{p},0)) = 1$ it follows that $\det(\grad \Phi(\mbf{p}, t)) > 0$, and $\grad \Phi$ is invertible.

\subsubsection{Compatibility of function spaces}
To use the evolving function space framework discussed in \S\ref{subsection: function spaces} we must first show compatibility of the evolving function space pairs, as in Definition~\ref{defn: compatible family}.
For this, we begin with the following lemma.
\begin{lemma}
	\label{transformation lemma}
	Let $\Phi: \Omega(0) \rightarrow \Omega(t)$ and $f: \Omega(0) \rightarrow \mbb{R}$ be sufficiently smooth, we define $\widehat{f} := f \circ \Phi : \Omega(t) \rightarrow \mbb{R}$.
	Then the gradient and Hessian of $f$ transform as
	\begin{gather}
		(\grad f) \circ \Phi = (\grad \Phi)^{-T} \grad \widehat{f},\label{transformed grad}\\
		(\grad^2 f) \circ \Phi = (\grad \Phi^1)^{-T} \grad \big((\grad \Phi)^{-T} \grad \widehat{f}\big) \label{transformed hessian},
	\end{gather}
	and the Laplacian transforms as
	\begin{align}
		(\Delta f) \circ \Phi = \frac{1}{\sqrt{g}} \Div (\sqrt{g} G^{-1} \grad \widehat{f}), \label{transformed laplacian}
	\end{align}
	where $G = (\grad \Phi)^T \grad \Phi, g = \det(G) = \det(\grad \Phi)^2$.
	Similarly if $f:\Omega(0) \times [0,T] \rightarrow \mbb{R}$ is sufficiently smooth then one has
	\begin{align}
		\left( \ddt{f}\right) \circ \Phi = \ddt{\widehat{f}} - (\grad \Phi)^{-1} \frac{\mathrm{d} \Phi}{\mathrm{d}t} \cdot \grad \widehat{f} = \ddt{\widehat{f}} - (\grad \Phi)^{-1} \mbf{V} \cdot \grad \widehat{f}.\label{transformed ddt}
	\end{align}
\end{lemma}
\begin{proof}
	\eqref{transformed grad} and \eqref{transformed laplacian} follow almost identically to the proofs of \cite[Lemma 3.3]{church2020domain} and \cite[Lemma 3.4]{church2020domain} respectively.
	The proof of \eqref{transformed hessian} follows a straightforward calculation using \eqref{transformed grad} with $f$ replaced by $\frac{\partial f}{\partial x_i}$ for $i= 1,2,3$.
	
	The proof for \eqref{transformed ddt} follows by using the chain rule, \eqref{transformed grad}, and the definition of $\Phi$ so that
	\[ \frac{\mathrm{d}}{\mathrm{d}t} f(\Phi(x,t),t) = \left( \ddt{f} \right) \circ \Phi + \frac{\mathrm{d} \Phi}{\mathrm{d}t} \cdot(\grad {f}) \circ \Phi = \left( \ddt{f} \right) \circ \Phi + \mbf{V} \cdot (\grad \Phi)^{-T} \grad \widehat{f}. \]
\end{proof}
This result extends to vector-valued functions in the obvious way.
With this preliminary lemma one can show the following compatibility result.
\begin{lemma}
	\label{pullback lemma}
	For $\Phi,f, \widehat{f}$ as in the previous lemma, one has
	\[ \| f \|_{\mbf{W}^{k,p}(\Omega(0))} \leq C(\mbf{q}, \boldsymbol{\omega})\| \widehat{f} \|_{\mbf{W}^{k,p}(\Omega(t))},  \]
	for $k=0,1,2$ and $p \in [1,\infty]$.
	Here $C(\mbf{q}, \boldsymbol{\omega})$ is a constant depending only on $\mbf{q}$, $\mbf{\omega}$, and $\Omega(0)$.
\end{lemma}
\begin{proof}
	We show the result for $k =2$, and $p \neq \infty$.
	By a straightforward change of variables and using \eqref{transformed grad}, \eqref{transformed hessian} one has
	\begin{align*}
		\| f \|_{\mbf{W}^{2,p}(\Omega(0))}^p = \int_{\Omega(t)} J\left(|\widehat{f}|^p + |(\grad \Phi)^{-T} \grad \widehat{f}|^p + |(\grad \Phi)^{-T} \grad \big((\grad \Phi)^{-T} \grad \widehat{f}\big)|^p\right),
	\end{align*}
	for $J = \det(\grad \Phi)$.
	From \eqref{grad bound} it is clear that
	\[ \|J\|_{L^\infty(\Omega(t))} \leq \exp \left( C \int_0^t \left\| \frac{\mathrm{d}\mbf{q}}{\mathrm{d}t} + \boldsymbol{\omega} \times (\mbf{x} - \mbf{q}(t)) \right\|_{\mbf{C}^{2,1}(\Omega(s))} \, \mathrm{d}s \right), \]
	and from \eqref{grad bound2} it is a straightforward calculation to show that
	\[ \| (\grad \Phi)^{-T} \|_{\mbf{L}^\infty(\Omega(t))} \leq \exp\left( C \int_0^t \left\| \frac{\mathrm{d}\mbf{q}}{\mathrm{d}t} + \boldsymbol{\omega} \times (\mbf{x} - \mbf{q}(t)) \right\|_{\mbf{C}^{2,1}(\Omega(s))} \, \mathrm{d}s \right). \] 
	Lastly one uses the fact that $\grad (\grad \Phi)^{-T} = (\grad \Phi)^{-T} \grad (\grad \Phi)^T (\grad \Phi)^{-T}$ to see that
    \begin{align*}
        |(\grad \Phi)^{-T} \grad \big((\grad \Phi)^{-T} \grad \widehat{f}\big)|^p &\leq \|(\grad \Phi)^{-T}\|_{\mbf{L}^\infty(\Omega(t))}^{3p}\|\grad (\grad \Phi)^T\|_{\mbf{L}^\infty(\Omega(t))}^p |\grad \widehat{f}|^p\\
        &+ \|(\grad \Phi)^{-T}\|_{\mbf{L}^\infty(\Omega(t))}^{2p}|\grad^2 \widehat{f}|^p,
    \end{align*}
	from which the result readily follows.
	The case of $p = \infty$ uses similar arguments and is omitted.
\end{proof}

From the regularity of $\mbf{q}$ and $\boldsymbol{\omega}$ one finds that $C \left\| \frac{\mathrm{d}\mbf{q}}{\mathrm{d}t} + \boldsymbol{\omega} \times (\mbf{x} - \mbf{q}(t)) \right\|_{\mbf{C}^{2,1}(\Omega(t))}$ can be bounded independent of $t$, and so this translates into a uniform in time bound for $C(\mbf{q}, \boldsymbol{\omega})$.
One can similarly show an inequality of the form
\[ \| \widehat{f} \|_{\mbf{W}^{k,p}(\Omega(t))} \leq \widetilde{C}(\mbf{q}, \boldsymbol{\omega})\| {f} \|_{\mbf{W}^{k,p}(\Omega(0))},  \]
for some constant $\widetilde{C}(\mbf{q}, \boldsymbol{\omega})$ independent of $t$ but possibly depending on $\mbf{q}, \boldsymbol{\omega}$.
From this it is straightforward to show compatibility (as in Definition~\ref{defn: compatible family}) of the pairs $(\mbf{W}^{k,p}(\Omega(t)), \Phi(t))$ for $k=0,1,2$ and $p \in [1,\infty]$ --- we omit further details.
With this preliminary set up, we may now use the evolving function spaces detailed in \cite{alphonse2023function,alphonse2015abstract}.

\section{The iterative argument}
\label{section:FSI iteration}

We now expound upon our iterative process.
Given functions $(\mbf{q}, \boldsymbol{\omega}) \in \mbf{H}^2([0,T]) \times \mbf{H}^1([0,T])$ which are compatible with the initial data in the sense of Definition \ref{defn: compatible with IC} one can define an evolving domain $\bigcup_{t \in [0,T]}\Omega(t) \times \{t \} \subset \mbb{R}^3 \times [0,T] $, which is parametrised by the flow map $\Phi(\cdot, t)$ constructed in \S\ref{subsection: construction}.
With this we use the evolving function space framework discussed in \S\ref{subsection: function spaces} to solve the evolving Navier--Stokes equations on $\Omega(t)$ for \emph{strong solutions}
\[\mbf{u} \in H^1_{\mbf{L}^2}(\Phi) \cap L^2_{\mbf{H}^2}(\Phi), \qquad p \in  L^2_{H^1}(\Phi),\]
as in Section~\ref{section: evolving NS}.
With these strong solutions we define functions $(\mbf{Q}, \mbf{W}) \in \mbf{H}^2([0,\widetilde{T}]) \times \mbf{H}^1([0,\widetilde{T}])$, for $\widetilde{T}$ as in Lemma~\ref{lemma: fsi energy estimates}, to be the unique solutions to
\begin{subequations}
\label{eqn: updated rigid body}
\begin{gather}
	m \frac{\mathrm{d}^2 \mbf{Q}}{\mathrm{d}t^2} = - \int_{\partial B(t)} \mbb{T} \boldsymbol{\nu},\label{eqn: updated position}\\
	\frac{\mathrm{d}}{\mathrm{d}t} \left( \mbb{J} \mbf{W} \right) = - \int_{\partial B(t)} (\mbf{x} - \mbf{q}(t)) \times \mbb{T} \boldsymbol{\nu}, \label{eqn: updated angular}
\end{gather}
\end{subequations}
such that
\[ \mbf{Q}(0) = \mbf{q}_0, \quad \mbf{Q}'(0) = \mbf{v}_0, \quad \mbf{W}(0) = \boldsymbol{\omega}_0, \]
and
\[ \mbb{J} := \int_{B(t)} \rho_B \left( |\mbf{x}-\mbf{q}(t)|^2 \mbb{I} - (\mbf{x} - \mbf{q}(t)) \otimes (\mbf{x} - \mbf{q}(t)) \right) .\]
We recall that $m = \rho_B |\Omega(0)|$ is the constant mass of the rigid body $B(t)$, and $\rho_B$ is the constant density of the rigid body.
Note that in contrast to~\eqref{eqn: FSI rigid body} the integral terms in \eqref{eqn: updated rigid body} are now given.

We summarise this construction as a map $\mathscr{F}: D(\mathscr{F}) \rightarrow \mbf{H}^2([0,\widetilde{T}]) \times \mbf{H}^1([0,\widetilde{T}])$, given by
\[ \mathscr{F} : (\mbf{q}, \boldsymbol{\omega}) \mapsto (\mbf{Q}, \mbf{W}). \]
Here the domain,~$D(\mathscr{F})$, is an appropriate subset of $\mbf{H}^2([0,T]) \times \mbf{H}^1([0,T])$ which we shall define properly in \S\ref{subsection: fixed point prelims}.
This map is highly non-local due to the integrals in \eqref{eqn: updated rigid body} involving the pair $(\mbf{u}, p)$ obtained by solving the Navier--Stokes equations on the domain defined by $(\mbf{q}, \boldsymbol{\omega})$.

\begin{remark}
\label{remark: construction}
At this stage it is not clear that $\mbf{q}$ and $\boldsymbol{\omega}$ being compatible with the initial data (in the sense of Definition \ref{defn: compatible with IC}) implies that $\mbf{Q}$ and $\boldsymbol{W}$ are compatible with the initial data.
In particular, it may be the case that $\mbf{Q}$ and $\mbf{W}$ allow for collisions with the boundary before time $T$ --- this will be a point we revisit later.
\end{remark}

\begin{remark}
    \label{remark: regularity}
    From the details of this iteration it is clear why we required strong solutions, as to make sense of the integrals in~\eqref{eqn: updated position} and~\eqref{eqn: updated angular} we understand $\grad \mbf{u} $ and $p$ in the trace sense which further regularity than Leray--Hopf type solutions allow.
        Subsequently since one has that the strong solutions $(\mbf{u},p) $ exist on a time interval $[0,\widetilde{T}] \subseteq [0,T]$ there is an issue regarding the time interval on which we may use this construction --- this will be resolved with the above issue concerning the compatibility of the pair $(\mbf{Q}, \mbf{W})$.
        It is likely that one can extend the approach in this paper by using weaker function spaces, provided one can show regularity $p|_{\partial B} \in L^2_{L^1}(\Phi)$ and $\grad \mbf{u}|_{\partial B} \in L^2_{L^1}(\Phi)$.
        For example, this is the case if one considers $p \in L^r_{B^{1/s}_{s,1}}(\Phi)$, for $B^s_{p,q}(\Omega(t))$ denoting a Besov space on $\Omega(t)$, cf.~\cite[Chapter 7]{adams2003sobolev}.
        This approach may let one consider less regular initial data than we have here, similar to the approach in \cite{geissert2013Lp}.
\end{remark}

The nature of our construction now means that the existence of solutions to \eqref{eqn: FSI fluid}, \eqref{eqn: FSI rigid body} becomes a matter of proving that the map $\mathscr{F}$ admits a fixed point.
To do this one must resolve the issues discussed in Remark~\ref{remark: construction} --- namely that we need $\widetilde{T} = T$ so that the domain and codomain of $\mathscr{F}$ are the same space.
This issue arises since strong solutions to \eqref{eqn: evolving NS} are only local in time, and one must also avoid collision\footnote{For no slip boundary conditions it is known that no such collision can occur~\cite{hillairet2007lack,hillairet2009collisions,san2002global}, but this is not the case for slip boundary conditions~\cite{chemetov2017motion}.} of the rigid body with the fixed boundary $\Gamma$.
We resolve this issue in following lemma which shows one can choose a time interval that avoids this issue.

\begin{lemma}
	\label{lemma: non degeneracy1}
    Let $(\mbf{q}, \boldsymbol{\omega}) \in \mbf{H}^2([0,T]) \times \mbf{H}^1([0,T])$ be compatible with the initial data in the sense of Definition \ref{defn: compatible with IC}, and such that
    \[ \|\mbf{q}\|_{\mbf{H}^2([0,T])} + \|\boldsymbol{\omega}\|_{\mbf{H}^1([0,T])} \leq R, \]
    for some given $R > 0$, determined by the initial data, which is sufficiently large.
    Then there exists some sufficiently small $T$, independent of $\mbf{q}$ and $\boldsymbol{\omega}$, such that $(\mbf{Q}, \mbf{W}) := \mathscr{F}(\mbf{q}, \boldsymbol{\omega})$ exists on $[0,T]$.
\end{lemma}
\begin{proof}
    Recall from Lemma \ref{lemma: fsi energy estimates} that $\mbf{Q}$ and $\mbf{W}$ are defined on $[0,\widetilde{T}]$, where there are two situations which may cause this time interval to be a strict subset of $[0,T]$: a collision of the rigid body with $\Gamma$, and the interval on which strong solutions to \eqref{eqn: evolving NS} shrinking due to the condition \eqref{smalltimecondition}.
	We consider this second case first, as it will help our treatment of the first case.
	To this end, for our given functions $(\mbf{q}, \boldsymbol{\omega})$, we show that $\widetilde{T}$ as in Lemma~\ref{lemma: fsi energy estimates} is independent of $\mbf{q}$ and $\boldsymbol{\omega}$.
    Recall that $\widetilde{T}$ was chosen so that
	\begin{align*}
		&\frac{1}{2} \log \left( \frac{\rho^2 \|\grad \mbf{u}_0\|_{\mbf{L}^2(\Omega(0))}^4 + C\left(\kappa \int_0^{\widetilde{T}} \|\mbf{B}\|_{\mbf{H}^1(\Omega(t))}^2 \right)^2}{4 + \rho^2 \|\grad \mbf{u}_0\|_{\mbf{L}^2(\Omega(0))}^4 + C\left(\kappa\int_0^{\widetilde{T}} \|\mbf{B}\|_{\mbf{H}^1(\Omega(t))}^2 \right)^2} \right)\\
        &+\int_0^{\widetilde{T}}C \left(\frac{1}{\kappa^2} + \|\Div\mbf{V}\|_{L^\infty(\Omega(t))} + \|\mbf{V}\|_{\mbf{L}^\infty(\Omega(t))}^2 +\|\widetilde{\mbf{u}}\|_{\mbf{L}^4(\Omega(t))}^8 +\|\widetilde{\mbf{u}}\|_{\mbf{W}^{1,4}(\Omega(t))}^2\right)< 0.
	\end{align*}
	This is equivalent to
	\begin{align}
		\rho^2 \| \grad \mbf{u}_0\|_{\mbf{L}^2(\Omega(0))}^4 \mathcal{E}(\widetilde{T}) + C (\mathcal{E}(\widetilde{T}) -1) \mathcal{B}(\widetilde{T}) < 4 + \rho^2 \| \grad \mbf{u}_0\|_{\mbf{L}^2(\Omega(0))}^4 ,\label{smalltimecondition2}
	\end{align}
    where we have defined
    \begin{gather*}
            \mathcal{E}(\widetilde{T}) := \exp\left(C\int_0^{\widetilde{T}} \left(\frac{1}{\kappa^2} + \|\Div\mbf{V}\|_{L^\infty(\Omega(t))} + \|\mbf{V}\|_{\mbf{L}^\infty(\Omega(t))}^2 +\|\widetilde{\mbf{u}}\|_{\mbf{L}^4(\Omega(t))}^8 +\|\widetilde{\mbf{u}}\|_{\mbf{W}^{1,4}(\Omega(t))}^2\right)\right),\\
            \mathcal{B}(\widetilde{T}) := \left(\kappa \int_0^{\widetilde{T}} \|\mbf{B}\|_{\mbf{H}^1(\Omega(t))}^2\right)^2.
        \end{gather*}
    We recall the body force, $\mbf{B}$, as defined in~\eqref{eqn: body force defn}, is determined by $\mbf{q}$ and $\boldsymbol{\omega}$.
    From here on we fix $\kappa = \widetilde{T}^{\frac{1}{3}}$.
	From \eqref{eqn: harmonic extension bound} and \eqref{stokes regularity1} we observe that one has
	\begin{align*}
		\|\mbf{V}\|_{\mbf{C}^{2,1}(\Omega(t))} &\leq C \left\| \frac{\mathrm{d}\mbf{q}}{\mathrm{d}t} + \boldsymbol{\omega} \times (\mbf{x} - \mbf{q}(t)) \right\|_{\mbf{C}^{2,1}(\Omega(t))},\\
        \|\widetilde{\mbf{u}}\|_{\mbf{H}^2(\Omega(t))} &\leq C \left\| \frac{\mathrm{d}\mbf{q}}{\mathrm{d}t} + \boldsymbol{\omega} \times (\mbf{x} - \mbf{q}(t)) \right\|_{\mbf{H}^{2}(\Omega(t))},
	\end{align*}
    respectively.
    Using these we find that
    \begin{align*}
        \mathcal{E}(\widetilde{T}) &\leq \exp\left(C\widetilde{T}^{\frac{1}{3}} + C \int_0^{\widetilde{T}} \left\| \frac{\mathrm{d}\mbf{q}}{\mathrm{d}t} + \boldsymbol{\omega} \times (\mbf{x} - \mbf{q}(t)) \right\|_{\mbf{C}^{2,1}(\Omega(t))}^8 \right)\\
        &\leq \exp\left(C \widetilde{T}^{\frac{1}{3}} + C \widetilde{T} R^8\right) =: F_1(\widetilde{T}, R),
    \end{align*}
    where we have used the Sobolev embeddings\footnote{An important observation here is that in one dimension these Sobolev embeddings hold with a constant independent of $T$.} $\mbf{H}^2([0,T]) \hookrightarrow \mbf{C}^{1,\frac{1}{2}}([0,T])$ and $\mbf{H}^1([0,T]) \hookrightarrow \mbf{C}^{0,\frac{1}{2}}([0,T])$, and the fact that $\kappa = \widetilde{T}^{\frac{1}{3}}$.
    We need to obtain a similar bound for $\mathcal{B}(\widetilde{T})$ by combining \eqref{body force bound1}, \eqref{stokes regularity4} so that
    \begin{align*}
        &\mathcal{B}(\widetilde{T}) = \widetilde{T}^{\frac{2}{3}}\left(\int_0^{\widetilde{T}} \|\mbf{B}\|_{\mbf{L}^2(\Omega(t))}^2\right)^2\\
        &\leq C\widetilde{T}^{\frac{2}{3}} \int_0^{\widetilde{T}} \left\| \frac{\mathrm{d}\mbf{q}}{\mathrm{d}t} + \boldsymbol{\omega} \times (\mbf{x} - \mbf{q}) \right\|_{\mbf{C}^{2,1}(\Omega(t))}^8 \exp \left( C \int_0^t \left\| \frac{\mathrm{d}\mbf{q}}{\mathrm{d}t} + \boldsymbol{\omega} \times (\mbf{x} - \mbf{q}) \right\|_{\mbf{C}^{2,1}(\Omega(s))} \, \mathrm{d}s \right)\\
        &+ C\widetilde{T}^{\frac{2}{3}}\int_0^{\widetilde{T}} \left\| \frac{\mathrm{d}^2\mbf{q}}{\mathrm{d}t^2} + \frac{\mathrm{d}\boldsymbol{\omega}}{\mathrm{d}t}\times(\mbf{x} - \mbf{q}) - \boldsymbol{\omega}(t) \times \frac{\mathrm{d}\mbf{q}}{\mathrm{d}t} \right\|_{\mbf{H}^1(\Omega(t))}^4 \exp \left( C \int_0^t \left\| \frac{\mathrm{d}\mbf{q}}{\mathrm{d}t} + \boldsymbol{\omega} \times (\mbf{x} - \mbf{q}) \right\|_{\mbf{C}^{2,1}(\Omega(s))} \, \mathrm{d}s \right).
    \end{align*}
    Using the boundedness of $\mbf{q}$ and $\boldsymbol{\omega}$ we find that
    \[ \mathcal{B}(\widetilde{T}) \leq C\widetilde{T}^{\frac{8}{3}} R^8 \exp(C\widetilde{T}R) + C\widetilde{T}^{\frac{2}{3}} R^4 \exp(C \widetilde{T}R) =: F_2(\widetilde{T}, R). \]

    Hence one finds that
    \begin{multline*}
        \rho^2 \| \grad \mbf{u}_0\|_{\mbf{L}^2(\Omega(0))}^4 \mathcal{E}(\widetilde{T}) + C (\mathcal{E}(\widetilde{T}) -1) \mathcal{B}(\widetilde{T})\\
        \begin{aligned}
        \leq \rho^2 \| \grad \mbf{u}_0\|_{\mbf{L}^2(\Omega(0))}^4 F_1(\widetilde{T}, R)+ C \left( F_1(\widetilde{T}, R) - 1 \right)F_2(\widetilde{T}, R),
        \end{aligned}
    \end{multline*}
    where one observes that for a fixed $R$,
    \[ \lim_{\widetilde{T} \rightarrow 0^+} F_1(\widetilde{T}, R) = 1, \qquad \lim_{\widetilde{T} \rightarrow 0^+} F_2(\widetilde{T}, R) = 0. \]
    Thus one can now take $\widetilde{T}$ sufficiently small in terms of $R$ and the initial data so that \eqref{smalltimecondition2} holds independent of $\mbf{q} $ and $ \boldsymbol{\omega}$.
    Hence one has $L^2$ maximal regularity of \eqref{eqn: evolving NS} on a time interval independent of $\mbf{q}$ and $\boldsymbol{\omega}$.

    We now outline how we avoid collisions with the boundary.
    From \eqref{eqn: updated position} we see that
    \[\mbf{Q}(t) = \mbf{q}_0 + t\mbf{v}_0- \frac{1}{m}\int_0^t\int_0^s\int_{\partial B(r)} \mbb{T} \boldsymbol{\nu} \, \mathrm{d}r \, \mathrm{d}s,\]
    where $m$ is the mass of the rigid body.
    As such we find that 
    \[ |\mbf{Q}(t) - \mbf{q}_0| \leq Ct\left(1 + \int_{0}^{\widetilde{T}} \|p\|_{H^1(\Omega(t))} + \|\mbf{u}\|_{\mbf{H}^2(\Omega(t))}\right) \leq C t,\]
    where we have used \eqref{fsi energyestimate2},~\eqref{fsi energyestimate3}, and the previously established bounds on $\mathcal{E}(\widetilde{T})$ and $\mathcal{B}(\widetilde{T})$ for the last inequality.
    Recall that when we pose the problem (see the discussion in Section \ref{section: FSI Intro}) we suppose the existence of some $\delta > 0$ such that
	\[ \mathrm{dist}(B(0), \Gamma) > \delta > 0, \]
	and hence the bound on $|\mbf{Q}(t) - \mbf{q}_0|$ implies that there exists $\widetilde{T}$ depending only on $R$ and the initial data such that
	\[ \mathrm{dist}(\widetilde{B}(t), \Gamma) > \frac{\delta}{2} > 0, \]
    where $\widetilde{B}$ is the rigid body governed by $\mbf{Q} $ and $ \mbf{W}$, for all $t \in [0, \widetilde{T}]$.
    Note that this argument is only an outline of the proof, as one also requires a similar argument applied to the angular velocity, $\mbf{W}$, but we omit further details as it is more or less the same calculation.
\end{proof}

The preceding lemma lets us now unambiguously consider $(\mbf{q}, \boldsymbol{\omega})$ and $(\mbf{Q}, \mbf{W})$ on the same time interval.
We now show that this the new centre of mass, $\mbf{Q}$, and angular momentum, $\mbf{W}$, inherit the same bounds as $\mbf{q}$ and $\boldsymbol{\omega}$.
\begin{lemma}
	\label{lemma: non degeneracy2}
    Let $(\mbf{q}, \boldsymbol{\omega}) \in \mbf{H}^2([0,T]) \times \mbf{H}^1([0,T])$ be compatible with the initial data in the sense of Definition \ref{defn: compatible with IC}, and such that
    \[ \|\mbf{q}\|_{\mbf{H}^2([0,T])} + \|\boldsymbol{\omega}\|_{\mbf{H}^1([0,T])} \leq R, \]
    for some given $R > 0$, determined by the initial data, which is sufficiently large.
    If the relative density $\frac{\rho}{\rho_B}$ is sufficiently small then $\mbf{Q}, \mbf{W}$ are such that
    \[ \|\mbf{Q}\|_{\mbf{H}^2([0,T])} + \|\mbf{W}\|_{\mbf{H}^1([0,T])} \leq R. \]
    Here the smallness of $\frac{\rho}{\rho_B}$ is determined by $\Omega(0)$ and $\mathrm{diam}(B(0))$.
\end{lemma}
\begin{proof}
    From Lemma~\ref{lemma: non degeneracy1} we find that if $(\mbf{q},\boldsymbol{\omega})$, defined on a sufficiently small time interval $[0,T]$, is compatible with the initial data in the sense of Definition \ref{defn: compatible with IC} then $(\mbf{Q},\mbf{W}) := \mathscr{F}(\mbf{q},\boldsymbol{\omega})$ is defined on $[0, T]$ and is compatible with the initial data.
    We wish to show that $(\mbf{Q}, \mbf{W})$ inherits the same bound as $(\mbf{q}, \boldsymbol{\omega})$.
    To this end we observe that
    \[ \int_0^{T} \left|\frac{\mathrm{d}^2 \mbf{Q}}{\mathrm{d}t^2}\right|^2 \leq \frac{1}{m^2} \int_0^{T} \left( \int_{\partial B(t)} \mbb{T} \boldsymbol{\nu} \right)^2 \leq  \frac{C_B}{m^2} \int_0^{T} \|p\|_{H^1(\Omega(t))}^2 + \|\mbf{u}\|_{\mbf{H}^2(\Omega(t))}^2, \]
    where $C_B$ denotes a constant, depending on $\Omega(0)$, appearing due to use of the trace theorem.
    By using \eqref{fsi energyestimate2},~\eqref{fsi energyestimate3}, and the above bounds for $\mathcal{E}(T)$ and $\mathcal{B}(T)$ one finds that
    \begin{align*}
        \int_0^{T} \left|\frac{\mathrm{d}^2 \mbf{Q}}{\mathrm{d}t^2}\right|^2 &\leq \frac{C_B}{m^2} \int_{0}^{T} \|\mbf{B}\|_{\mbf{L}^2(\Omega(t))}^2 + C \int_0^{T}\| \mbf{q}'(t) + \boldsymbol{\omega}(t)\times(\mbf{x} - \mbf{q}(t)) \|_{\mbf{H}^{2}(\Omega(t))}^2\\
        &+C\left(1 + \int_0^{T} \|\widetilde{\mbf{u}}\|_{\mbf{H}^2(\Omega(t))}^2 +  \sup_{t \in [0,T]} \|\grad \widetilde{\mbf{u}}\|_{\mbf{L}^2(\Omega(t))}^2\right)\mathcal{D}(T)^2.
    \end{align*}
    The aim is to show this right-hand side is appropriately bounded so that we may conclude that \(\|\mbf{Q}\|_{\mbf{H}^2([0,T])} \leq \frac{R}{2}\).
    One would hope that we can fix $R$ to be sufficiently large (in terms of the initial data) and $T$ to be sufficiently small (in terms of $R$ and the initial data) so that this is true, however the first term, $\frac{C_B}{m^2} \int_{0}^{T} \|\mbf{B}\|_{\mbf{L}^2(\Omega(t))}^2$, makes this problematic.
    This is because the estimates in Appendix~\ref{appendix:stokes} yield
    \[ \int_{0}^{T} \|\mbf{B}\|_{\mbf{L}^2(\Omega(t))}^2  = \mathcal{O}(\rho^2 R^2),\]
    where we cannot make the constant vanish as $R \rightarrow \infty $ or $T \rightarrow 0$.
    Instead, we notice that the definition of $\mbf{B}$ contains a factor of $\rho$, and we obtain a factor of $\frac{1}{\rho_B}$ from $\frac{1}{m}$ so that we may control this term by only considering a relative density, $\frac{\rho}{\rho_B}$, which is sufficiently small.
    The desired smallness of $\frac{\rho}{\rho_B}$ depends only on $\Omega(0)$ and $\mathrm{diam}(B(0))$.
    In this case we may choose $T$ to be sufficiently small so that
    \[ \frac{C_B}{m^2} \int_{0}^{T} \|\mbf{B}\|_{\mbf{L}^2(\Omega(t))}^2 \leq \widetilde{C_B}\frac{\rho^2 R^2}{\rho_B^2}< \frac{R^2}{8},\]
    where $\widetilde{C_B}$ is a constant depending only on the geometry of $\Omega(0)$.
    One can then take $R$ sufficiently large, and $T$ sufficiently small, so that the other terms are also bounded by $\frac{R^2}{8}$ as well --- we omit further details on these calculations as they are more or less the same as those used above.
    This outlines how one controls the highest order term in $\|\mbf{Q}\|_{\mbf{H}^2([0,T])}$, and the control on the lower order terms follows similarly, so that one concludes that
    \[ \|\mbf{Q}\|_{\mbf{H}^2([0,T])} \leq \frac{R}{2}, \]
    for some fixed sufficiently large $R > 0$, and small $T$ depending on $R$ and the initial data.

    We now consider similar estimates for $\mbf{W}$, which we recall is such that
    \[\frac{\mathrm{d}}{\mathrm{d}t} \left( \mbb{J} \mbf{W} \right) = - \int_{\partial B(t)} (\mbf{x} - \mbf{q}(t)) \times \mbb{T} \boldsymbol{\nu}.\]
    Clearly one has that
    \[ \mbb{J} \frac{\mathrm{d}\mbf{W}}{\mathrm{d}t} = -\frac{\mathrm{d} \mbb{J}}{\mathrm{d}t} \mbf{W} - \int_{\partial B(t)} (\mbf{x} - \mbf{q}(t)) \times \mbb{T} \boldsymbol{\nu}, \]
    where one can use the Reynolds transport theorem to calculate that
    \begin{align*}
        \frac{\mathrm{d}}{\mathrm{d}t}\mbb{J} &= \frac{\mathrm{d}}{\mathrm{d}t}\int_{B(t)} \rho_B \left( |\mbf{x}-\mbf{q}|^2 \mbb{I} - (\mbf{x} - \mbf{q}) \otimes (\mbf{x} - \mbf{q}) \right)\\
		&=\int_{B(t)} \rho_B \left( 2\left(\mbf{x}-\mbf{q}\right)\cdot\left(\mbf{x}-\frac{\mathrm{d}\mbf{q}}{\mathrm{d}t}\right) \mbb{I} - \left(\mbf{x}-\frac{\mathrm{d}\mbf{q}}{\mathrm{d}t}\right) \otimes (\mbf{x} - \mbf{q}) - \left(\mbf{x}-\frac{\mathrm{d}\mbf{q}}{\mathrm{d}t}\right)\otimes \left(\mbf{x}-\mbf{q}\right) \right)\\
		&+ \int_{\partial B(t)} \rho_B \left( |\mbf{x}-\mbf{q}|^2 \mbb{I} - (\mbf{x} - \mbf{q}) \otimes (\mbf{x} - \mbf{q}) \right) \left( \frac{\mathrm{d}\mbf{q}}{\mathrm{d}t} + \boldsymbol{\omega} \times (\mbf{x} - \mbf{q}) \right) \cdot \boldsymbol{\nu}, 
    \end{align*}
    where this final term uses the fact that $\mbf{V}|_{\partial B(t)} =\frac{\mathrm{d}\mbf{q}}{\mathrm{d}t} + \boldsymbol{\omega} \times (\mbf{x} - \mbf{q}) . $
    From this calculation we see that 
    \[ \sup_{t \in [0,T]} \left| \frac{\mathrm{d} \mbb{J}}{\mathrm{d}t} \right| \leq C R, \]
    where we have used that
    \[ \| \mbf{q} \|_{\mbf{C}^1([0,T])} + \| \boldsymbol{\omega} \|_{\mbf{C}^0([0,T])} \leq C (\| \mbf{q} \|_{\mbf{H}^2([0,T])} + \| \boldsymbol{\omega} \|_{\mbf{H}^1([0,T])}) \leq C R,  \]
    for a constant independent of $T$ and $R$.

    Now we also recall from classical mechanics that the inertial tensor, $\mbb{J}$, is positive definite and hence invertible.
    Thus we may write
    \[\frac{\mathrm{d}\mbf{W}}{\mathrm{d}t} = -\mbb{J}^{-1}\frac{\mathrm{d} \mbb{J}}{\mathrm{d}t} \mbf{W} - \mbb{J}^{-1}\int_{\partial B(t)} (\mbf{x} - \mbf{q}(t)) \times \mbb{T} \boldsymbol{\nu},\]
    and since integrating the ODE for $\mbf{W}$ yields
    \[ \mbb{J}(t) \mbf{W}(t) = \mbb{J}(0) \boldsymbol{\omega}_0 - \int_0^t \int_{\partial B(s)} (\mbf{x} - \mbf{q}(s)) \times \mbb{T} \boldsymbol{\nu} \, \mathrm{d}s \]
    we find that
    \[ \frac{\mathrm{d}\mbf{W}}{\mathrm{d}t} = -\mbb{J}^{-1}\frac{\mathrm{d} \mbb{J}}{\mathrm{d}t} \mbb{J}^{-1}\left(\mbb{J}(0) \boldsymbol{\omega}_0 - \int_0^t \int_{\partial B(s)} (\mbf{x} - \mbf{q}(s)) \times \mbb{T} \boldsymbol{\nu} \, \mathrm{d}s\right) - \mbb{J}^{-1}\int_{\partial B(t)} (\mbf{x} - \mbf{q}(t)) \times \mbb{T} \boldsymbol{\nu}. \]

    It is now clear we need to see that $\sup_{t \in [0,T]} |\mbb{J}^{-1}|$ is bounded appropriately.
    For this we observe that the rigid body motion of $B(t)$ implies that each $\mbf{x}(t) \in  B(t)$ can be written as
    \[ \mbf{x}(t) = \mbf{q}(t) + \mbb{O}(t)(\mbf{y} - \mbf{q}_0), \]
    for some $\mbf{y} \in B(0)$, and $\mbb{O}(t)$ is the orthogonal matrix such that
    \[\frac{\mathrm{d}}{\mathrm{d}t} \mbb{O}(t)\mbf{z}_0 = \boldsymbol{\omega}(t) \times \mbb{O}(t) \mbf{z}_0,\]
    for all $\mbf{z}_0 \in \mbb{R}^3$ and $\mbb{O}(0) = \mbb{I}$.
    A change of coordinates now yields that
    \[ \mbb{J}(t) = \int_{B(0)} \rho_B \left(|\mbb{O}(t)(\mbf{y}-\mbf{q}_0)|^2 \mbb{I} - \mbb{O}(t)(\mbf{y}-\mbf{q}_0) \otimes \mbb{O}(t)(\mbf{y}-\mbf{q}_0)\right), \]
    from which one finds that
    \[ \mbb{J}(t) = \mbb{O}(t) \mbb{J}(0) \mbb{O}(t)^T. \]
    Thus one finds that for any eigenvector, $\mbf{z}$, of $\mbb{J}(0)$ we have that $\mbb{O}(t)\mbf{z}$ is an eigenvector $\mbb{J}(t)$ with the same eigenvalue.
    Hence the largest eigenvalue of $\mbb{J}(t)^{-1}$ is bounded independent of $t$ and $R$, and hence norm equivalence readily translates this into a bound for $\sup_{t \in [0,T]}|\mbb{J}^{-1}(t)|$.

    As we saw for the centre of mass, $\mbf{Q}$, we cannot immediately show the bound for $\mbf{W}$ by taking $R$ large and $T$ small.
    In this case the problematic term is now $\mbb{J}^{-1}\int_{\partial B(t)} (\mbf{x} - \mbf{q}(t)) \times \mbb{T} \boldsymbol{\nu}$, which we control by only allowing sufficiently small relative densities --- as we did in showing the bound for $\mbf{Q}$.
    In this case, we observe that the entries of $\mbb{J}$ depend linearly on $\rho_B$ so that the entries of $\mbb{J}^{-1}$ depend linearly on $\frac{1}{\rho_B}$.
    Hence we only allow a relative density, $\frac{\rho}{\rho_B}$, which is sufficiently small so that one finds
    \[\int_0^{T}\left|\mbb{J}^{-1}\int_{\partial B(t)} (\mbf{x} - \mbf{q}(t)) \times \mbb{T} \boldsymbol{\nu}\right|^2 \leq C_B^2\frac{\rho^2 R^2}{\rho_B^2} \leq \frac{R^2}{8},\]
    as we saw above, where again $C_B$ is a constant arising from the use of the trace theorem.
    We omit further details on this calculation.
    
    Thus by arguing along the same lines as we did for the bound on $\frac{\mathrm{d}^2 \mbf{Q}}{\mathrm{d}t^2}$ we find that \( \int_0^{T} \left| \frac{\mathrm{d} \mbf{W}}{\mathrm{d}t} \right|^2 \leq \frac{R^2}{4}, \)
    and the bound for the lower order term follows, from which we conclude that by taking $T$ sufficiently small one has
    \[ \|\mbf{W}\|_{\mbf{H}^1([0,T])} \leq \frac{R}{2}, \]
    for some sufficiently small $R > 0$.
    Thus we have shown that
    \[ \|\mbf{q}\|_{\mbf{H}^2([0,T])} + \|\boldsymbol{\omega}\|_{\mbf{H}^1([0,T])} \leq R \Rightarrow \|\mbf{Q}\|_{\mbf{H}^2([0,T])} + \|\mbf{W}\|_{\mbf{H}^1([0,T])} \leq R,\]
    provided that $\frac{\rho}{\rho_B}$ is sufficiently small.
\end{proof}
These lemmas show that $\mathscr{F}$ maps some small ball in $\mbf{H}^2([0,T]) \times \mbf{H}^1([0,T])$ onto itself, and hence we may consider the use of fixed point arguments.

\subsection{Some preliminary calculations}
\label{subsection: fixed point prelims}
Owing to Lemma~\ref{lemma: non degeneracy1} and Lemma~\ref{lemma: non degeneracy2} one finds that if the relative density, $\frac{\rho}{\rho_B}$, is sufficiently small we may uncontroversially consider $\mathscr{F}$ as a map from some suitably defined subset of $\mbf{H}^2([0,T]) \times \mbf{H}^1([0,T])$ onto itself.
As such we define the domain of $\mathscr{F}$ to be given by
\begin{gather*}
   D(\mathscr{F}) :=\left\{ (\mbf{q}, \boldsymbol{\omega}) \in B_R(\mbf{H}^2 \times \mbf{H}^1) \mid (\mbf{q}, \boldsymbol{\omega}) \text{ are compatible in the sense of Definition}~\ref{defn: compatible with IC}\right\},\\
   B_R(\mbf{H}^2 \times \mbf{H}^1) := \left\{(\mbf{q}, \boldsymbol{\omega}) \in \mbf{H}^2([0,T]) \times \mbf{H}^1([0,T]) \mid  \|\mbf{q}\|_{\mbf{H}^2([0,T])} + \|\boldsymbol{\omega}\|_{\mbf{H}^1([0,T])} \leq R \right\}.
\end{gather*}
The plan for proving Theorem \ref{thm: FSI existence} is to show that $\mathscr{F}$ is a contraction when $T$ is sufficiently small.
It is straightforward to verify that the domain $D(\mathscr{F})$ is a complete metric space, and so the use of the Banach fixed point theorem is indeed justified.
In order to show that $\mathscr{F}$ is a contraction we firstly require some preliminary results.
From here on we shall consider two pairs, $(\mbf{q}^i, \boldsymbol{\omega}^i) \in D(\mathscr{F})$, such that $\mathscr{F}(\mbf{q}^i, \boldsymbol{\omega}^i) =: (\mbf{Q}^i, \mbf{W}^i)$ for $i=1,2$.
We denote the associated velocity field as $\mbf{V}^i$, the associated parametrisation as $\Phi^i$, and the associated strong solution of the Navier--Stokes equations on $\Omega^i(t)$ as $(\mbf{u}^i, p^i) \in H^1_{\mbf{L}^2}(\Phi^i) \cap {L}^2_{\mbf{H}^2}(\Phi^i) \times L^2_{{H}^1}(\Phi^i)$.
It is important to note that $\Omega^1(t) \neq \Omega^2(t)$ for $t \neq 0$ --- that is to say $(\mbf{u}^1, p^1)$ and $(\mbf{u}^2, p^2)$ are defined on different evolving domains.
This will be the main difficulty for our subsequent analysis.

We begin with the following lemma relating to control of $\grad \Phi^i$ to the difference of the pairs $(\mbf{q}^i, \boldsymbol{\omega}^i)$.
\begin{lemma}
\label{lemma: gradient difference}
    For $(\mbf{q}^i, \boldsymbol{\omega}^i) \in D(\mathscr{F})$, for $i=1,2$, one has that
	\begin{align*}
    &\sup_{t \in [0,T]} \|\grad \Phi^1(t) - \grad \Phi^2(t)\|_{\mbf{C}^{1,1}(\Omega(0))}\\
    &\leq C \int_0^T \left\| \frac{\mathrm{d}(\mbf{q}^1-\mbf{q}^2)}{\mathrm{d}t} + (\boldsymbol{\omega}^1 - \boldsymbol{\omega}^2)\times \mbf{x} - \boldsymbol{\omega}^1 \times (\mbf{q}^1 - \mbf{q}^2) - (\boldsymbol{\omega}^1 - \boldsymbol{\omega}^2)\times \mbf{q}^2  \right\|_{\mbf{C}^{2,1}(\Omega(0))},
    \end{align*}
    for a constant $C$ independent of $\mbf{q}^1$, $\mbf{q}^2$, $\boldsymbol{\omega}^1$ and $\boldsymbol{\omega}^2$.
\end{lemma}
\begin{proof}
	To begin, we observe that
	\begin{gather*}
		\frac{\mathrm{d}}{\mathrm{d}t}(\grad \Phi^1(\mbf{x},t) - \grad \Phi^2(\mbf{x},t)) = \grad\mbf{V}^1(\Phi^1(\mbf{x},t),t) \grad \Phi^1(\mbf{x},t) - \grad\mbf{V}^2(\Phi^2(\mbf{x},t),t) \grad \Phi^2(\mbf{x},t),
	\end{gather*}
	with initial condition $\grad \Phi^1(\mbf{x},0) - \grad \Phi^2(\mbf{x},0) = 0$.
	One rewrites this ODE as
	\begin{align*}
		\frac{\mathrm{d}}{\mathrm{d}t}(\grad \Phi^1(\mbf{x},t) - \grad \Phi^2(\mbf{x},t)) &= (\grad\mbf{V}^1(\Phi^1(\mbf{x},t),t) - \grad\mbf{V}^2(\Phi^2(\mbf{x},t),t)) \grad \Phi^1(\mbf{x},t)\\
		&+ \grad\mbf{V}^2(\Phi^2(\mbf{x},t),t) (\grad \Phi^1(\mbf{x},t) -  \grad \Phi^2(\mbf{x},t)),
	\end{align*}
	from which one can readily observe that
	\begin{equation}
    \begin{aligned}
    &\|\grad \Phi^1(t) - \grad \Phi^2(t)\|_{\mbf{C}^{1,1}(\Omega(0))}\\
        &\leq \int_0^t \|\grad\mbf{V}^2(\Phi^2(s),s)\|_{\mbf{C}^{1,1}(\Omega(0))} \|\grad \Phi^1(s) -  \grad \Phi^2(s)\|_{\mbf{C}^{1,1}(\Omega(0))} \, \mathrm{d}s\\
		&+\int_0^t \|(\grad\mbf{V}^1(\Phi^1(s),s) - \grad\mbf{V}^2(\Phi^2(s),s))\|_{\mbf{C}^{1,1}(\Omega(0))} \|\grad \Phi^1(s)\|_{\mbf{C}^{1,1}(\Omega(0))} \, \mathrm{d}s.
        \end{aligned}\label{jacobian pf1}
	\end{equation}
	It is now clear that using Gr\"onwall's inequality yields the result, provided that one has an appropriate bound on $\|(\grad\mbf{V}^1(\Phi^1(s),s) - \grad\mbf{V}^2(\Phi^2(s),s))\|_{\mbf{C}^{1,1}(\Omega(0))}$.
	To this end we recall \eqref{eqn: velocity harmonic extension} defining $\mbf{V}^i$.
	By pulling back to $\Omega(0)$ and using Lemma \ref{pullback lemma} we find that the functions $\widehat{\mbf{V}^i}(t) := \mbf{V}^i(t)\circ \Phi^i(t)$ solve
	\begin{gather*}
		\begin{cases}
			-\Div \left(\det(\grad\Phi^i(t)) \grad\Phi^i(t) (\grad\Phi^i(t))^{T} \grad \widehat{\mbf{V}^i}(t) \right) = 0,&\text{ on }\ \Omega(0),\\
			\widehat{\mbf{V}^i}(\mbf{x},t) = \frac{\mathrm{d}\mbf{q}^i}{\mathrm{d}t} + \boldsymbol{\omega}^i \times (\mbf{x} - \mbf{q}^i), &\text{ on } \partial B(0),\\
			\widehat{\mbf{V}^i}(\mbf{x},t) = 0, &\text{ on } \Gamma,
		\end{cases}
	\end{gather*}
	respectively in a weak sense, for $i = 1,2$.
	In order to study the difference of these functions we now introduce an intermediate function $\widetilde{\mbf{V}^{1,2}}$ which is the unique solution to the PDE
	\begin{gather*}
	\begin{cases}
		-\Div \left(\det(\grad\Phi^2(t)) \grad\Phi^2(t) (\grad\Phi^2(t))^{T} \grad \widetilde{\mbf{V}^{1,2}}(t) \right) = 0,& \text{ on }\ \Omega(0),\\
		\widetilde{\mbf{V}^{1,2}}(\mbf{x},t) = \frac{\mathrm{d}\mbf{q}^1}{\mathrm{d}t} + \boldsymbol{\omega}^1 \times (\mbf{x} - \mbf{q}^1),& \text{ on } \partial B(0),\\
		\widetilde{\mbf{V}^{1,2}}(\mbf{x},t) = 0,& \text{ on } \Gamma.
	\end{cases}
	\end{gather*}
	Now one has
	\begin{align*}\left\|\widehat{\mbf{V}^1}(t) - \widehat{\mbf{V}^2}(t) \right\|_{\mbf{C}^{2,1}(\Omega(0))} \leq \left\|\widetilde{\mbf{V}^{1,2}}(t) - \widehat{\mbf{V}^2}(t) \right\|_{\mbf{C}^{2,1}(\Omega(0))} + \left\|\widehat{\mbf{V}^1}(t) - \widetilde{\mbf{V}^{1,2}}(t) \right\|_{\mbf{C}^{2,1}(\Omega(0))},
	\end{align*}
	and it is straightforward to verify, using standard elliptic regularity results \cite{gilbarg1977elliptic}, that
	\begin{align*}
		&\left\|\widetilde{\mbf{V}^{1,2}}(t) - \widehat{\mbf{V}^2}(t) \right\|_{\mbf{C}^{2,1}(\Omega(0))}\\
        &\leq C \left\| \frac{\mathrm{d}(\mbf{q}^1-\mbf{q}^2)}{\mathrm{d}t} + (\boldsymbol{\omega}^1 - \boldsymbol{\omega}^2)\times \mbf{x} - \boldsymbol{\omega}^1 \times (\mbf{q}^1 - \mbf{q}^2) - (\boldsymbol{\omega}^1 - \boldsymbol{\omega}^2)\times \mbf{q}^2  \right\|_{\mbf{C}^{2,1}(\Omega(0))},
	\end{align*}
	and
	\begin{align*}
		&\left\|\widehat{\mbf{V}^1}(t) - \widetilde{\mbf{V}^{1,2}}(t) \right\|_{\mbf{C}^{2,1}(\Omega(0))}\\
		&\leq C \left\|\det(\grad\Phi^2(t)) \grad\Phi^2(t) (\grad\Phi^2(t))^{T} - \det(\grad\Phi^1(t)) \grad\Phi^1(t)(\grad\Phi^1(t))^{T} \right\|_{\mbf{C}^{1,1}(\Omega(0))}\\
		&\leq C \left\| \grad\Phi^2(t) - \grad\Phi^1(t) \right\|_{\mbf{C}^{1,1}(\Omega(0))},
	\end{align*}
	where the final inequality is a consequence of the fact that $(\mbf{q}^i, \boldsymbol{\omega}^i) \in B_R(\mbf{H}^2 \times \mbf{H}^1)$ for $i = 1,2$.
	Now by noting that
	\begin{align*}
		&\left\|\grad\mbf{V}^1(\Phi^1(s),s) - \grad\mbf{V}^2(\Phi^2(s),s) \right\|_{\mbf{C}^{1,1}(\Omega(0))} \leq C\left\|\widehat{\mbf{V}^1}(s) - \widehat{\mbf{V}^2}(s)  \right\|_{\mbf{C}^{2,1}(\Omega(0))}\\
		&\leq C \left\| \frac{\mathrm{d}(\mbf{q}^1-\mbf{q}^2)}{\mathrm{d}t} + (\boldsymbol{\omega}^1 - \boldsymbol{\omega}^2)\times \mbf{x} - \boldsymbol{\omega}^1 \times (\mbf{q}^1 - \mbf{q}^2) - (\boldsymbol{\omega}^1 - \boldsymbol{\omega}^2)\times \mbf{q}^2  \right\|_{\mbf{C}^{2,1}(\Omega(0))}\\
        &+ C \left\| \grad\Phi^2(t) - \grad\Phi^1(t) \right\|_{\mbf{C}^{1,1}(\Omega(0))},
	\end{align*}
	it is clear that the result follows by using Gr\"onwall's inequality.
\end{proof}

One may notice that our use of the Gr\"onwall inequality means that the constant appearing in the previous lemma depends on $T$.
However, since we will be interested in taking $T$ to be sufficiently small we may ignore this issue and assume without loss of generality that the constant is independent of $T$.
Next we show a lemma regarding the control on the time derivative of $\Phi^1 - \Phi^2$.

\begin{lemma}
\label{lemma: gradient difference ddt}
For $(\mbf{q}^i, \boldsymbol{\omega}^i) \in D(\mathscr{F})$ for $i=1,2$ one has that
	\begin{align*}
    &\sup_{t \in [0,T]}\left\|\frac{\mathrm{d}}{\mathrm{d}t}(\Phi^1 - \Phi^2) \right\|_{\mbf{C}^{2,1}(\Omega(0))}\\ 
	&\leq C \sup_{t \in [0,T]}\left\| \frac{\mathrm{d}(\mbf{q}^1-\mbf{q}^2)}{\mathrm{d}t} + (\boldsymbol{\omega}^1 - \boldsymbol{\omega}^2)\times \mbf{x} - \boldsymbol{\omega}^1 \times (\mbf{q}^1 - \mbf{q}^2) - (\boldsymbol{\omega}^1 - \boldsymbol{\omega}^2)\times \mbf{q}^2  \right\|_{\mbf{C}^{2,1}(\Omega(0))}\\
    &+ C \int_0^T \left\| \frac{\mathrm{d}(\mbf{q}^1-\mbf{q}^2)}{\mathrm{d}t} + (\boldsymbol{\omega}^1 - \boldsymbol{\omega}^2)\times \mbf{x} - \boldsymbol{\omega}^1 \times (\mbf{q}^1 - \mbf{q}^2) - (\boldsymbol{\omega}^1 - \boldsymbol{\omega}^2)\times \mbf{q}^2  \right\|_{\mbf{C}^{2,1}(\Omega(0))},
	\end{align*}
    for a constant $C$ independent of $\mbf{q}^1$, $\mbf{q}^2$, $\boldsymbol{\omega}^1$ and $\boldsymbol{\omega}^2$.
\end{lemma}
\begin{proof}
	By definition of $\Phi^i$ one immediately finds that
    \[ \frac{\mathrm{d}}{\mathrm{d}t}(\Phi^1 - \Phi^2) = \mbf{V}^1(\Phi^1(t),t) - \mbf{V}^2(\Phi^2(t),t), \]
    and the result thus follows as a consequence of the calculations in the proof of Lemma \ref{lemma: gradient difference}.
\end{proof}

\subsubsection{Pullbacks and the Piola transform}
Before using Lemma \ref{lemma: gradient difference} and Lemma \ref{lemma: gradient difference ddt} to show that $\mathscr{F}$ is a contraction we need a way to preserve the divergence-free property of $\mbf{u}^i$ under some coordinate transformation.
This will help us to control the pressure terms which arise in the calculations to show that $\mathscr{F}$ is a contraction.
To do this we use the Piola transform, which is often used in finite element approximations of equations of continuum mechanics (see \cite{aznaran2022transformations,ern2021finite,rognes2010efficient}), and more recently in the analysis of fluid equations on evolving domains (see \cite{djurdjevac2023evolving,elliott2024navier,olshanskii2022tangential}).
We begin by defining maps $\Phi^{2,1}(t) : \Omega^2(t) \rightarrow \Omega^1(t)$ given by
\[ \Phi^{2,1}(t) := \Phi^1(t) \circ \Phi^2(t)^{-1} .\]
We shall also denote the inverse of $\Phi^{2,1}(t)$ as
\[ \Phi^{1,2}(t) = \Phi^2(t) \circ \Phi^1(t)^{-1} : \Omega^1(t) \rightarrow \Omega^2(t). \]
We consider the associated Piola transform to be defined as
\[\mathcal{P}_t^{1,2} \boldsymbol{\psi}(\mbf{x}) = {\det(\grad\Phi^{1,2}(t)(\mbf{x}))} (\grad\Phi^{2,1}(t) \boldsymbol{\psi} )\circ \Phi^{1,2}(t) \]
for a function $\boldsymbol{\psi}: \Omega^2(t) \rightarrow \mbb{R}^3$.
A useful relation that we state for later use is that
\begin{align}
    \boldsymbol{\psi} \circ \Phi^{1,2}(t) = \frac{1}{\det(\grad\Phi^{1,2})} (\grad\Phi^{1,2}) \mathcal{P}_t^{1,2}\boldsymbol{\psi}. \label{eqn: piola pullback relation}
\end{align}
The Piola transform defined above maps functions in $\mbf{W}^{k,p}(\Omega^2(t))$ onto functions in $\mbf{W}^{k,p}(\Omega^1(t))$, for a result in this direction we refer the reader to \cite[Appendix A]{djurdjevac2023evolving}.
One can define an inverse Piola transform in the same manner, but we do not consider this here.
For further details on the Piola transform, as well as its applications in mathematical elasticity,  we refer the reader to \cite{ciarlet2021mathematical}.
We now recall the following property of the Piola transform which will be useful in our fixed point argument.

\begin{lemma}[{\cite[Lemma 2.8]{djurdjevac2023evolving}}]
	For $t \in [0,T]$, let $\boldsymbol{\psi} \in C^1(\Omega^2(t); \mbb{R}^3)$.
	Then we have
	\begin{align}
		\Div \mathcal{P}^{1,2}_t \boldsymbol{\psi} = \det(\grad\Phi^{1,2}(t))(\Div \boldsymbol{\psi}) \circ \Phi^{1,2}(t). \label{piola divergence}
	\end{align}
    In particular the Piola transform preserves divergence-free functions.
\end{lemma}
Using the established bounds on $\Phi^i$ one can readily show the compatibility of the evolving function spaces on $\Omega^2(t)$ with the Piola transform in the sense of \cite{alphonse2023function,alphonse2015abstract} by adapting the arguments in Lemma \ref{pullback lemma}.
With these preliminaries we now prove $\mathscr{F}$ is a contraction.

\subsection{The fixed point argument}
\label{subsection: fixed point}
    As we have done throughout we assume that~$(\mbf{q}^i, \boldsymbol{\omega}^i) \in D(\mathscr{F})$, for $i = 1,2$, and we write~$\mathscr{F}(\mbf{q}^i, \boldsymbol{\omega}^i) = (\mbf{Q}^i, \mbf{W}^i)$.
    Our aim is to show that
    \[ \|\mbf{Q}^1 - \mbf{Q}^2\|_{\mbf{H}^2([0,T])} +  \|\mbf{W}^1 - \mbf{W}^2\|_{\mbf{H}^1([0,T])} \leq L \left(\|\mbf{q}^1 - \mbf{q}^2\|_{\mbf{H}^2([0,T])} +  \|\boldsymbol{\omega}^1 - \boldsymbol{\omega}^2\|_{\mbf{H}^1([0,T])}\right),  \]
    for a constant $0 < L < 1$.
    To begin we observe that $\mbf{Q}^1 - \mbf{Q}^2$ solves the ODE
    \[ m \frac{\mathrm{d}^2}{\mathrm{d}t^2} (\mbf{Q}^1 - \mbf{Q}^2) = \int_{\partial B^2(t)} \mbb{T}^2 \boldsymbol{\nu}^2 - \int_{\partial B^1(t)} \mbb{T}^1 \boldsymbol{\nu}^1, \]
    pointwise for almost all $t \in [0,T]$ and such that
    \[ (\mbf{Q}^1 - \mbf{Q}^2)(0) = 0, \qquad (\mbf{Q}^1 - \mbf{Q}^2)'(0) = 0. \]
    By pulling the first integral back onto $\partial B^1(t)$ one finds that
    \begin{align}
    m \frac{\mathrm{d}^2}{\mathrm{d}t^2} (\mbf{Q}^1 - \mbf{Q}^2) = \int_{\partial B^1(t)} \left( \det(\grad \Phi^{1,2})\widehat{\mbb{T}^2} - {\mbb{T}^1} \right) \boldsymbol{\nu}^1, \label{contraction pf1}
    \end{align}
    where $\widehat{\mbb{T}^2} = \mbb{T}^2 \circ \Phi^{1,2}(t)$.
    Using Lemma \ref{transformation lemma} and \eqref{eqn: piola pullback relation} one finds that
    \begin{multline*}
        \int_{\partial B^1(t)} \det(\grad \Phi^{1,2})\widehat{\mbb{T}^2} \boldsymbol{\nu}^1= -\int_{\partial B^1(t)} \det(\grad \Phi^{1,2}) \widehat{p^2} \boldsymbol{\nu}^1\\
        \begin{aligned}
        &+ \mu \int_{\partial B^1(t)} \left( (\grad \Phi^{1,2})^{-T} \grad \Phi^{1,2} \grad \widehat{\mbf{u}^2} +   (\grad \widehat{\mbf{u}^2})^T(\grad \Phi^{1,2})^{T} (\grad \Phi^{1,2})^{-1}  \right) \boldsymbol{\nu}^1\\
        &+\mu \int_{\partial B^1(t)}\det(\grad\Phi^{1,2})(\grad \Phi^{1,2})^{-T}\grad \left( \frac{1}{\det(\grad\Phi^{1,2})} \grad \Phi^{1,2} \right)\widehat{\mbf{u}^2} \boldsymbol{\nu}^1\\
        &+\mu \int_{\partial B^1(t)}\det(\grad\Phi^{1,2})\left((\grad \Phi^{1,2})^{-T}\grad \left( \frac{1}{\det(\grad\Phi^{1,2})} \grad \Phi^{1,2} \right)\widehat{\mbf{u}^2} \right)^T\boldsymbol{\nu}^1,
        \end{aligned}
    \end{multline*}
    where
    \[ \widehat{p^2} := p^2 \circ \Phi^{1,2}, \qquad \widehat{\mbf{u}^2} := \mathcal{P}_t^{1,2} \mbf{u}^2. \]

    We now use this expression in \eqref{contraction pf1} to see that
    \begin{align}
        m \frac{\mathrm{d}^2}{\mathrm{d}t^2}(\mbf{Q}^1 - \mbf{Q}^2) &= \int_{\partial B^1(t)} (p^1 - \widehat{p^2}) \boldsymbol{\nu}^1 + \mu \int_{\partial B^1(t)} \grad (\widehat{\mbf{u}^2} -\mbf{u}^1) + (\grad (\widehat{\mbf{u}^2} -\mbf{u}^1))^T \notag\\
        &+\int_{\partial B^1(t)} (1-\det(\grad \Phi^{1,2})) \widehat{p^2} \boldsymbol{\nu}^1\notag\\
        &+ \mu \int_{\partial B^1(t)} \left((\grad \Phi^{1,2})^{-T} \grad \Phi^{1,2} - \mbb{I}\right) \grad \widehat{\mbf{u}^2} \boldsymbol{\nu}^1\notag\\
        &+ \mu \int_{\partial B^1(t)} (\grad \widehat{\mbf{u}^2})^T\left((\grad \Phi^{1,2})^{T} (\grad \Phi^{1,2})^{-1} - \mbb{I} \right) \boldsymbol{\nu}^1 \label{contraction pf2}\\
        &+\mu \int_{\partial B^1(t)}\det(\grad\Phi^{1,2})(\grad \Phi^{1,2})^{-T}\grad \left( \frac{1}{\det(\grad\Phi^{1,2})} \grad \Phi^{1,2}\right)\widehat{\mbf{u}^2} \boldsymbol{\nu}^1\notag\\
        &+\mu \int_{\partial B^1(t)}\det(\grad\Phi^{1,2})\left((\grad \Phi^{1,2})^{-T}\grad \left( \frac{1}{\det(\grad\Phi^{1,2})} \grad \Phi^{1,2} \right)\widehat{\mbf{u}^2} \right)^T\boldsymbol{\nu}^1. \notag
    \end{align}
    The plan now is to bound the first two terms in \eqref{contraction pf2} by establishing bounds analogous to those in Lemma~\ref{lemma: fsi energy estimates} and Lemma~\ref{lemma: fsi energy estimates2}, and bounding the other terms by using the uniform bounds on $\widehat{p^2}, \widehat{\mbf{u}^2}$ in combination with Lemma \ref{lemma: gradient difference}.
    The issue now is that we must emulate the proof of Lemma \ref{lemma: fsi energy estimates} to establish appropriate bounds for $p^1 - \widehat{p^2}$ and $\mbf{u}^1 - \widehat{\mbf{u}^2}$ in $L^2_{H^1}(\Phi^1)$ and $L^2_{\mbf{H}^2}(\Phi^1)$ respectively.
    For this we firstly must find the equations which $\widehat{\mbf{u}^2}, \widehat{p^2}$ solve, which we do by using Lemma \ref{transformation lemma} and \eqref{eqn: piola pullback relation} again.

    As before, we write $\mbf{u}^i = \overline{\mbf{u}}^i + \widetilde{\mbf{u}}^i$ and $p^i = \overline{p}^i + \widetilde{p}^i$, where $(\overline{\mbf{u}}^i,\overline{p}^i)$ solves a system like \eqref{eqn: weakNS}, \eqref{eqn: weakNS2} with homogeneous boundary conditions, and $\widetilde{\mbf{u}}^i$ solves a Stokes problem with inhomogeneous boundary conditions like in Appendix \ref{appendix:stokes}.
    The difficulty now comes from bounding the $\overline{\mbf{u}}^i, \overline{p}^i$ as in Lemma \ref{lemma: fsi energy estimates}, and this shall be the majority of the proof.

    \begin{remark}
        Our use of the Piola transform is due to the fact that recreate the proof of Lemma~\ref{lemma: fsi energy estimates} we wish to use \eqref{leray laplacian inequality} to control the second order terms and this is only valid for divergence-free functions.
        We note that this is also where our higher regularity assumption on the boundary of $\partial \Omega(0)$ comes into play as $\Phi^{1,2}$ must be sufficiently smooth so that
        \[ \mathcal{P}_t^{1,2} : \mbf{W}^{2,q}(\Omega^2(t)) \rightarrow \mbf{W}^{2,q}(\Omega^1(t)), \quad q \in [1,\infty]. \]
        If one can remove the use of the Piola transform then we imagine our approach would apply to domains with a $C^{1,1}$ boundary.
    \end{remark}
    
    It is a routine calculation to verify that the functions $(\utwo, \ptwo)$ defined by
    \[\utwo := \mathcal{P}_t^{1,2} \overline{\mbf{u}}^2, \qquad \ptwo := \overline{p}^2 \circ \Phi^{1,2}(t)  ,\]
    solve
    \begin{align}
        &\rho \int_{\Omega^1(t)} \grad \Phi^{1,2} \ddt{\utwo} \cdot \boldsymbol{\eta} + \rho \int_{\Omega^1(t)} \det(\grad \Phi^{1,2}) \ddt{}\left( \frac{1}{\det(\grad \Phi^{1,2})} \grad \Phi^{1,2}\right) \utwo \cdot \boldsymbol{\eta} \notag\\
        &- \rho\int_{\Omega^1(t)}\left((\grad \Phi^{1,2})^{-1}\ddt{\Phi^{1,2}} \cdot \grad \Phi^{1,2} \grad \utwo\right) \boldsymbol{\eta}\notag\\
        &- \rho\int_{\Omega^1(t)}\det(\grad \Phi^{1,2})\left((\grad \Phi^{1,2})^{-1}\ddt{\Phi^{1,2}} \cdot \grad \left( \frac{1}{\det(\grad \Phi^{1,2})} \grad \Phi^{1,2} \right) \utwo \right)\cdot \boldsymbol{\eta}\notag\\
        &+ \rho \int_{\Omega^1(t)} \frac{1}{\det(\grad \Phi^{1,2})} (\grad \Phi^{1,2} )^T \utwo \cdot \grad \utwo \boldsymbol{\eta}\notag\\
        &+ \rho \int_{\Omega^1(t)} \utwo \cdot \grad \left( \frac{1}{\det(\grad \Phi^{1,2})} \grad \Phi^{1,2} \right) \utwo \cdot \boldsymbol{\eta} + \int_{\Omega^1(t)} \det(\grad \Phi^{1,2}) (\grad \Phi^{1,2})^{-T}\grad\ptwo \cdot \boldsymbol{\eta} \label{contraction pf3}\\
        &+ \mu\int_{\Omega^1(t)} (\grad \Phi^{1,2})^{-1} (\grad \Phi^{1,2})^{-T} \grad \Phi^{1,2} \grad \utwo : \grad \boldsymbol{\eta}\notag\\
        &+ \mu \int_{\Omega^1(t)} \det(\grad\Phi^{1,2}) (\grad \Phi^{1,2})^{-1} (\grad \Phi^{1,2})^{-T} \grad \left(\frac{1}{\det(\grad \Phi^{1,2})} \grad \Phi^{1,2} \right) \utwo : \grad \boldsymbol{\eta}\notag\\
        &+ \rho \int_{\Omega^1(t)} ((\grad \Phi^{1,2})^{T}(\grad \Phi^{1,2})^{-1}\widetilde{\mbf{u}}^2 \cdot \grad \utwo) \boldsymbol{\eta}\notag\\
        &+ \rho \int_{\Omega^1(t)} \det(\grad \Phi^{1,2}) \left((\grad \Phi^{1,2})^{-1} \widetilde{\mbf{u}}^2 \cdot \grad \left(\frac{1}{\det(\grad \Phi^{1,2})} \grad \Phi^{1,2}\right) \utwo \right)\boldsymbol{\eta}\notag\\
        &+ \rho \int_{\Omega^1(t)} \utwo \cdot \grad \widetilde{\mbf{u}}^2 \boldsymbol{\eta} + \int_{\Omega^1(t)} \det(\grad \Phi^{1,2})\widehat{\mbf{B}^2} \cdot \boldsymbol{\eta} = 0, \notag
    \end{align}
    \begin{align}
        \int_{\Omega^1(t)} q \Div \utwo = 0, \label{contraction pf4}
    \end{align}
    for all $q \in L^2(\Omega^1(t)), \boldsymbol{\eta} \in \mbf{H}_0^1(\Omega^1(t))$ and almost all $t \in [0,T]$.
   Here we are also using notation
   \[ \widehat{\widetilde{\mbf{u}}^2} := \widetilde{\mbf{u}}^2 \circ \Phi^{1,2}(t), \qquad \widehat{\mbf{B}^2} := \mbf{B}^2 \circ \Phi^{1,2}(t). \]
   Thus we may combine \eqref{eqn: weakNS}, \eqref{eqn: weakNS2} for $(\mbf{u}^1, p^1)$ with \eqref{contraction pf3}, \eqref{contraction pf4} to see that
   \begin{align}
       &\rho \int_{\Omega^1(t)} \ddt{(\overline{\mbf{u}}^1 - \utwo)} \cdot \boldsymbol{\eta} + \rho \int_{\Omega^1(t)} \overline{\mbf{u}}^1 \cdot \grad\overline{\mbf{u}}^1 \boldsymbol{\eta} - \rho \int_{\Omega^1(t)} \utwo \cdot \grad \utwo \boldsymbol{\eta} \notag\\
       & + \int_{\Omega^1(t)} \grad(\overline{p}^1 - \ptwo) \cdot \boldsymbol{\eta} + \mu \int_{\Omega^1(t)} \grad (\overline{\mbf{u}}^1 - \utwo): \grad \boldsymbol{\eta} \label{contraction pf5}\\
       & + \rho \int_{\Omega^1(t)}  \left((\widetilde{\mbf{u}}^1 \cdot \grad \overline{\mbf{u}}^1) + (\overline{\mbf{u}}^1 \cdot \grad\widetilde{\mbf{u}}^1) \right)\boldsymbol{\eta} - \rho \int_{\Omega^1(t)}  \left((\widehat{\widetilde{\mbf{u}}^2} \cdot \grad \utwo) + (\utwo \cdot \grad\widehat{\widetilde{\mbf{u}}^2}) \right)\boldsymbol{\eta} \notag \\
       &+ \int_{\Omega^1(t)} (\mbf{B}^1 - \widehat{\mbf{B}^2}) \cdot \boldsymbol{\eta}= \sum_{k=1}^{12} I_k[\boldsymbol{\eta}], \notag
   \end{align}
   \begin{align}
   \int_{\Omega^1(t)} q \Div (\overline{\mbf{u}}^1 - \utwo) = 0, \label{contraction pf6}
   \end{align}
   for all $q \in L^2(\Omega^1(t)), \boldsymbol{\eta} \in \mbf{H}_0^1(\Omega^1(t))$ and almost all $t \in [0,T]$.
   Here the $I_k[\boldsymbol{\eta}]$ are some linear functionals to be bounded, given by:
   \begin{align*}
       I_1[\boldsymbol{\eta}] &:= \rho \int_{\Omega^1(t)} (\grad \Phi^{1,2} - \mbb{I}) \ddt{\utwo} \cdot \boldsymbol{\eta},\\
       I_2[\boldsymbol{\eta}] &:= \rho \int_{\Omega^1(t)} \det(\grad \Phi^{1,2}) \ddt{}\left( \frac{1}{\det(\grad \Phi^{1,2})} \grad \Phi^{1,2}\right) \utwo \cdot \boldsymbol{\eta} ,\\
       I_3[\boldsymbol{\eta}] &:= - \rho\int_{\Omega^1(t)}\left((\grad \Phi^{1,2})^{-1}\ddt{\Phi^{1,2}} \cdot \grad \Phi^{1,2} \grad \utwo\right) \boldsymbol{\eta},\\
       I_4[\boldsymbol{\eta}] &:= - \rho\int_{\Omega^1(t)}\det(\grad \Phi^{1,2})\left((\grad \Phi^{1,2})^{-1}\ddt{\Phi^{1,2}} \cdot \grad \left( \frac{1}{\det(\grad \Phi^{1,2})} \grad \Phi^{1,2} \right) \utwo \right)\cdot \boldsymbol{\eta},\\
       I_5[\boldsymbol{\eta}] &:= \rho \int_{\Omega^1(t)} \left(\frac{1}{\det(\grad \Phi^{1,2})} (\grad \Phi^{1,2} )^T - \mbb{I}\right)\utwo \cdot \grad \utwo \boldsymbol{\eta} ,\\
       I_6[\boldsymbol{\eta}] &:= \rho \int_{\Omega^1(t)} \utwo \cdot \grad \left( \frac{1}{\det(\grad \Phi^{1,2})} \grad \Phi^{1,2} \right) \utwo \cdot \boldsymbol{\eta},\\
       I_7[\boldsymbol{\eta}] &:= \int_{\Omega^1(t)} \left( \det(\grad \Phi^{1,2})  (\grad \Phi^{1,2})^{-1} - \mbb{I} \right)\grad \ptwo \cdot \boldsymbol{\eta},\\
       I_8[\boldsymbol{\eta}] &:= \mu\int_{\Omega^1(t)} \left((\grad \Phi^{1,2})^{-1} (\grad \Phi^{1,2})^{-T} \grad \Phi^{1,2} - \mbb{I}\right) \grad \utwo : \grad \boldsymbol{\eta},\\
       I_9[\boldsymbol{\eta}] &:= \mu \int_{\Omega^1(t)} \det(\grad\Phi^{1,2}) (\grad \Phi^{1,2})^{-1} (\grad \Phi^{1,2})^{-T} \grad \left(\frac{1}{\det(\grad \Phi^{1,2})} \grad \Phi^{1,2} \right) \utwo : \grad \boldsymbol{\eta},\\
       I_{10}[\boldsymbol{\eta}] &:= \rho \int_{\Omega^1(t)} ((\grad \Phi^{1,2})^{T}(\grad \Phi^{1,2})^{-1}- \mbb{I})\widehat{\widetilde{\mbf{u}}^2} \cdot \grad \utwo \boldsymbol{\eta},\\
       I_{11}[\boldsymbol{\eta}] &:= \rho \int_{\Omega^1(t)} \det(\grad \Phi^{1,2}) \left((\grad \Phi^{1,2})^{-1} \widehat{\widetilde{\mbf{u}}^2} \cdot \grad \left(\frac{1}{\det(\grad \Phi^{1,2})} \grad \Phi^{1,2}\right) \utwo \right)\boldsymbol{\eta},\\
       I_{12}[\boldsymbol{\eta}] &:= \int_{\Omega^1(t)} (\det(\grad \Phi^{1,2}) - 1)\widehat{\mbf{B}^2} \cdot \boldsymbol{\eta}.
   \end{align*}
   One can think of these terms as ``consistency errors'' arising from the fact that $(\utwo, \ptwo)$ are not the solution of~\eqref{eqn: evolving NS2} on $\Omega^1(t)$, and instead solve a perturbed version of the system due to the different domain.
   The plan from here is to test \eqref{contraction pf5} with $A_{\sigma}(\overline{\mbf{u}}^1 - \utwo)$ and repeat the arguments from \eqref{lemma: fsi energy estimates}, provided the functions $I_k[\cdot]$ are suitably bounded.
   As such we firstly bound these functionals in the following lemma.

   \begin{lemma}
   \label{lemma: fsi functional bounds}
       The functionals $I_k[\cdot]$, as defined above, for $k=1,\ldots,12$ are bounded as follows:
       \begin{align*}
           \left|I_1[\boldsymbol{\eta}]\right| &\leq C T \left( \|\mbf{q}^1-\mbf{q}^2\|_{\mbf{H}^2([0,T])} + \|\boldsymbol{\omega}^1-\boldsymbol{\omega}^2\|_{\mbf{H}^1([0,T])} \right) \left\|\ddt{\utwo}\right\|_{\mbf{L}^2(\Omega^1(t))} \|\boldsymbol{\eta}\|_{\mbf{L}^2(\Omega^1(t))},\\
           \left|I_2[\boldsymbol{\eta}]\right| &\leq C \left( \|\mbf{q}^1-\mbf{q}^2\|_{\mbf{H}^2([0,T])} + \|\boldsymbol{\omega}^1-\boldsymbol{\omega}^2\|_{\mbf{H}^1([0,T])} \right) \|\utwo\|_{\mbf{L}^2(\Omega^1(t))} \|\boldsymbol{\eta}\|_{\mbf{L}^2(\Omega^1(t))},\\
           \left|I_3[\boldsymbol{\eta}]\right| &\leq C \left( \|\mbf{q}^1-\mbf{q}^2\|_{\mbf{H}^2([0,T])} + \|\boldsymbol{\omega}^1-\boldsymbol{\omega}^2\|_{\mbf{H}^1([0,T])} \right) \left\|\utwo\right\|_{\mbf{H}^1(\Omega^1(t))} \|\boldsymbol{\eta}\|_{\mbf{L}^2(\Omega^1(t))},\\
           \left|I_4[\boldsymbol{\eta}]\right| &\leq C \left( \|\mbf{q}^1-\mbf{q}^2\|_{\mbf{H}^2([0,T])} + \|\boldsymbol{\omega}^1-\boldsymbol{\omega}^2\|_{\mbf{H}^1([0,T])} \right) \left\|\utwo\right\|_{\mbf{L}^2(\Omega^1(t))} \|\boldsymbol{\eta}\|_{\mbf{L}^2(\Omega^1(t))},\\
           \left|I_5[\boldsymbol{\eta}]\right| &\leq C T \left( \|\mbf{q}^1-\mbf{q}^2\|_{\mbf{H}^2([0,T])} + \|\boldsymbol{\omega}^1-\boldsymbol{\omega}^2\|_{\mbf{H}^1([0,T])} \right) \|\utwo\|_{\mbf{L}^4(\Omega^1(t))} \|\utwo\|_{\mbf{W}^{1,4}(\Omega^1(t))}\|\boldsymbol{\eta}\|_{\mbf{L}^2(\Omega^1(t))},\\
           \left|I_6[\boldsymbol{\eta}]\right| &\leq C T \left( \|\mbf{q}^1-\mbf{q}^2\|_{\mbf{H}^2([0,T])} + \|\boldsymbol{\omega}^1-\boldsymbol{\omega}^2\|_{\mbf{H}^1([0,T])} \right) \|\utwo\|_{\mbf{L}^4(\Omega^1(t))}^2 \|\boldsymbol{\eta}\|_{\mbf{L}^2(\Omega^1(t))},\\
           \left|I_7[\boldsymbol{\eta}]\right| &\leq C T \left( \|\mbf{q}^1-\mbf{q}^2\|_{\mbf{H}^2([0,T])} + \|\boldsymbol{\omega}^1-\boldsymbol{\omega}^2\|_{\mbf{H}^1([0,T])} \right) \|\ptwo\|_{H^1(\Omega^1(t))} \|\boldsymbol{\eta}\|_{\mbf{L}^2(\Omega^1(t))},\\
           \left|I_8[\boldsymbol{\eta}]\right| &\leq C T \left( \|\mbf{q}^1-\mbf{q}^2\|_{\mbf{H}^2([0,T])} + \|\boldsymbol{\omega}^1-\boldsymbol{\omega}^2\|_{\mbf{H}^1([0,T])} \right) \|\utwo\|_{\mbf{H}^2(\Omega^1(t))} \|\boldsymbol{\eta}\|_{\mbf{L}^2(\Omega^1(t))},\\
           \left|I_9[\boldsymbol{\eta}]\right| &\leq C T \left( \|\mbf{q}^1-\mbf{q}^2\|_{\mbf{H}^2([0,T])} + \|\boldsymbol{\omega}^1-\boldsymbol{\omega}^2\|_{\mbf{H}^1([0,T])} \right) \|\utwo\|_{\mbf{H}^1(\Omega^1(t))} \|\boldsymbol{\eta}\|_{\mbf{L}^2(\Omega^1(t))},\\
           \left|I_{10}[\boldsymbol{\eta}]\right| &\leq C T \left( \|\mbf{q}^1-\mbf{q}^2\|_{\mbf{H}^2([0,T])} + \|\boldsymbol{\omega}^1-\boldsymbol{\omega}^2\|_{\mbf{H}^1([0,T])} \right) \|\widehat{\widetilde{\mbf{u}}^2}\|_{\mbf{L}^{4}(\Omega^1(t))}\|\utwo\|_{\mbf{W}^{1,4}(\Omega^1(t))} \|\boldsymbol{\eta}\|_{\mbf{L}^2(\Omega^1(t))},\\
           \left|I_{11}[\boldsymbol{\eta}]\right| &\leq C T \left( \|\mbf{q}^1-\mbf{q}^2\|_{\mbf{H}^2([0,T])} + \|\boldsymbol{\omega}^1-\boldsymbol{\omega}^2\|_{\mbf{H}^1([0,T])} \right) \|\widehat{\widetilde{\mbf{u}}^2}\|_{\mbf{L}^{4}(\Omega^1(t))}\|\utwo\|_{\mbf{L}^{4}(\Omega^1(t))} \|\boldsymbol{\eta}\|_{\mbf{L}^2(\Omega^1(t))},\\
           \left|I_{12}[\boldsymbol{\eta}]\right| &\leq C T \left( \|\mbf{q}^1-\mbf{q}^2\|_{\mbf{H}^2([0,T])} + \|\boldsymbol{\omega}^1-\boldsymbol{\omega}^2\|_{\mbf{H}^1([0,T])} \right) \|\widehat{\mbf{B}}^2\|_{\mbf{L}^2(\Omega^1(t))} \|\boldsymbol{\eta}\|_{\mbf{L}^2(\Omega^1(t))},
       \end{align*}
       where $C$ denotes a constant independent of $\mbf{q}^1$, $\mbf{q}^2$, $\boldsymbol{\omega}^1$, $\boldsymbol{\omega}^2$ and $T$.
   \end{lemma}
   \begin{proof}
    For $I_1[\cdot]$ one immediately finds that
    \[ \left|I_1[\boldsymbol{\eta}]\right| \leq \rho \left\| \grad \Phi^{1,2} - \mbb{I}\right\|_{\mbf{L}^\infty(\Omega^1(t))}\left\|\ddt{\utwo}\right\|_{\mbf{L}^2(\Omega^1(t))} \|\boldsymbol{\eta}\|_{\mbf{L}^2(\Omega^1(t))},  \]
       where one can calculate that
       \[\grad \Phi^{1,2}(t) = \grad (\Phi^2(t) \circ \Phi^{1}(t)^{-1}) = \left(\grad \Phi^2(t)(\grad \Phi^1(t))^{-1} \right) \circ \Phi^1(t)^{-1}.\]
       Then clearly one has that
       \begin{align*}
           \|\grad \Phi^{1,2} - \mbb{I}\|_{\mbf{L}^{\infty}(\Omega^1(t))} &\leq C \|\grad \Phi^{2} (\grad \Phi^1)^{-1}  - \mbb{I}\|_{\mbf{L}^{\infty}(\Omega(0))}\\
           &\leq C \|\grad \Phi^{2} -\grad \Phi^1 \|_{\mbf{L}^{\infty}(\Omega(0))} \|(\grad \Phi^1)^{-1} \|_{\mbf{L}^{\infty}(\Omega(0))},
       \end{align*}
       whence the bound follows from using Lemma \ref{lemma: gradient difference} and the uniform bounds on $\|\grad \Phi^i\|_{\mbf{L}^\infty(\Omega(0))}$.
       
       For $I_2[\cdot]$ one immediately finds that
       \begin{align*}
           \left|I_2[\boldsymbol{\eta}]\right| \leq C \left\|\ddt{}\left( \frac{1}{\det(\grad \Phi^{1,2})} \grad \Phi^{1,2}\right) \right\|_{\mbf{L}^\infty(\Omega^1(t))} \|\utwo\|_{\mbf{L}^2(\Omega^1(t))} \|\boldsymbol{\eta}\|_{\mbf{L}^2(\Omega^1(t))},
       \end{align*}
       where we have used the uniform bound for $\grad \Phi^{1,2}$.
       Clearly the issue is bounding the $\mbf{L}^\infty$ term, and to control this term we firstly find an expression for $\ddt{}\grad \Phi^{1,2}$.
       Firstly we observe that since
       \[\grad \Phi^{1,2}(t) = \left(\grad \Phi^2(t)(\grad \Phi^1(t))^{-1} \right) \circ \Phi^1(t)^{-1},\]
        one has that
       \[  \grad \Phi^2 (\grad \Phi^1)^{-1} = (\grad\Phi^{1,2}) \circ \Phi^1. \]
       Differentiating this equation in time yields
       \[ \grad \ddt{\Phi^2} (\grad \Phi^1)^{-1} + \grad \Phi^2 \ddt{}\left( \grad \Phi^1 \right)^{-1}  = \left(\ddt{} \grad \Phi^{1,2}\right)\circ \Phi^1 + ((\grad^2 \Phi^{1,2}) \circ \Phi^1) \ddt{\Phi^1}, \]
       which simplifies to
       \begin{align*}
           \left(\ddt{} \grad \Phi^{1,2}\right)\circ \Phi^1 = \grad \mbf{V}^2(\Phi^2) \grad \Phi^2 (\grad \Phi^1)^{-1} -  \grad \Phi^2 (\grad \Phi^1)^{-1} \grad \mbf{V}^1(\Phi^1) - ((\grad^2 \Phi^{1,2}) \circ \Phi^1) \mbf{V}^1(\Phi^1),
       \end{align*}
       where we understand $\mbf{V}^i(\Phi^i)$ to mean $\mbf{V}^i(\Phi^i(\cdot,t),t)$.
       Hence one finds that 
       \begin{align*}
           &\left\|\ddt{} (\grad \Phi^{1,2}) \right\|_{\mbf{L}^\infty(\Omega^1(t))}\\
           &\leq C \left\|\grad \mbf{V}^2(\Phi^2) \grad \Phi^2 (\grad \Phi^1)^{-1} -  \grad \Phi^2 (\grad \Phi^1)^{-1} \grad \mbf{V}^1(\Phi^1) - ((\grad^2 \Phi^{1,2}) \circ \Phi^1) \mbf{V}^1(\Phi^1)\right\|_{\mbf{L}^\infty(\Omega(0))}.
       \end{align*}
       To bound this term we see that
       \begin{align*}
           &\grad \mbf{V}^2(\Phi^2) \grad \Phi^2 (\grad \Phi^1)^{-1} -  \grad \Phi^2 (\grad \Phi^1)^{-1} \grad \mbf{V}^1(\Phi^1)\\
           &= \grad \Phi^2 \left( (\grad \Phi^2)^{-1} \grad \mbf{V}^2(\Phi^2) \grad \Phi^2 - (\grad \Phi^1)^{-1} \grad \mbf{V}^1(\Phi^1) \grad \Phi^1\right) (\grad \Phi^1)^{-1},
       \end{align*}
       and hence by using Lemma \ref{lemma: gradient difference} and Lemma \ref{lemma: gradient difference ddt} one finds that
       \begin{multline*}
       \left\| \grad \mbf{V}^2(\Phi^2) \grad \Phi^2 (\grad \Phi^1)^{-1} -  \grad \Phi^2 (\grad \Phi^1)^{-1} \grad \mbf{V}^1(\Phi^1) \right \|_{\mbf{L}^\infty(\Omega(0))}\\
        \leq C (1 + T) \left( \|\mbf{q}^1-\mbf{q}^2\|_{\mbf{H}^2([0,T])} + \|\boldsymbol{\omega}^1-\boldsymbol{\omega}^2\|_{\mbf{H}^1([0,T])} \right).
       \end{multline*}
       For the higher order term we firstly calculate
       \begin{align*}
           \grad^2 \Phi^{1,2} &= \grad \left( (\grad \Phi^2(\grad \Phi^1)^{-1}) \circ (\Phi^1)^{-1}\right)\\
           &= \left(\grad^2 \Phi^2 (\grad \Phi^1)^{-2} - \grad \Phi^2 (\grad \Phi^1)^{-1} \grad^2 \Phi^1 (\grad \Phi^1)^{-2}\right)\circ (\Phi^1)^{-1},
       \end{align*}
       and similarly to the above one writes
       \begin{multline*}
           \left(\grad^2 \Phi^2 (\grad \Phi^1)^{-2} - \grad \Phi^2 (\grad \Phi^1)^{-1} \grad^2 \Phi^1 (\grad \Phi^1)^{-2}\right)\\
           = \grad \Phi^2\left( (\grad \Phi^2)^{-1} \grad^2 \Phi^2  -  (\grad \Phi^1)^{-1} \grad^2 \Phi^1\right)(\grad \Phi^1)^{-2}.
       \end{multline*}
       Thus by using our uniform bounds and Lemma \ref{lemma: gradient difference} one finds that
       \[ \left\| ((\grad^2 \Phi^{1,2}) \circ \Phi^1) \mbf{V}^1(\Phi^1)\right\|_{\mbf{L}^\infty(\Omega(0))} \leq C T \left( \|\mbf{q}^1-\mbf{q}^2\|_{\mbf{H}^2([0,T])} + \|\boldsymbol{\omega}^1-\boldsymbol{\omega}^2\|_{\mbf{H}^1([0,T])} \right). \]
       The bound for $I_2[\cdot]$ follows after noting the time derivative of the determinant term may be bounded similarly since
       \[\ddt{} \left(\frac{1}{\det(\grad \Phi^{1,2})}\right) = -\frac{1}{\det(\grad \Phi^{1,2})} \tr{(\grad \Phi^{1,2})^{-1} \ddt{}\grad \Phi^{1,2}},\]
       from Jacobi's formula.

       To bound $I_3[\cdot]$ one can calculate that
       \[ \ddt{}\Phi^{1,2} = \mbf{V}^2(\Phi^{1,2}) - \left(\grad \Phi^2 (\grad \Phi^1)^{-1}\right) \circ (\Phi^1)^{-1} \mbf{V}^1, \]
       and hence 
       \[ \left(\ddt{}\Phi^{1,2} \right) \circ \Phi^1 = \grad \Phi^2 \left( (\grad \Phi^2)^{-1} \mbf{V}^2(\Phi^2) - (\grad \Phi^1)^{-1} \mbf{V}^1(\Phi^1)\right), \]
        so that the result follows in the obvious way.

        One can show the bound for $I_4[\cdot]$ after using the above expressions for $\ddt{} \Phi^{1,2}$ and $\grad^2 \Phi^{1,2}$ and arguing as we have above.
        For $I_5[\cdot]$ we have that
        \[ \frac{1}{\det(\grad \Phi^{1,2})} (\grad \Phi^{1,2})^T = \left(\frac{\det(\grad \Phi^1)}{\det(\grad \Phi^2)}(\grad \Phi^1)^{-T} (\grad \Phi^2)^T\right) \circ (\Phi^1)^{-1}, \]
        and hence by pulling back onto $\Omega(0)$
        \begin{align*}
        &\left\|\frac{1}{\det(\grad \Phi^{1,2})} (\grad \Phi^{1,2} )^T - \mbb{I}\right\|_{\mbf{L}^\infty(\Omega^1(0))} \leq C \left\| \frac{\det(\grad \Phi^1)}{\det(\grad \Phi^2)}(\grad \Phi^1)^{-T} (\grad \Phi^2)^T - \mbb{I} \right\|_{\mbf{L}^\infty(\Omega(0))}\\
        &\leq  C\left\| \det(\grad \Phi^1)(\grad \Phi^1)^{-T}\right\|_{\mbf{L}^\infty(\Omega(0))} \left\| \frac{1}{\det(\grad \Phi^2)} (\grad \Phi^2)^T - \frac{1}{\det(\grad \Phi^1)} (\grad \Phi^1)^T \right\|_{\mbf{L}^\infty(\Omega(0))}\\
        &\leq C\left\| \grad \Phi^1 - \grad \Phi^2 \right\|_{\mbf{L}^\infty(\Omega(0))},
        \end{align*}
        so that the result will readily follow from Lemma \ref{lemma: gradient difference} after noting that
        \[\left\|\utwo \cdot \grad \utwo \boldsymbol{\eta}\right\|_{\mbf{L}^1(\Omega^1(t))} \leq \|\utwo\|_{\mbf{L}^4(\Omega^1(t))} \|\utwo\|_{\mbf{W}^{1,4}(\Omega^1(t))}\|\boldsymbol{\eta}\|_{\mbf{L}^2(\Omega^1(t))}. \]
        The bounds for $I_6[\cdot], I_7[\cdot]$ are straightforward after using the above expressions for $\grad^2 \Phi^{1,2}$ and $\grad \Phi^{1,2}$ respectively.

        Since we will want to choose $\boldsymbol{\eta} = A_{\sigma}(\overline{\mbf{u}}^1 - \utwo)$ we cannot bound $\boldsymbol{\eta}$ in the $\mbf{H}^1$ norm, and so to bound $I_8[\cdot]$ we firstly integrate by parts to see that
        \begin{multline*}
            \mu\int_{\Omega^1(t)} \left((\grad \Phi^{1,2})^{-1} (\grad \Phi^{1,2})^{-T} \grad \Phi^{1,2} - \mbb{I}\right) \grad \utwo : \grad \boldsymbol{\eta}\\
            = -\mu\int_{\Omega^1(t)} \Div\left(\left((\grad \Phi^{1,2})^{-1} (\grad \Phi^{1,2})^{-T} \grad \Phi^{1,2} - \mbb{I}\right) \grad \utwo \right) \cdot \boldsymbol{\eta}.
        \end{multline*}
        We then observe that
        \begin{align*}
            (\grad \Phi^{1,2})^{-1} (\grad \Phi^{1,2})^{-T} \grad \Phi^{1,2} - \mbb{I} &= \left( \grad \Phi^1 (\grad \Phi^2)^{-1} (\grad \Phi^2)^{-T} (\grad \Phi^1)^T \grad \Phi^2 (\grad \Phi^1)^{-1} \right) \circ (\Phi^1)^{-1} - \mbb{I}\\
            &= \grad \Phi^1 (\grad \Phi^2)^{-1} (\grad \Phi^2)^{-T} \left( (\grad \Phi^1)^T- (\grad \Phi^2)^T\right)\grad \Phi^2 (\grad \Phi^1)^{-1},
        \end{align*}
        which we can bound in the usual way.
        One may also repeat a similar argument for the higher order term $\Div\left((\grad \Phi^{1,2})^{-1} (\grad \Phi^{1,2})^{-T} \grad \Phi^{1,2} - \mbb{I}\right)$ by using the expression for $\grad^2 \Phi^{1,2}$.
        In both cases one concludes by using Lemma \ref{lemma: gradient difference} once again.
       The bound for $I_9[\cdot]$ is similar, where one first integrates by parts for
       \begin{multline*}
           \mu \int_{\Omega^1(t)} \det(\grad\Phi^{1,2}) (\grad \Phi^{1,2})^{-1} (\grad \Phi^{1,2})^{-T} \grad \left(\frac{1}{\det(\grad \Phi^{1,2})} \grad \Phi^{1,2} \right) \utwo : \grad \boldsymbol{\eta}\\= -\mu \int_{\Omega^1(t)} \Div\left(\det(\grad\Phi^{1,2}) (\grad \Phi^{1,2})^{-1} (\grad \Phi^{1,2})^{-T} \grad \left(\frac{1}{\det(\grad \Phi^{1,2})} \grad \Phi^{1,2} \right) \utwo\right) \cdot\boldsymbol{\eta},
       \end{multline*}
       where we note that this will involve a third order derivative of $\Phi^{1,2}$, which is defined almost everywhere since the $\Phi^i$ are $\mbf{C}^{2,1}$ in space.
       To control this we first calculate $\grad^3 \Phi^{1,2}$ as
       \[ \grad^3 \Phi^{1,2} = \grad \left(\left(\grad^2 \Phi^2 (\grad \Phi^1)^{-2} - \grad \Phi^2 (\grad \Phi^1)^{-1} \grad^2 \Phi^1 (\grad \Phi^1)^{-2}\right)\circ (\Phi^1)^{-1}\right), \]
       which one can expand to find
       \begin{align*}
           &\grad \left(\left(\grad^2 \Phi^2 (\grad \Phi^1)^{-2} - \grad \Phi^2 (\grad \Phi^1)^{-1} \grad^2 \Phi^1 (\grad \Phi^1)^{-2}\right)\circ (\Phi^1)^{-1}\right)\\
           &= \left(\grad^3 \Phi^2 (\grad \Phi^1)^{-3}\right) \circ (\Phi^1)^{-1}\\
           &- \left(\grad^2 \Phi^2 (\grad \Phi^1)^{-1} \grad^2 \Phi^1 (\grad \Phi^1)^{-3}\right)\circ (\Phi^1)^{-1}\\
           &- \left( \grad^2 \Phi^2(\grad \Phi^1)^{-2} \grad^2 \Phi^1 (\grad \Phi^1)^{-2} \right) \circ (\Phi^1)^{-1}\\
           & - \left(\grad^2 \Phi^2 (\grad \Phi^1)^{-1} \grad^2 \Phi^1 (\grad \Phi^1)^{-3} \right) \circ (\Phi^1)^{-1}\\
           & + \left( \grad \Phi^2 (\grad \Phi^1)^{-1} \grad^2 \Phi^1 (\grad \Phi^1)^{-1}\grad^2 \Phi^1 (\grad \Phi^1)^{-3} \right) \circ (\Phi^1)^{-1}\\
           & - \left(\grad \Phi^2 (\grad \Phi^1)^{-1} \grad^3 \Phi^1 (\grad \Phi^1)^{-3}\right)\circ (\Phi^1)^{-1}\\
           & + \left(\grad \Phi^2 (\grad \Phi^1)^{-1} \grad^2 \Phi^1 (\grad \Phi^1)^{-1} \grad^2 \Phi^1  (\grad \Phi^1)^{-3}\right)\circ (\Phi^1)^{-1}\\
           & + \left(\grad \Phi^2 (\grad \Phi^1)^{-1} \grad^2 \Phi^1 (\grad \Phi^1)^{-2} \grad^2 \Phi^1  (\grad \Phi^1)^{-2}\right)\circ (\Phi^1)^{-1}.
       \end{align*}
       By appropriately grouping these terms together one finds
       \begin{align*}
           (\grad^3 \Phi^{1,2}) \circ \Phi^1 &= \grad \Phi^2 \left[ (\grad \Phi^2)^{-1}\grad^3 \Phi^2 - (\grad \Phi^1)^{-1}\grad^3 \Phi^1\right] (\grad \Phi^1)^{-3}\\
           &+2 \grad \Phi^2 \left[ (\grad \Phi^2)^{-1}\grad^2 \Phi^2 - (\grad \Phi^1)^{-1}\grad^2 \Phi^1\right] (\grad \Phi^1)^{-1} \grad^2 \Phi^1 (\grad \Phi^1)^{-3}\\
           &+ \grad \Phi^2 \left[ (\grad \Phi^1)^{-1} \grad^2 \Phi^1 - (\grad \Phi^2)^{-1} \grad^2 \Phi^2 \right] (\grad \Phi^1)^{-2} \grad^2 \Phi^1 (\grad^2 \Phi^1)^{-2},
       \end{align*}
       where the terms in the square brackets can be suitably bounded by using Lemma \ref{lemma: gradient difference}, and the terms outside the brackets are uniformly bounded.
       It is then straightforward, albeit tedious, to show that this gives an appropriate bound on the divergence term, as well as the determinant term (by similar logic to the bound for $I_2[\cdot]$) we omit further details.
       Likewise the remaining bounds for $I_{10}[\cdot], I_{11}[\cdot], I_{12}[\cdot]$ are dealt with by similar arguments to those above and so we omit these calculations too.
   \end{proof}
   \begin{remark}
       It may seem that the bounds for $I_2[\cdot], I_3[\cdot], I_4[\cdot]$ are problematic as they do not gain a power of $T$, however this is compensated for by the $L^\infty$ in time bounds for the corresponding fluid terms appearing in inequalities, which allows one to gain a power of $T$ after integrating in time.
   \end{remark}

With this lemma we now show that $\mathscr{F}: D(\mathscr{F}) \rightarrow D(\mathscr{F})$ is a contraction, under suitable smallness assumptions as in Lemma~\ref{lemma: non degeneracy2}.

\begin{lemma}
\label{lemma: fsi contraction}
    Given some sufficiently large $R > 0$, and a sufficiently small relative density, $\frac{\rho}{\rho_B}$, as in Lemma~\ref{lemma: non degeneracy2}, there exists some small $T > 0$ (determined by $R$ and the initial data) such that $\mathscr{F}$ is a contraction.
\end{lemma}
\begin{proof}
To streamline this proof we omit many of the calculations and instead refer back to earlier proofs which provide details on related calculations --- namely Lemma~\ref{lemma: fsi energy estimates} and Lemma~\ref{lemma: fsi energy estimates2}.
By taking $\boldsymbol{\eta} = A_{\sigma}(\overline{\mbf{u}}^1 - \utwo)$ in \eqref{contraction pf5} one obtains
\begin{align}
\begin{aligned}
        &\rho \int_{\Omega^1(t)} \ddt{(\overline{\mbf{u}}^1 - \utwo)} \cdot A_{\sigma}(\overline{\mbf{u}}^1 - \utwo) + \mu \int_{\Omega^1(t)} |A_{\sigma}(\overline{\mbf{u}}^1 - \utwo)|^2\\
        &+ \rho \int_{\Omega^1(t)} \overline{\mbf{u}}^1 \cdot \grad\overline{\mbf{u}}^1 A_{\sigma}(\overline{\mbf{u}}^1 - \utwo)
        - \rho \int_{\Omega^1(t)} \utwo \cdot \grad \utwo A_{\sigma}(\overline{\mbf{u}}^1 - \utwo)\\
       &+ \rho \int_{\Omega^1(t)}  \left((\widetilde{\mbf{u}}^1 \cdot \grad \overline{\mbf{u}}^1) + (\overline{\mbf{u}}^1 \cdot \grad\widetilde{\mbf{u}}^1) \right)A_{\sigma}(\overline{\mbf{u}}^1 - \utwo) \\
       &- \rho \int_{\Omega^1(t)}  \left((\widehat{\widetilde{\mbf{u}}^2} \cdot \grad \utwo) + (\utwo \cdot \grad\widehat{\widetilde{\mbf{u}}^2}) \right)A_{\sigma}(\overline{\mbf{u}}^1 - \utwo)\\
       &+ \int_{\Omega^1(t)} (\mbf{B}^1 - \widehat{\mbf{B}^2}) \cdot A_{\sigma}(\overline{\mbf{u}}^1 - \utwo) = \sum_{k=1}^{12} I_k[A_{\sigma}(\overline{\mbf{u}}^1 - \utwo)].
\end{aligned}\label{contraction pf7}
\end{align}
While $A_{\sigma}(\overline{\mbf{u}}^1 - \utwo)$ is not an element of $\mbf{H}_0^1(\Omega^1(t))$, the above equality can be justified by an approximation argument choosing $\boldsymbol{\eta}_k \in \mbf{H}^1_0(\Omega^1(t))$ such that $\boldsymbol{\eta}_k \rightarrow A_{\sigma}(\overline{\mbf{u}}^1 - \utwo)$ in $\mbf{L}^2(\Omega^1(t))$ as $k \rightarrow \infty$.

To control the advective terms we observe that
\begin{align*}
\overline{\mbf{u}}^1 \cdot \grad\overline{\mbf{u}}^1 A_{\sigma}(\overline{\mbf{u}}^1 - \utwo) - \utwo \cdot \grad \utwo A_{\sigma}(\overline{\mbf{u}}^1 - \utwo) &= (\overline{\mbf{u}}^1 - \utwo) \cdot \grad\overline{\mbf{u}}^1 A_{\sigma}(\overline{\mbf{u}}^1 - \utwo)\\
&+ \utwo \cdot \grad (\overline{\mbf{u}}^1 - \utwo) A_{\sigma}(\overline{\mbf{u}}^1 - \utwo),
\end{align*}
and so by using Young's inequality one finds
\begin{align*}
    \left| \int_{\Omega^1(t)} \overline{\mbf{u}}^1 \cdot \grad\overline{\mbf{u}}^1 A_{\sigma}(\overline{\mbf{u}}^1 - \utwo) -  \int_{\Omega^1(t)} \utwo \cdot \grad \utwo A_{\sigma}(\overline{\mbf{u}}^1 - \utwo)\right| &\leq \frac{\mu}{16} \|A_{\sigma}(\overline{\mbf{u}}^1 - \utwo)\|_{\mbf{L}^2(\Omega^1(t))}^2\\
    &+ C \|\grad \overline{\mbf{u}}^1\|_{\mbf{L}^4(\Omega^1(t))}^2 \|\overline{\mbf{u}}^1 - \utwo\|_{\mbf{L}^4(\Omega^1(t))}^2\\
    &+ C\|\utwo\|_{\mbf{L}^\infty(\Omega^1(t))}^2 \|\grad(\overline{\mbf{u}}^1 - \utwo)\|_{\mbf{L}^2(\Omega^1(t))}^2,
\end{align*}
where we use the Sobolev embedding $\mbf{H}^1(\Omega^1(t)) \hookrightarrow \mbf{L}^4(\Omega^1(t))$ and Poincar\'e's inequality for
\[ \|\grad \overline{\mbf{u}}^1\|_{\mbf{L}^4(\Omega^1(t))}^2 \|\overline{\mbf{u}}^1 - \utwo\|_{\mbf{L}^4(\Omega^1(t))}^2 \leq \|\grad \overline{\mbf{u}}^1\|_{\mbf{L}^4(\Omega^1(t))}^2 \|\grad(\overline{\mbf{u}}^1 - \utwo)\|_{\mbf{L}^2(\Omega^1(t))}^2 \]
Notice that, in contrast to Lemma \ref{lemma: fsi energy estimates}, we have not gained a term like $\|\grad(\overline{\mbf{u}}^1 - \utwo)\|_{\mbf{L}^2(\Omega^1(t))}^6$, and so we will be able to apply Gr\"onwall's inequality, rather than Lemma \ref{biharilasalle}, for our estimate.
The same argument applies to the terms involving $\widetilde{\mbf{u}}^1, \widehat{\widetilde{\mbf{u}}^2}$, for example
\begin{align*}
    \widetilde{\mbf{u}}^1 \cdot \grad \overline{\mbf{u}}^1 A_{\sigma}(\overline{\mbf{u}}^1 - \utwo) -\widehat{\widetilde{\mbf{u}}^2} \cdot \grad \utwo A_{\sigma}(\overline{\mbf{u}}^1 - \utwo) &= (\widetilde{\mbf{u}}^1 - \widehat{\widetilde{\mbf{u}}^2}) \cdot \grad \overline{\mbf{u}}^1 A_{\sigma}(\overline{\mbf{u}}^1 - \utwo)\\
    &+ \widehat{\widetilde{\mbf{u}}^2} \cdot \grad (\overline{\mbf{u}}^1-\utwo) A_{\sigma}(\overline{\mbf{u}}^1 - \utwo),
\end{align*}
which is bounded similarly to the above.

From here on one may now argue as in Lemma \ref{lemma: fsi energy estimates} to obtain a bound of the form
\begin{align*}
    \frac{\rho}{2} \frac{\mathrm{d}}{\mathrm{d}t} \int_{\Omega^1(t)} |\grad (\overline{\mbf{u}}^1 - \utwo)|^2 + \frac{\mu}{2} \int_{\Omega^1(t)} |A_{\sigma} (\overline{\mbf{u}}^1 - \utwo)|^2 &\leq \frac{1}{2}T^\frac{1}{3} \|\mbf{B}^1 - \widehat{\mbf{B}^2}\|_{\mbf{H}^1(\Omega^1(t))}^2\\
        &+ K(t)\|\grad \mbf{u}\|_{\mbf{L}^2(\Omega^1(t))}^2 + \sum_{k=1}^{12} |I_k[A_{\sigma}(\overline{\mbf{u}}^1 - \utwo)]|,
\end{align*}
where $K(t)$ is some integrable function, and we have again taken $\kappa = T^{\frac{1}{3}}$.
Clearly one obtains a suitable bound on $\overline{\mbf{u}}^1 - \utwo$ in $L^2_{\mbf{H}^2}(\Phi^1)$ by using the bounds from Lemma \ref{lemma: fsi functional bounds}, an appropriate bound on $\|\mbf{B}^1 - \widehat{\mbf{B}^2}\|_{\mbf{H}^1(\Omega(t))}$, Young's inequality, and concludes by using Gr\"onwall's inequality.
The only ingredient here we have not discussed is the required bound on~$\mbf{B}^1 - \widehat{\mbf{B}^2}$ in $L^2_{\mbf{H}^1}(\Phi^1)$.
This term is bounded by arguing along the same lines as we have in this proof, we refer the reader to \S \ref{subsection: stokes perturbation} for some of the details on this calculation.
After a dense and tedious calculation one finds that
\begin{align}
    \|\mbf{B}^1 - \widehat{\mbf{B}^2}\|_{\mbf{H}^1(\Omega^1(t))} \leq C \left| \frac{\mathrm{d}^2 \mbf{q}^1}{\mathrm{d}t^2} - \frac{\mathrm{d}^2 \mbf{q}^2}{\mathrm{d}t^2}\right| + C\left| \frac{\mathrm{d} \boldsymbol{\omega}^1}{\mathrm{d}t} - \frac{\mathrm{d} \boldsymbol{\omega}^2}{\mathrm{d}t}\right| + \text{ lower order terms}, \label{eqn: body force difference}
\end{align}
which is sufficient for our purposes, and lets one conclude an appropriate bound for $\overline{\mbf{u}}^1 - \utwo$ in $L^2_{\mbf{H}^2}(\Phi^1)$.

The bound for $\overline{p}^1 - \ptwo$ in $L^2_{H^1}(\Phi^1)$ is more involved, for which one emulates the proof of Lemma \ref{lemma: fsi energy estimates2}.
In particular one finds that
\[ \int_0^T \| \overline{p}^1 - \ptwo \|_{H^1(\Omega^1(t))}^2 \leq \int_0^T \|\mbf{B}^1 - \widehat{\mbf{B}^2}\|_{\mbf{L}^2(\Omega^1(t))} + C T^s \left( \|\mbf{q}^1-\mbf{q}^2\|_{\mbf{H}^2([0,T])} + \|\boldsymbol{\omega}^1-\boldsymbol{\omega}^2\|_{\mbf{H}^1([0,T])} \right),  \]
for some $s > 0$.
We now face the same difficulty as in the proof of Lemma \ref{lemma: non degeneracy2}, namely that there is no power of $T$ involved in the first term on the right-hand side, and hence~\eqref{eqn: body force difference} may not be sufficient to prove that~$\mathscr{F}$ is a contraction.
We resolve this in the same way as in the proof of Lemma \ref{lemma: non degeneracy2}, where we observe that one gains a factor of $\rho$ from the definition of $\mbf{B}^i$, and one will gain a factor of $\frac{1}{\rho_B}$ from the $m$ on the left-hand side of \eqref{contraction pf2}.
In particular, we will be able to show that $\mathscr{F}$ is a contraction if the relative density $\frac{\rho}{\rho_B}$ is sufficiently small.
We omit further details on the bound for $\overline{p}^1 - \ptwo$.

Once one obtains the bounds for $\overline{\mbf{u}}^1 - \utwo$ and $ \overline{p}^1 - \ptwo$ in $L^2_{\mbf{H}^2}(\Phi^1)$ and $L^2_{H^1}(\Phi^1)$ respectively it is then straightforward to show bounds for the functions ${\mbf{u}}^1 - \widehat{\mbf{u}^2}$ and ${p}^1 - \widehat{p^2}$ by arguing along the lines of Appendix \ref{appendix:stokes} to show that $\widetilde{\mbf{u}}^1 - \widehat{\widetilde{\mbf{u}}^2}$ and $\widetilde{p}^1 - \widehat{\widetilde{p}^2}$ are bounded appropriately (see \eqref{stokes regularity5} in particular).
Finally, we use our bounds in \eqref{contraction pf2} to control the first two integrals on the right-hand side, and one controls the other terms by essentially the same arguments those in the proof of Lemma \ref{lemma: fsi functional bounds}.
All in all, this gives us a bound of the form
\begin{align}
\begin{aligned}
    \|\mbf{Q}^1 - \mbf{Q}^2\|_{\mbf{H}^2([0,T])}^2 &\leq C_1\frac{\rho^2}{\rho_B^2} \left( \|\mbf{q}^1-\mbf{q}^2\|_{\mbf{H}^2([0,T])}^2 + \|\boldsymbol{\omega}^1-\boldsymbol{\omega}^2\|_{\mbf{H}^1([0,T])}^2 \right)\\
    &+ C_2T^{2s} \left( \|\mbf{q}^1-\mbf{q}^2\|_{\mbf{H}^2([0,T])}^2 + \|\boldsymbol{\omega}^1-\boldsymbol{\omega}^2\|_{\mbf{H}^1([0,T])}^2 \right),
    \end{aligned}
    \label{contraction pf8}
\end{align}
for some $s > 0$.
Here the constant $C_1$ depends only on $\Omega(0)$ and $\mathrm{diam}(B(0))$, but $C_2$ depends on $R$ (which defines $D(\mathscr{F})$) and the initial data.
From here it is almost apparent that if the relative density, $\frac{\rho}{\rho_B}$, is sufficiently small one may choose $T$ to be sufficiently small (determined by the initial data) so that $\mathscr{F}$ is a contraction.
Of course for this to be the case we also require an analogous bound on the difference $\|\mbf{W}^1 - \mbf{W}^2\|_{\mbf{H}^1([0,T])}$.

We now outline how one shows such a bound.
Firstly, it is clear that
\begin{align}
    \frac{\mathrm{d}}{\mathrm{d}t}\left( \mbb{J}^1 \mbf{W}^1 - \mbb{J}^2 \mbf{W}^2 \right) = \int_{\partial B^2(t)} (\mbf{x} - \mbf{q}^2(t)) \times \mbb{T}^2 \boldsymbol{\nu}^2 - \int_{\partial B^1(t)} (\mbf{x} - \mbf{q}^1(t)) \times \mbb{T}^1 \boldsymbol{\nu}^1,
    \label{contraction pf9}
\end{align}
with initial data
\[ (\mbf{W}^1 - \mbf{W}^2)(0) = 0. \]
The idea here is that one can write the integrals on the right-hand side in a similar way to \eqref{contraction pf2}, but we now have the added complication of dealing with two distinct inertial tensors on the left-hand side.
For this we observe that
\begin{align}
\begin{aligned}
\frac{\mathrm{d}}{\mathrm{d}t}\left( \mbb{J}^1 \mbf{W}^1 - \mbb{J}^2 \mbf{W}^2 \right) &= \mbb{J}^1 \frac{\mathrm{d} (\mbf{W}^1 - \mbf{W}^2)}{\mathrm{d}t} + (\mbb{J}^1 - \mbb{J}^2) \frac{\mathrm{d} \mbf{W}^2}{\mathrm{d}t}\\
&+ \frac{\mathrm{d}}{\mathrm{d}t}\mbb{J}^1  (\mbf{W}^1 - \mbf{W}^2) + \frac{\mathrm{d}}{\mathrm{d}t}(\mbb{J}^1 - \mbb{J}^2) \mbf{W}^2,
\end{aligned}
\label{contraction pf10}
\end{align}
and hence it is clear we require bounds on $(\mbb{J}^1 - \mbb{J}^2)$ and $\frac{\mathrm{d}}{\mathrm{d}t}(\mbb{J}^1 - \mbb{J}^2)$.
We recall from the proof of Lemma \ref{lemma: non degeneracy2} that we may write
\[ \mbb{J}^i(t) =  \int_{B(0)} \rho_B \left(|\mbb{O}^i(t)(\mbf{y}-\mbf{q}_0)|^2 \mbb{I} - \mbb{O}^i(t)(\mbf{y}-\mbf{q}_0) \otimes \mbb{O}^i(t)(\mbf{y}-\mbf{q}_0)\right) \]
where $\mbb{O}^i(t)$ is the orthogonal matrix such that
\[\frac{\mathrm{d}}{\mathrm{d}t} \mbb{O}^i(t)\mbf{z}_0 = \boldsymbol{\omega}^i(t) \times \mbb{O}^i(t) \mbf{z}_0,\]
for all $\mbf{z}_0 \in \mbb{R}^3$ and $\mbb{O}^i(0) = \mbb{I}$.
Thus one obtains that
\[ \mbb{J}^1(t) - \mbb{J}^2(t) = \int_{B(0)} \rho_B \left(\mbb{O}^2(t)(\mbf{y}-\mbf{q}_0) \otimes \mbb{O}^2(t)(\mbf{y}-\mbf{q}_0) - \mbb{O}^1(t)(\mbf{y}-\mbf{q}_0) \otimes \mbb{O}^1(t)(\mbf{y}-\mbf{q}_0)\right).  \]
It is now clear that we must bound $\mbb{O}^1(t) - \mbb{O}^2(t)$ which we observe is such that
\[ \frac{\mathrm{d}}{\mathrm{d}t} (\mbb{O}^1(t) - \mbb{O}^2(t))\mbf{z}_0 = (\boldsymbol{\omega}^1(t) - \boldsymbol{\omega}^2(t)) \times \mbb{O}^1(t) \mbf{z}_0 + \boldsymbol{\omega}^2(t) \times (\mbb{O}^1(t) - \mbb{O}^2(t))\mbf{z}_0. \]
Thus applying the Gr\"onwall inequality one obtains
\[ \left| (\mbb{O}^1(t) - \mbb{O}^2(t))\mbf{z}_0 \right| \leq |\mbf{z}_0|\left(\int_0^t |\boldsymbol{\omega}^1(s) - \boldsymbol{\omega}^2(s)| \, \mathrm{d}s \right) \exp\left( \int_0^t |\boldsymbol{\omega}^2(s)| \, \mathrm{d}s\right), \]
and we may use the embedding $\mbf{H}^1([0,T]) \hookrightarrow \mbf{C}^0([0,T])$ so that
\[ \left| (\mbb{O}^1(t) - \mbb{O}^2(t))\mbf{z}_0 \right| \leq Ct|\mbf{z}_0|\|\boldsymbol{\omega}^1 -\boldsymbol{\omega}^2\|_{\mbf{H}^1([0,T])}\exp\left( \int_0^t |\boldsymbol{\omega}^2(s)| \, \mathrm{d}s\right), \]
It is an immediate consequence to observe that
\[ \left| \frac{\mathrm{d}}{\mathrm{d}t} (\mbb{O}^1(t) - \mbb{O}^2(t))\mbf{z}_0  \right| \leq |(\boldsymbol{\omega}^1(t) - \boldsymbol{\omega}^2(t))| |\mbf{z}_0| + Ct|\mbf{z}_0|\|\boldsymbol{\omega}^1 -\boldsymbol{\omega}^2\|_{\mbf{H}^1([0,T])}\exp\left( \int_0^t |\boldsymbol{\omega}^2(s)| \, \mathrm{d}s\right). \]
These readily translate into bounds for $(\mbb{J}^1 - \mbb{J}^2)$ and $\frac{\mathrm{d}}{\mathrm{d}t}(\mbb{J}^1 - \mbb{J}^2)$.

We now use \eqref{contraction pf9} and \eqref{contraction pf10} to see that
\begin{equation}
\begin{aligned}
    \frac{\mathrm{d} (\mbf{W}^1 - \mbf{W}^2)}{\mathrm{d}t} &=  -(\mbb{J}^1)^{-1}(\mbb{J}^1 - \mbb{J}^2) \frac{\mathrm{d} \mbf{W}^2}{\mathrm{d}t} - (\mbb{J}^1)^{-1}\frac{\mathrm{d}}{\mathrm{d}t}\mbb{J}^1  (\mbf{W}^1 - \mbf{W}^2)\\
    &- (\mbb{J}^1)^{-1}\frac{\mathrm{d}}{\mathrm{d}t}(\mbb{J}^1 - \mbb{J}^2) \mbf{W}^2 + (\mbb{J}^1)^{-1} \int_{\partial B^2(t)} (\mbf{x} - \mbf{q}^2(t)) \times \mbb{T}^2 \boldsymbol{\nu}^2\\
    &- (\mbb{J}^1)^{-1}\int_{\partial B^1(t)} (\mbf{x} - \mbf{q}^1(t)) \times \mbb{T}^1 \boldsymbol{\nu}^1, 
\end{aligned}
\label{contraction pf11}
\end{equation}
where we recall from the proof of Lemma \ref{lemma: non degeneracy2} that $(\mbb{J}^1)^{-1}$ is uniformly bounded in $\mbf{L}^\infty([0,T])$.
It is clear how most of these terms are controlled, but we briefly explain how one controls the term involving $\frac{\mathrm{d}}{\mathrm{d}t}(\mbb{J}^1 - \mbb{J}^2)$.
Firstly we use the $\mbf{L}^\infty([0,T])$ bound on $(\mbb{J}^1)^{-1}$ to see that
\begin{align*}
    \int_0^T \left| (\mbb{J}^1)^{-1}\frac{\mathrm{d}}{\mathrm{d}t}(\mbb{J}^1 - \mbb{J}^2) \mbf{W}^2 \right|^2 &\leq C\int_0^T \left|\frac{\mathrm{d}}{\mathrm{d}t}(\mbb{J}^1 - \mbb{J}^2) \mbf{W}^2\right|^2\\
    &\leq C\int_0^T \left|\frac{\mathrm{d}}{\mathrm{d}t}(\mbb{J}^1 - \mbb{J}^2)\right|^2,
\end{align*}
where we have used the fact that $\mbf{W}^2$ is uniformly bounded in $\mbf{H}^1([0,T])$ (due to Lemma \ref{lemma: non degeneracy2}) and the embedding $\mbf{H}^1([0,T]) \hookrightarrow \mbf{C}^0([0,T])$ for the second inequality.
One then uses the above bound for $\frac{\mathrm{d}}{\mathrm{d}t} (\mbb{O}^1(t) - \mbb{O}^2(t))$ to deduce a bound for $\frac{\mathrm{d}}{\mathrm{d}t} (\mbb{J}^1(t) - \mbb{J}^2(t))$ so that one obtains a bound of the form
\begin{align*}
    \int_0^T \left|\frac{\mathrm{d}}{\mathrm{d}t}(\mbb{J}^1 - \mbb{J}^2)\right|^2 &\leq C \int_0^T \left(|\boldsymbol{\omega}^1 - \boldsymbol{\omega}^2|^2 + t \|\boldsymbol{\omega}^1 - \boldsymbol{\omega}^2\|_{\mbf{H}^1([0,T])}^2\right)\\
    & \leq CT \|\boldsymbol{\omega}^1 - \boldsymbol{\omega}^2\|_{\mbf{H}^1([0,T])}^2,
\end{align*}
where we have once again used the embedding $\mbf{H}^1([0,T]) \hookrightarrow \mbf{C}^0([0,T])$.
One proceeds in this manner for the other terms, and controls the integral terms in \eqref{contraction pf11} in the same manner as we did for the bounds on $\mbf{Q}^1 - \mbf{Q}^2$.
After repeating these arguments one ultimately obtains a bound of the form
\begin{align}
\begin{aligned}
    \|\mbf{W}^1 - \mbf{W}^2\|_{\mbf{H}^1([0,T])}^2 &\leq C_3\frac{\rho^2}{\rho_B^2} \left( \|\mbf{q}^1-\mbf{q}^2\|_{\mbf{H}^2([0,T])}^2 + \|\boldsymbol{\omega}^1-\boldsymbol{\omega}^2\|_{\mbf{H}^1([0,T])}^2 \right)\\
    &+ C_4 T^{2s} \left( \|\mbf{q}^1-\mbf{q}^2\|_{\mbf{H}^2([0,T])}^2 + \|\boldsymbol{\omega}^1-\boldsymbol{\omega}^2\|_{\mbf{H}^1([0,T])}^2 \right),
    \end{aligned}
    \label{contraction pf12}
\end{align}
for some $s > 0$, where $C_3$ depends only on $\Omega(0)$ and $\mathrm{diam}(B(0))$, and $C_3$ depends on $R$ and the initial data.
By combining \eqref{contraction pf8} and \eqref{contraction pf12} is clear that if the relative density, $\frac{\rho}{\rho_B}$, is sufficiently small in terms of the geometry of the problem, and $T$ is sufficiently small in terms of the initial data, then the map $\mathscr{F}$ is a contraction.
\end{proof}

\begin{proof}[Proof of Theorem \ref{thm: FSI existence}]
It is straightforward to show that $D(\mathscr{F})$ is a complete metric space, and hence using the Banach fixed point theorem, one finds that~$\mathscr{F}$ admits a unique fixed point, $(\mbf{q}, \boldsymbol{\omega}) \in D(\mathscr{F})$.
Defining $(\mbf{u},p)$ to be the unique strong solution to \eqref{eqn: evolving NS} defined by the fixed point, $(\mbf{q}, \boldsymbol{\omega})$, we find that the quadruple $(\mbf{q}, \boldsymbol{\omega}, \mbf{u}, p)$ is a strong solution to \eqref{eqn: FSI fluid}, \eqref{eqn: FSI rigid body} in the sense of Definition \ref{defn: strong soln fsi}.
Since any strong solution (in the sense of Definition \ref{defn: strong soln fsi}) to \eqref{eqn: FSI fluid}, \eqref{eqn: FSI rigid body} yields a fixed point of $\mathscr{F}$, and $\mathscr{F}$ has a unique fixed point, it is clear that such a strong solution is unique.
All in all we have shown Theorem \ref{thm: FSI existence}.
\end{proof}

\begin{remark}
We have shown that strong solutions, in the sense of \eqref{defn: strong soln fsi}, are unique under suitable assumptions on $T$ and $\frac{\rho}{\rho_B}$.
In the literature there are stronger uniqueness results for \eqref{eqn: FSI fluid}, \eqref{eqn: FSI rigid body}.
For example in \cite{muha2021uniqueness} a weak-strong uniqueness result is proven for strong solutions satisfying a Prodi--Serrin condition.
Similar weak-strong uniqueness results have also been established for related problems in \cite{chemetov2019weak,disser2016inertial,glass2015uniqueness,kreml2020weak,maity2024uniqueness}.
Note however that the choice of $\Phi$ defining the function spaces for the fluid quantities in Theorem \ref{thm: FSI existence} is clearly not unique, and depends on the choice of extension of the boundary data.
\end{remark}

\subsection{An alternate iteration}
\label{subsection: alternate iteration}

We now briefly discuss an alternate approach to our iterative argument which may be useful in the analysis of numerical methods, as evaluating the boundary integrals in~\eqref{eqn: updated position} and~\eqref{eqn: updated angular} may be problematic for the numerical analysis.
As such one may want to reformulate this iteration using only integral in the bulk domain, $\Omega(t)$.
For this we choose $\zeta(t): \Omega(t) \rightarrow \mbb{R}$ to be a sufficiently smooth function such that $\zeta(t) = 1$ on $\partial B(t)$ and $\zeta(t) = 0$ on $\Gamma$.
One such example of $\zeta(t)$ is given by the harmonic extension of this boundary data.
For this one realizes that the boundary integral in~\eqref{eqn: updated position} is
\[ \int_{\partial B(t)} \mbb{T} \boldsymbol{\nu} = \int_{\partial \Omega(t)} \zeta \mbb{T} \boldsymbol{\nu} = \int_{\Omega(t)} \Div \left( \zeta \mbb{T} \right) = \int_{\Omega(t)} (\grad \zeta)^T \mbb{T} + \int_{\Omega(t)} \zeta \Div \mbb{T} . \]
Recalling that $\Div \mbb{T} = \rho \left( \ddt{\mbf{u}} + (\mbf{u} \cdot \grad) \mbf{u}\right)$, one can now rewrite this final integral by observing, from Lemma~\ref{lemma: reynolds transport theorem} and the fact that $\mbf{V} = \mbf{u}$ on $\partial \Omega(t)$,
\begin{align*}
    \rho \frac{\mathrm{d}}{\mathrm{d}t} \int_{\Omega(t)} \zeta \mbf{u} = \rho \int_{\Omega(t)} \zeta \left( \ddt{\mbf{u}} + (\mbf{u} \cdot \grad) \mbf{u} \right)  + \rho \int_{\Omega(t)} \ddt{\zeta} \mbf{u} + (\mbf{u} \cdot \grad \zeta) \mbf{u}.
\end{align*}
Hence one can alternatively write~\eqref{eqn: updated position} as
\begin{multline*}
    m \frac{\mathrm{d} \mbf{Q}}{\mathrm{d}t} = m \mbf{v}_0 - \rho \int_{\Omega(t)} \zeta \mbf{u} + \rho \int_{\Omega(0)} \zeta(0) \mbf{u}_0 + \rho \int_0^t \int_{\Omega(s)} \ddt{\zeta} \mbf{u} + (\mbf{u} \cdot \grad \zeta) \mbf{u} \, \mathrm{d}s\\
    - \int_0^t \int_{\Omega(s)} (\grad \zeta)^T \mbb{T} \, \mathrm{d}s.
\end{multline*}
One may perform similar calculations to rewrite~\eqref{eqn: updated angular} as
\begin{align*}
    \frac{\mathrm{d}}{\mathrm{d}t} (\mbb{J} \mbf{W}) = -\int_{\Omega(t)} (\grad \zeta)^T (\mbf{x} - \mbf{q}) \times \mbb{T} + \int_{\Omega(t)} \zeta (\grad \times \mbb{T}) \cdot (\mbf{x} - \mbf{q}).
\end{align*}

\section{A numerical implementation}
\label{section: FSI numerics}
\subsection{An evolving finite element method}
The novelty of our proof of Theorem~\ref{thm: FSI existence} is in its constructive nature.
As such it is natural to use this construction to influence the design of numerical methods solving \eqref{eqn: FSI fluid}, \eqref{eqn: FSI rigid body}.
We now consider such a numerical method, by discretising in space and time, and using evolving finite element methods, cf.~\cite{elliott2021unified}, for which we shall denote our timestep size as $\tau$ and our spatial mesh size as $h$.
We will discretise in space using $\underline{\mbb{P}^2}$--$\mbb{P}^1$ finite elements, cf.~\cite[Section 8.8]{Boffi13Book}), for the velocity-pressure pairs where the underline denotes vector-valued elements.
Similar problems have been studied by using a semi-Lagrangian approach as in
~\cite{Danilov2017NavierStokesMovingDomain}, an Eulerian approach as in~\cite{Burman2020NitscheFSI,Burman2022StokesTimeDependentDomains,Neilan2024EulerianFEMNavierStokes,vonWahl2022StokesMovingDomains}, or an arbitrary Lagrangian--Eulerian approach as in~\cite{Gao2025ALEMethodForFSI,Liu2013ALEStokes,Rao2025ALEStokes} --- we shall take this latter approach.

For the following numerical method we assume that one begins with data 
\[(\mbf{q}^k_{n,h})_{n=0,\ldots,N_T} \subset \mbb{R}^3, \quad (\mbf{v}^k_{n,h})_{n=0,\ldots,N_T} \subset \mbb{R}^3 , \quad (\boldsymbol{\omega}^k_{n,h})_{n=0,\ldots,N_T} \subset \mbb{R}^3, \quad \mbf{u}_{0,h} \in \underline{\mbb{P}^2}(\Omega_{0,h}),\]
which are approximations of the centre of mass, velocity of the rigid body, angular momentum, and initial fluid velocity.
Here $\Omega_{0,h}$ is some triangulated domain approximating the known domain $\Omega(0)$, which will be used as the initial domain in each iteration, i.e. $\Omega_{0,h}^k = \Omega_{0,h}^{k-1} = \ldots = \Omega_{0,h}$, and we shall denote the nodes of $\Omega_h(0)$ as $(\mbf{x}^0_{0,i})_{i=1,\ldots, N_h} \subset \mbb{R}^3$.
Likewise we denote the number of timesteps at $N_T := T/\tau$ and the number of degrees of freedom as $N_h$.
We note that the superscript, $(\cdot)^k$, refers to the iteration we are at, and the subscripts $(\cdot)_{n,h}$ refer to the timestep $n$, and the dependence on the mesh size $h$.

Given this data the algorithm is as follows:
\begin{enumerate}
    \item Given some known spatial domain, $\Omega_{n-1,h}^k$, one must compute the next spatial domain, $\Omega_{n,h}^k$, which in turn requires computation of a discrete harmonic extension, $\mbf{V}_{n-1,h}^k \in \underline{\mbb{P}^2}(\Omega_{n-1,h}^k)$.
    This is computed as the $\underline{\mbb{P}^2}(\Omega_{n-1,h}^k)$ finite element solution of the vector-valued Laplace problem:
    \begin{equation}
        \begin{cases}
			-\Delta \mbf{V}_{n-1,h}^k = 0, &\text{ on }\ \Omega_{n-1,h}^k,\\
			\mbf{V}_{n-1,h}^k = \mbf{v}^k_{n-1} + \boldsymbol{\omega}^k_{n-1} \times (\mbf{x} - \mbf{q}^k_{n-1}), &\text{ on } \partial B_{n-1,h}^k,\\
			\mbf{V}_{n-1,h}^k = 0, &\text{ on } \Gamma_h,
		\end{cases}
        \label{eqn: finite element harmonic extension}
    \end{equation}
    where here spatial domain has disjoint boundary $\partial \Omega_{n-1,h}^k = \Gamma_h \cup \partial B_{n-1,h}^k$, where $\Gamma_h$ denotes the fixed boundary and $\partial B_{n-1,h}^k$ the moving boundary.
    One computes a finite element approximation of the pushforward map, $\Phi^k_{n,h} \in \underline{\mbb{P}^2}(\Omega_{n-1,h}^k)$, by approximating~\eqref{eqn: parametrisation defn} as
    \begin{equation}
        \frac{\Phi^k_{n,h} - \mbf{x}}{\tau} = \mbf{V}_{n-1,h}^k,
        \label{eqn: finite element nodal update}
    \end{equation}
    where $\mbf{x}$ is the identity on $\Omega_{n-1,h}^k$.
    One now updates spatial domain via $\Omega_{n,h}^k := \Phi_{n,h}^k(\Omega_{n-1,h}^k)$, i.e. we define new nodes $\mbf{x}^k_{n,i} := \Phi_{n,h}^k(\mbf{x}^k_{n-1,i})$ for all $i \in \{1,\ldots, N_h\}$.
    \item We now solve the Navier--Stokes equations on $\Omega_{n,h}^k$.
    That is to say, we look for functions $(\mbf{u}_{n,h}^k, p_{n,h}^k) \in \underline{\mbb{P}^2}(\Omega_{n,h}^k) \times \mbb{P}^1(\Omega_{n,h}^k)$ solving the saddle-point problem
    \begin{subequations}
        \begin{gather}
        \begin{aligned}
            \frac{\rho}{\tau} \left(\int_{\Omega_{n,h}^k} \mbf{u}_{n,h}^k \cdot \boldsymbol{\phi}_{n,h}^k - \int_{\Omega_{n-1,h}^k} \mbf{u}_{n-1,h}^k \cdot \boldsymbol{\phi}_{n-1,h}^k \right) - \rho \int_{\Omega_{n,h}^k} (\mbf{u}_{n,h}^k \cdot \grad \boldsymbol{\phi}_{n,h}^k) \mbf{u}_{n,h}^k\\
            + \rho \int_{\Omega_{n,h}^k} (\mbf{V}_{n,h}^k \cdot \grad \boldsymbol{\phi}_{n,h}^k) \mbf{u}_{n,h}^k + \mu \int_{\Omega_h} \grad \mbf{u}_{n,h}^k : \grad \boldsymbol{\phi}_{n,h}^k = 0,
        \end{aligned}\label{eqn: finite element navier stokes1}\\
            \int_{\Omega_{n,h}^k} \psi_{n,h}^k \Div \mbf{u}_{n,h}^k = 0, \label{eqn: finite element navier stokes2}
        \end{gather}
        \label{eqn: finite element navier stokes}
    \end{subequations}
    for all $(\boldsymbol{\phi}_{n,h}^k, \psi_{n,h}^k) \in \underline{\mbb{P}^2}(\Omega_{n,h}^k) \times \mbb{P}^1(\Omega_{n,h}^k)$ and such that
    \begin{equation}
        \begin{gathered}
            \mbf{u}_{n,h}^k = \mbf{v}^k_{n} + \boldsymbol{\omega}^k_{n} \times (\mbf{x} - \mbf{q}^k_{n}), \quad \text{on } \partial B_{n,h}^k,\\
            \mbf{u}_{n,h}^k = 0, \quad \text{on } \Gamma_h.\\
        \end{gathered}
        \label{eqn: finite element boundary conditions}
    \end{equation}
    Here $\boldsymbol{\phi}_{n-1,h}^k \in \underline{\mbb{P}^2}(\Omega_{n-1,h}^k)$ denotes the finite element function with the same nodal values as the test function $\boldsymbol{\phi}_{n,h}^k \in \underline{\mbb{P}^2}(\Omega_{n,h}^k)$ but defined on the previous spatial domain, $\Omega_{n-1,h}^k$.
    \item Finally we create new guesses for the centre of mass, rigid body velocity, and angular momentum
    \[(\mbf{q}^{k+1}_{n,h})_{n=0,\ldots,N_T} \subset \mbb{R}^3, \quad (\mbf{v}^k_{n,h})_{n=0,\ldots,N_T} \subset \mbb{R}^3 , \quad (\boldsymbol{\omega}^{k+1}_{n,h})_{n=0,\ldots,N_T} \subset \mbb{R}^3,\]
    by approximating~\eqref{eqn: updated position} and~\eqref{eqn: updated angular} by
    \begin{subequations}
    \label{eqn: finite element ODEs}
        \begin{gather}
        m\left(\frac{\mbf{q}_{n,h}^{k+1} - \mbf{q}_{n-1,h}^{k+1}}{\tau}\right) = \mbf{v}_{n,h}^{k+1}, \label{eqn: finite element mass ODE}\\
        \frac{\mbf{v}_{n,h}^{k+1} - \mbf{v}_{n-1,h}^{k+1}}{\tau} = -\int_{\partial B_{n,h}^k} \mbb{T}_{n,h}^k \boldsymbol{\nu}_{n,h}^k, \label{eqn: finite element velocity ODE}\\
        \frac{\mbb{J}_{n,h}^{k}\boldsymbol{\omega}_{n,h}^{k+1} - \mbb{J}_{n-1,h}^k\boldsymbol{\omega}_{n-1,h}^{k+1}}{\tau} = -\int_{\partial B_{n,h}^k} (\mbf{x} - \mbf{q}_{n,h}^k) \times \mbb{T}_{n,h}^k \boldsymbol{\nu}_{n,h}^k, \label{eqn: finite element angular ODE}
    \end{gather}
    \end{subequations}
    where
    \begin{gather*}
        \mbb{T}_{n,h}^k := -p_{n,h}^k\mbb{I} + \mu\left( \grad \mbf{u}_{n,h}^k + (\grad \mbf{u}_{n,h}^k)^T \right),\\
        \mbb{J}_{n,h}^k := \int_{B_{n,h}^k} \rho_B \left( |\mbf{x} - \mbf{q}_{n,h}^k|^2 \mbb{I} - (\mbf{x} - \mbf{q}_{n,h}^k) \otimes (\mbf{x} - \mbf{q}_{n,h}^k) \right).
    \end{gather*}
    Note that here the initial conditions known since
    \[ \mbf{q}^{k+1}_{0,h} = \mbf{q}^{k}_{0,h}, \quad \mbf{v}^{k+1}_{0,h} = \mbf{v}^{k}_{0,h}, \quad \boldsymbol{\omega}^{k+1}_{0,h} = \boldsymbol{\omega}^{k}_{0,h}. \]
    \item Repeat this iteration until some desired stopping criteria.
\end{enumerate}

\begin{remark}
    \label{remark: finite element advective term}
    The advective terms appearing in~\eqref{eqn: finite element navier stokes1} appear from the use of Reynold's transport theorem since, at the continuous level,
    \begin{align*}
    \frac{\mathrm{d}}{\mathrm{d}t} \int_{\Omega(t)} \mbf{u} \cdot \boldsymbol{\phi} &= \int_{\Omega(t)} \ddt{\mbf{u}} \cdot \boldsymbol{\phi} + \int_{\Omega(t)} {\mbf{u}} \cdot \ddt{\boldsymbol{\phi}} + \int_{\partial \Omega(t)} \mbf{u} \cdot \boldsymbol{\phi} (\mbf{V} \cdot \boldsymbol{\nu})\\
    &= \int_{\Omega(t)} \ddt{\mbf{u}} \cdot \boldsymbol{\phi} +  \int_{\Omega(t)} {\mbf{u}} \cdot \ddt{\boldsymbol{\phi}} + \int_{\partial \Omega(t)} \mbf{u} \cdot \boldsymbol{\phi} (\mbf{u} \cdot \boldsymbol{\nu})\\
    &= \int_{\Omega(t)} \ddt{\mbf{u}} \cdot \boldsymbol{\phi} + \int_{\Omega(t)} (\mbf{u} \cdot \grad \mbf{u}) \boldsymbol{\phi} + \int_{\Omega(t)} (\mbf{u} \cdot \grad \boldsymbol{\phi}) \mbf{u} - \int_{\Omega(t)} (\mbf{V}\cdot \grad \boldsymbol{\phi}) \mbf{u},
    \end{align*}
    for all functions $\boldsymbol{\phi}$ of the form $\boldsymbol{\phi} = \boldsymbol{\phi}_0 \circ \Phi(t)^{-1}$ for functions $\boldsymbol{\phi}_0$ defined on $\Omega(0)$ where we have used that $\matdev_{\mbf{V}} \boldsymbol{\phi} = 0$ and hence $\ddt{\boldsymbol{\phi}} = -\mbf{V} \cdot \grad \boldsymbol{\phi}$.
    This is precisely the form of the finite element basis functions we consider.
\end{remark}

\begin{remark}
    \label{remark: finite element advection implementation}
    For practical numerical methods one may alternatively write the quasilinear term, $\int_{\Omega_{n,h}^k} (\mbf{u}_{n,h}^k \cdot \grad \boldsymbol{\phi}_{n,h}^k) \mbf{u}_{n,h}^k$, as $\frac{1}{2}\int_{\Omega_{n,h}^k} (\mbf{u}_{n,h}^k \cdot \grad \boldsymbol{\phi}_{n,h}^k) \mbf{u}_{n,h}^k - \int_{\Omega_{n,h}^k} (\mbf{u}_{n,h}^k \cdot \grad \mbf{u}_{n,h}^k)\boldsymbol{\phi}_{n,h}^k$ to preserve the anti-symmetric property of the associated trilinear form.
    We refer the reader to~\cite{John2016FEMforIncompressibleFlow} for related discussion.
\end{remark}

Notice that the above algorithm requires one to store the data for the mesh, the initial data, and the rigid body data: $(\mbf{q}^k_{n,h})_{n=0,\ldots,N_T}$, $(\mbf{v}^k_{n,h})_{n=0,\ldots,N_T}$, and $(\boldsymbol{\omega}^k_{n,h})_{n=0,\ldots,N_T}$.
If $\tau$ is taken to be sufficiently small then this may be very memory-intensive, and so an alternative method would be to perform these iterations over each timestep, for some fixed amount of iterations $k_{\max}$, so that one only requires the data from the previous timestep.
We illustrate the dependencies of these methods in Figure~\ref{fig: scheme1 dependencies} and Figure~\ref{fig: scheme2 dependencies}.
Notice that one benefit of the method we have discussed above is that it allows for parallel computation, as illustrated in Figure~\ref{fig: scheme1 dependencies}.

\begin{figure}[ht]
	\centering
	\begin{tikzpicture}
		\draw[->,thick] (0,1.5) to (8, 1.5);
		\node at (4, 2) {Timestep $n \rightarrow N_T$};
		\draw[->, thick] (-1.5,0) to (-1.5, -8);
		\node at (-2,-4) {{\rotatebox{90}{Iteration $k \rightarrow \infty$}}};
		
		\draw[color = black](0,0) circle(1);
		\node at (0, 0.2) {$k=1$};
		\node at (0, -0.2) {$n=1$};
		\draw[->,thick] (1,0) to (2,0);
		\draw[color = black](3,0) circle(1);
		\node at (3, 0.2) {$k=1$};
		\node at (3, -0.2) {$n=2$};
		\draw[->,thick] (4,0) to (5,0);
		\draw[color = black](6,0) circle(1);
		\node at (6, 0.2) {$k=1$};
		\node at (6, -0.2) {$n=3$};
		\draw[->,thick] (7,0) to (8,0);
		\node at (8.5,0) {\ldots};
		\draw[color = black](0,-3) circle(1);
		\node at (0, -2.8) {$k=2$};
		\node at (0, -3.2) {$n=1$};
		\draw[->,thick] (1,-3) to (2,-3);
		\draw[color = black](3,-3) circle(1);
		\node at (3, -2.8) {$k=2$};
		\node at (3, -3.2) {$n=2$};
		\draw[->,thick] (4,-3) to (5,-3);
		\draw[color = black](6,-3) circle(1);
		\node at (6, -2.8) {$k=2$};
		\node at (6, -3.2) {$n=3$};
		\draw[->,thick] (7,-3) to (8,-3);
		\node at (8.5,-3) {\ldots};
		\draw[color = black](0,-6) circle(1);
		\node at (0, -5.8) {$k=3$};
		\node at (0, -6.2) {$n=1$};
		\draw[->,thick] (1,-6) to (2,-6);
		\draw[color = black](3,-6) circle(1);
		\node at (3, -5.8) {$k=3$};
		\node at (3, -6.2) {$n=2$};
		\draw[->,thick] (4,-6) to (5,-6);
		\draw[color = black](6,-6) circle(1);
		\node at (6, -5.8) {$k=3$};
		\node at (6, -6.2) {$n=3$};
		\draw[->,thick] (7,-6) to (8,-6);
		\node at (8.5,-6) {\ldots};
		\node at (0,-7.5) {\rotatebox{90}{\ldots}};
		\node at (3,-7.5) {\rotatebox{90}{\ldots}};
		\node at (6,-7.5) {\rotatebox{90}{\ldots}};
		\draw[->, thick] (0,-1) to (0,-2);
		\draw[->, thick] (3,-1) to (3,-2);
		\draw[->, thick] (6,-1) to (6,-2);
		\draw[->, thick] (0,-4) to (0,-5);
		\draw[->, thick] (3,-4) to (3,-5);
		\draw[->, thick] (6,-4) to (6,-5);
	\end{tikzpicture}
	\caption{Diagram showing the dependencies for the presented method.
    Notice that quantities on a diagonal line may be computed in parallel, e.g. $(k=1, n=2)$ and $(k=2, n=1)$.}
    \label{fig: scheme1 dependencies}
\end{figure}
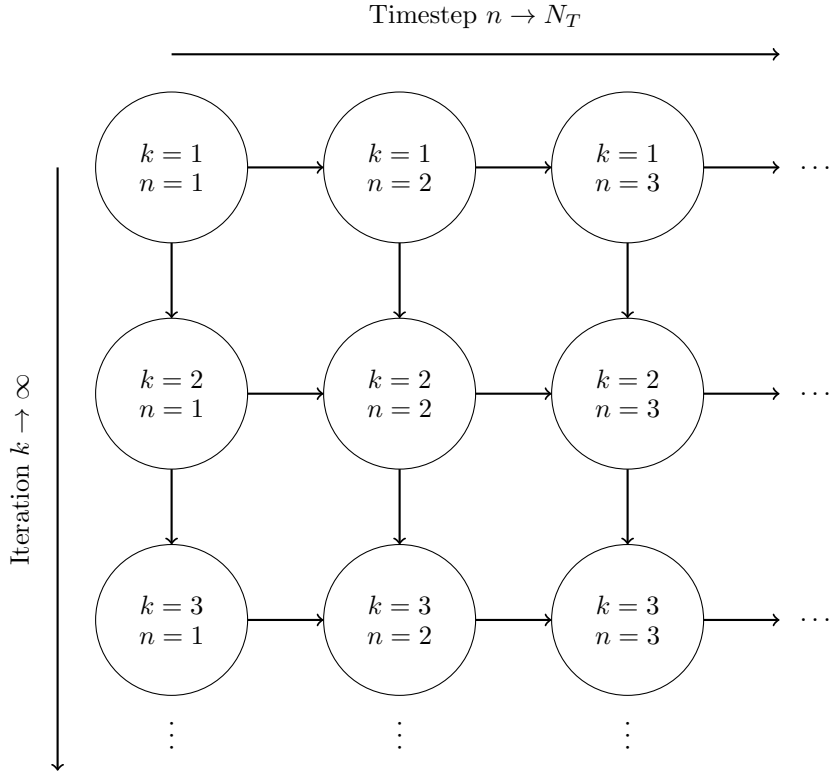

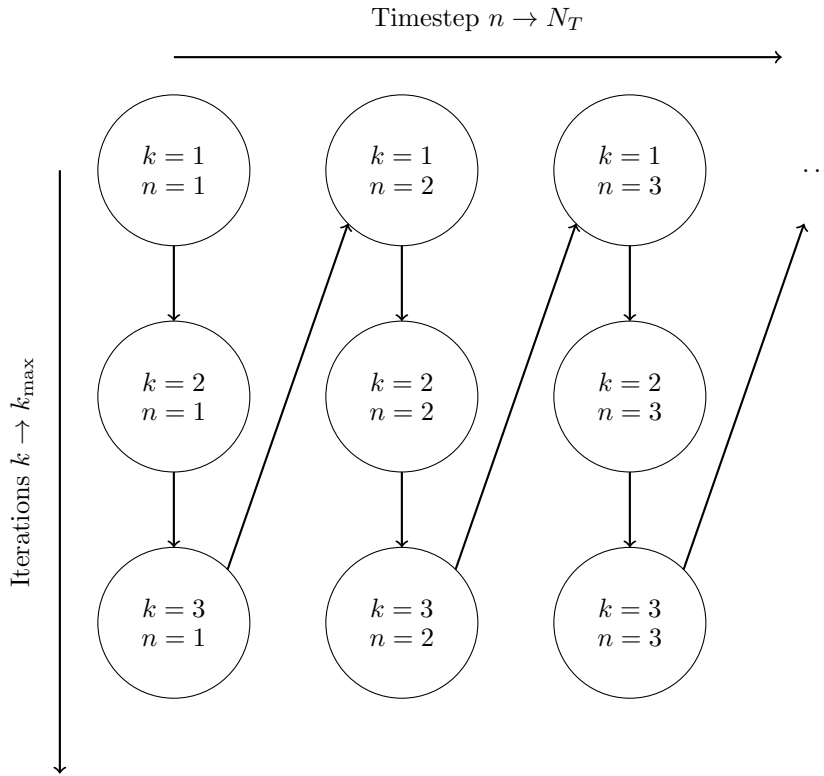
\begin{figure}[ht]
	\centering
	\begin{tikzpicture}
		\draw[->,thick] (0,1.5) to (8, 1.5);
		\node at (4, 2) {Timestep $n \rightarrow N_T$};
		\draw[->, thick] (-1.5,0) to (-1.5, -8);
		\node at (-2,-4) {{\rotatebox{90}{Iterations $k \rightarrow k_{\max}$}}};
		\draw[color = black](0,0) circle(1);
		\node at (0, 0.2) {$k=1$};
		\node at (0, -0.2) {$n=1$};
		\draw[color = black](3,0) circle(1);
		\node at (3, 0.2) {$k=1$};
		\node at (3, -0.2) {$n=2$};
		\draw[color = black](6,0) circle(1);
		\node at (6, 0.2) {$k=1$};
		\node at (6, -0.2) {$n=3$};
		\node at (8.5,0) {\ldots};
		\draw[color = black](0,-3) circle(1);
		\node at (0, -2.8) {$k=2$};
		\node at (0, -3.2) {$n=1$};
		\draw[color = black](3,-3) circle(1);
		\node at (3, -2.8) {$k=2$};
		\node at (3, -3.2) {$n=2$};
		\draw[color = black](6,-3) circle(1);
		\node at (6, -2.8) {$k=2$};
		\node at (6, -3.2) {$n=3$};
		\draw[color = black](0,-6) circle(1);
		\node at (0, -5.8) {$k=3$};
		\node at (0, -6.2) {$n=1$};
		\draw[color = black](3,-6) circle(1);
		\node at (3, -5.8) {$k=3$};
		\node at (3, -6.2) {$n=2$};
		\draw[color = black](6,-6) circle(1);
		\node at (6, -5.8) {$k=3$};
		\node at (6, -6.2) {$n=3$};
		\draw[->,thick] (0,-1) to (0, -2);
		\draw[->,thick] (0,-4) to (0, -5);
		\draw[->,thick] (3,-1) to (3, -2);
		\draw[->,thick] (3,-4) to (3, -5);
		\draw[->,thick] (6,-1) to (6, -2);
		\draw[->,thick] (6,-4) to (6, -5);
		\draw[->,thick] (0.707,-5.293) to (2.293, -0.707);
		\draw[->,thick] (3.707,-5.293) to (5.293, -0.707);
		\draw[->,thick] (6.707,-5.293) to (8.293, -0.707);
	\end{tikzpicture}
	\caption{Diagram showing the dependencies for the less memory-intensive method.
    Notice that the dependencies require this method to be computed in serial.}
	\label{fig: scheme2 dependencies}
\end{figure}

\subsection{Example with a ball falling under gravity}
We illustrate this numerical method with some two dimensional examples considering a ball, with radius $r = 0.1$, initially at rest in a fluid (with density $\rho = 1$ and viscosity $\mu = 1$) which is also initially at rest, moving under the influence of gravity.
Our finite element method was implemented using Firedrake \cite{rathgeber2016firedrake}, and we choose a timestep size $\tau = 5 \cdot 10^{-4}$ with a mesh size $h \approx 6.284 \cdot 10^{-2}$.
We choose our initial guess for the motion of the ball to be given by the equations
\[ \mbf{q}(t) = -\frac{9.8t^2}{2}\begin{pmatrix}
    0\\
    1
\end{pmatrix},
\quad 
\mbf{v}(t) = -9.8t\begin{pmatrix}
    0\\
    1
\end{pmatrix} \qquad \omega(t) = 0, \]
which one obtains in the case when there is no fluid.
Note that the angular velocity is a scalar quantity since we are in two dimensions.
We consider two examples: one where the ball is very dense ($\rho_B=\frac{200}{\pi}$) so that the relative density is small and Lemma \ref{lemma: fsi contraction} should hold, and one where the ball is less dense $(\rho_B = \frac{10}{\pi})$ to investigate the necessity of the smallness assumption on $\frac{\rho}{\rho_B}$.
For the ball with density $\rho_B = \frac{200}{\pi}$ we provide some numerical evidence of convergence in Table \ref{table: fsi convergence}, and show plots of the magnitude of the fluid in Figure \ref{fig: fsi plots}.
For the ball with density $\rho_B = \frac{10}{\pi}$ we provide some numerical evidence indicating a failure of convergence in Table~\ref{table: fsi failure}, Table~\ref{table: fsi failure refined}, and Figure~\ref{fig: fsi nonconvergence}.
\begin{table}[ht]
    \centering
    \begin{tabular}{|c|c|c|c|}
    \hline
        Iteration & $\mbf{q}$(T) & $\mbf{v}(T)$ & $\omega(T)$\\
        \hline
        0 &  $(0.5,0.4566)$ & $(-1.061\cdot 10^{-4},-8.144\cdot 10^{-1})$ & $-3.523\cdot 10^{-4}$\\
        \hline
        1 &  $(0.5,0.4561)$ & $(-8.824\cdot 10^{-5},-8.358\cdot 10^{-1})$ & $8.886\cdot 10^{-4}$\\
        \hline
        2 &  $(0.5,0.4561)$ & $(-8.842\cdot 10^{-5},-8.338\cdot 10^{-1})$ & $6.804\cdot 10^{-4}$\\
        \hline
        3 &  $(0.5,0.4561)$ & $(-8.861\cdot 10^{-5},-8.34\cdot 10^{-1})$ & $7.107\cdot 10^{-4}$\\
        \hline
        4 &  $(0.5,0.4561)$ & $(-8.861\cdot 10^{-5},-8.339\cdot 10^{-1})$ & $7.068\cdot 10^{-4}$\\
        \hline
    \end{tabular}
    \caption{Table of values of the position, velocity, and angular velocity of the ball (with density $\rho_B = \frac{200}{\pi}$) at time $T = 0.1$ for each iteration.}
    \label{table: fsi convergence}
\end{table}

\begin{figure}[p]
    \centering
    \begin{subfigure}[t]{0.3\linewidth}
        \centering
        \includegraphics[width =\linewidth]{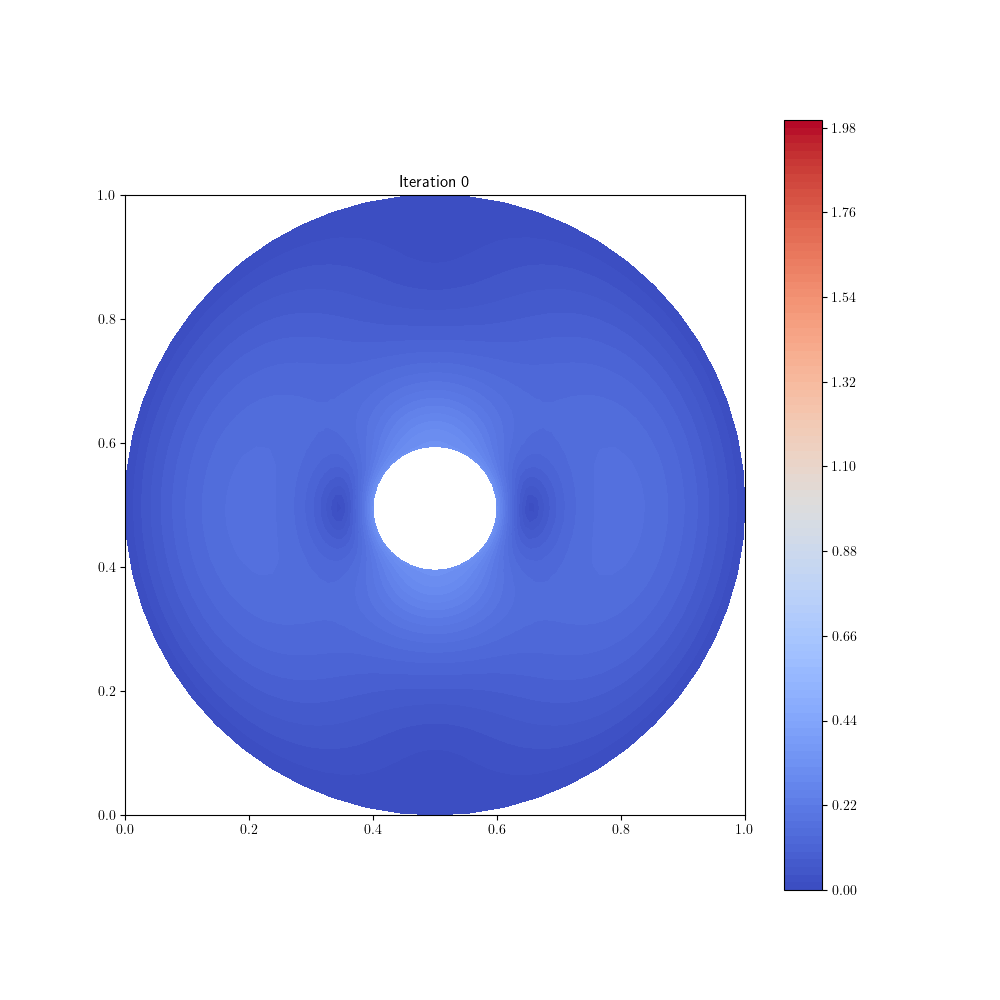}
        \caption{Iteration 0\\
        $t = 0.033$.}
    \end{subfigure}%
    ~
    \begin{subfigure}[t]{0.3\linewidth}
        \centering
        \includegraphics[width =\linewidth]{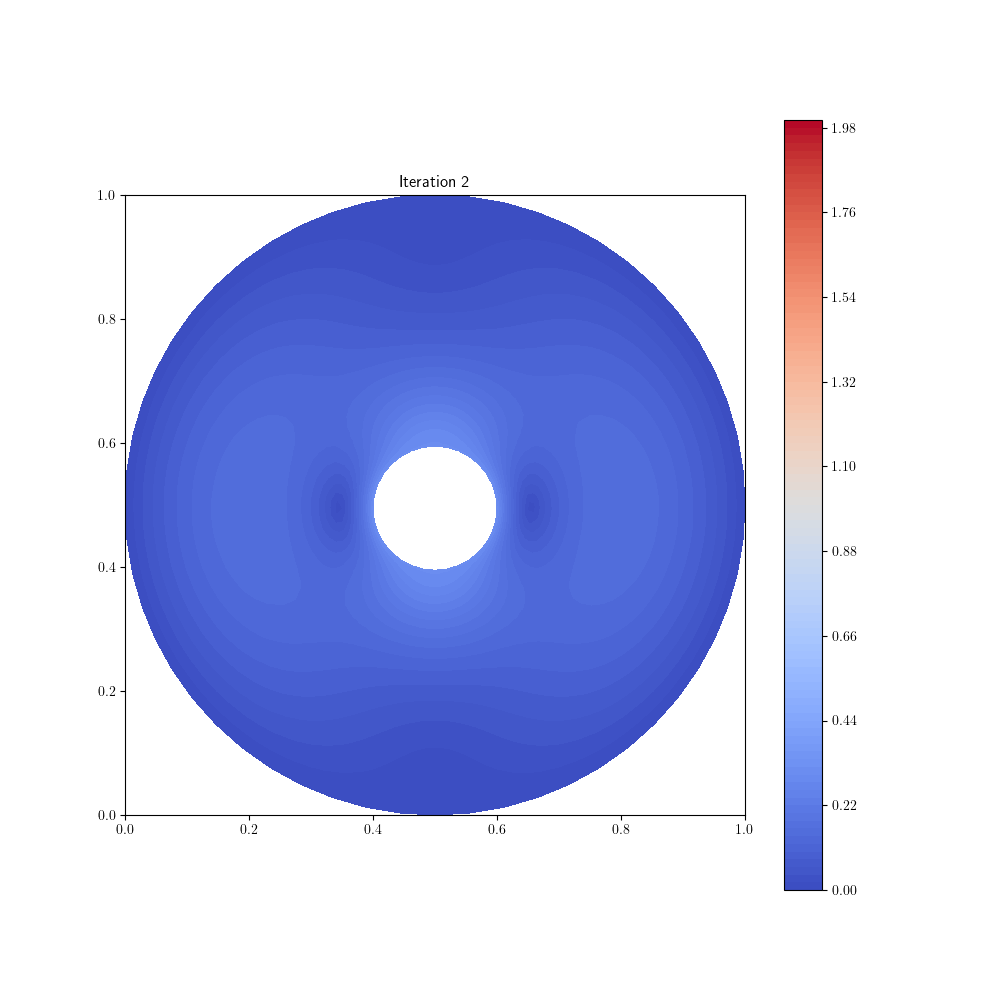}
        \caption{Iteration 2\\
        $t = 0.033$.}
    \end{subfigure}%
    ~
    \begin{subfigure}[t]{0.3\linewidth}
        \centering
        \includegraphics[width =\linewidth]{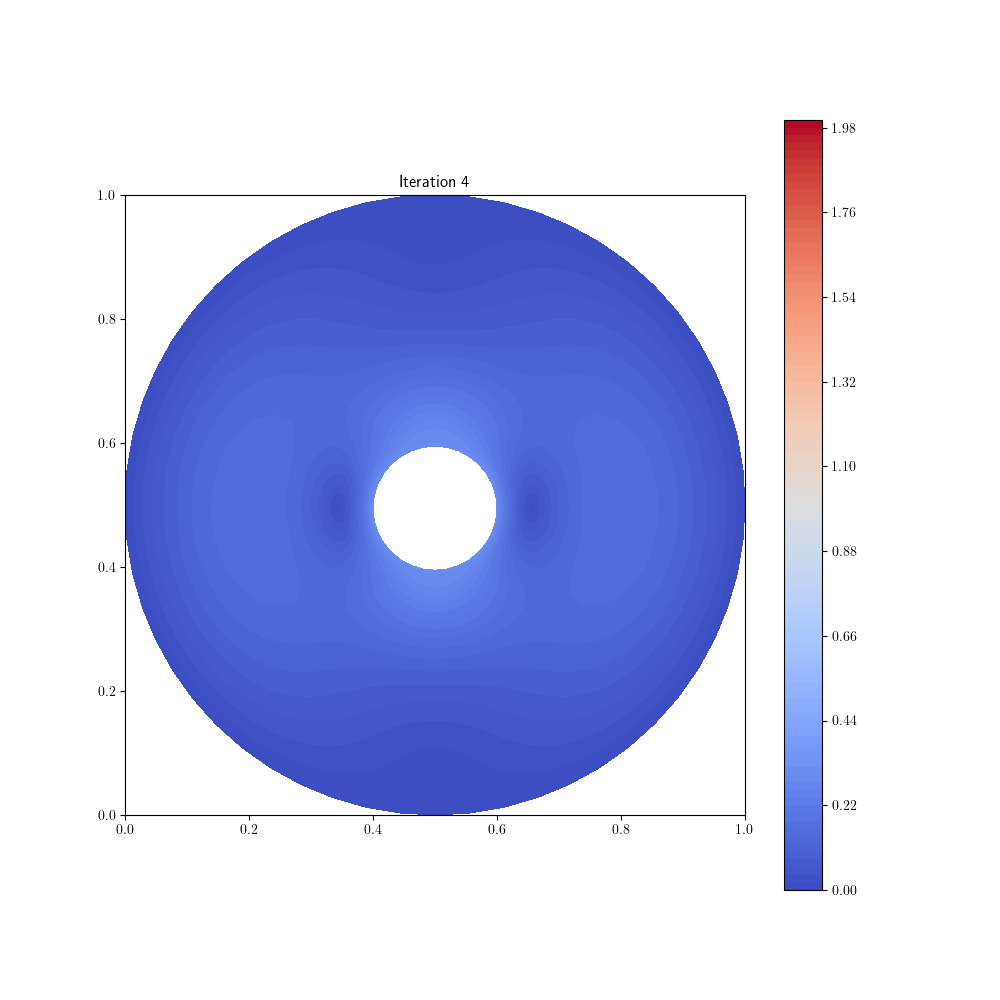}
        \caption{Iteration 4\\
        $t = 0.033$.}
    \end{subfigure}%
    \newline
    \begin{subfigure}[t]{0.3\linewidth}
        \centering
        \includegraphics[width =\linewidth]{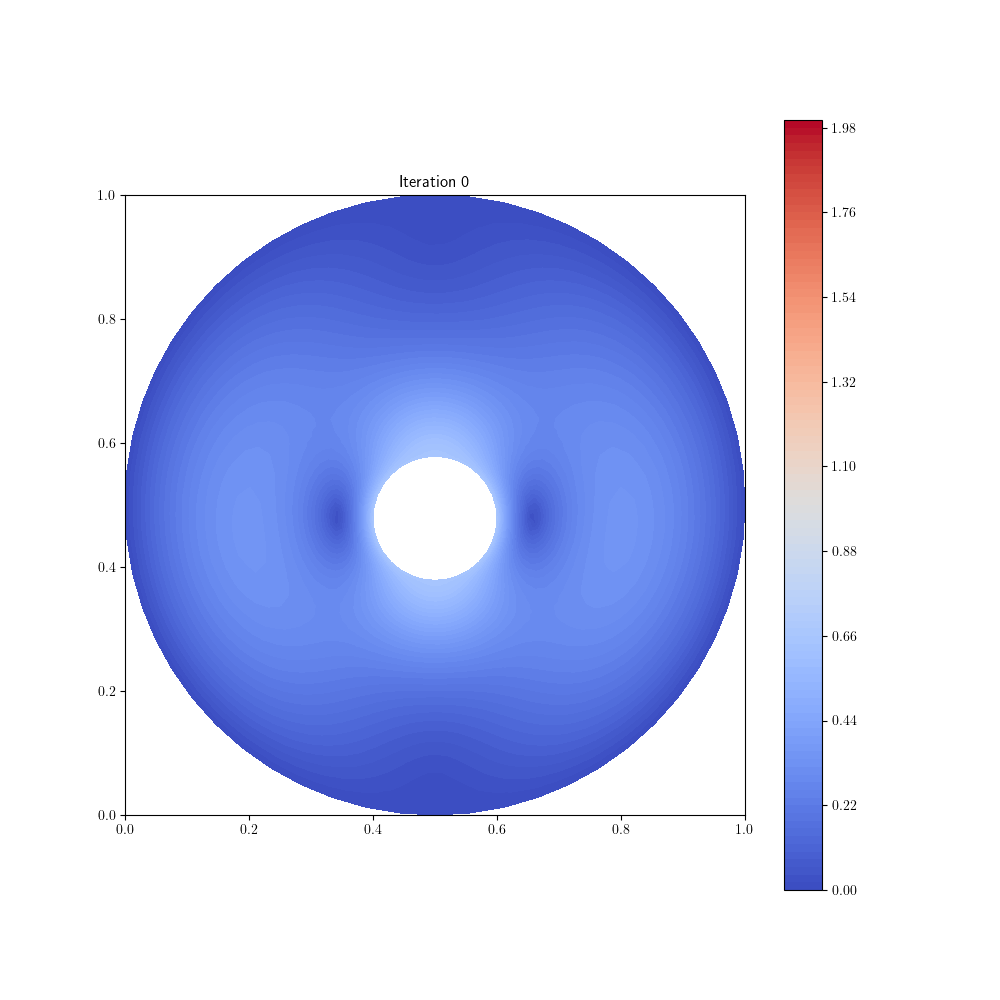}
        \caption{Iteration 0\\
        $t = 0.066$.}
    \end{subfigure}%
    ~
    \begin{subfigure}[t]{0.3\linewidth}
        \centering
        \includegraphics[width =\linewidth]{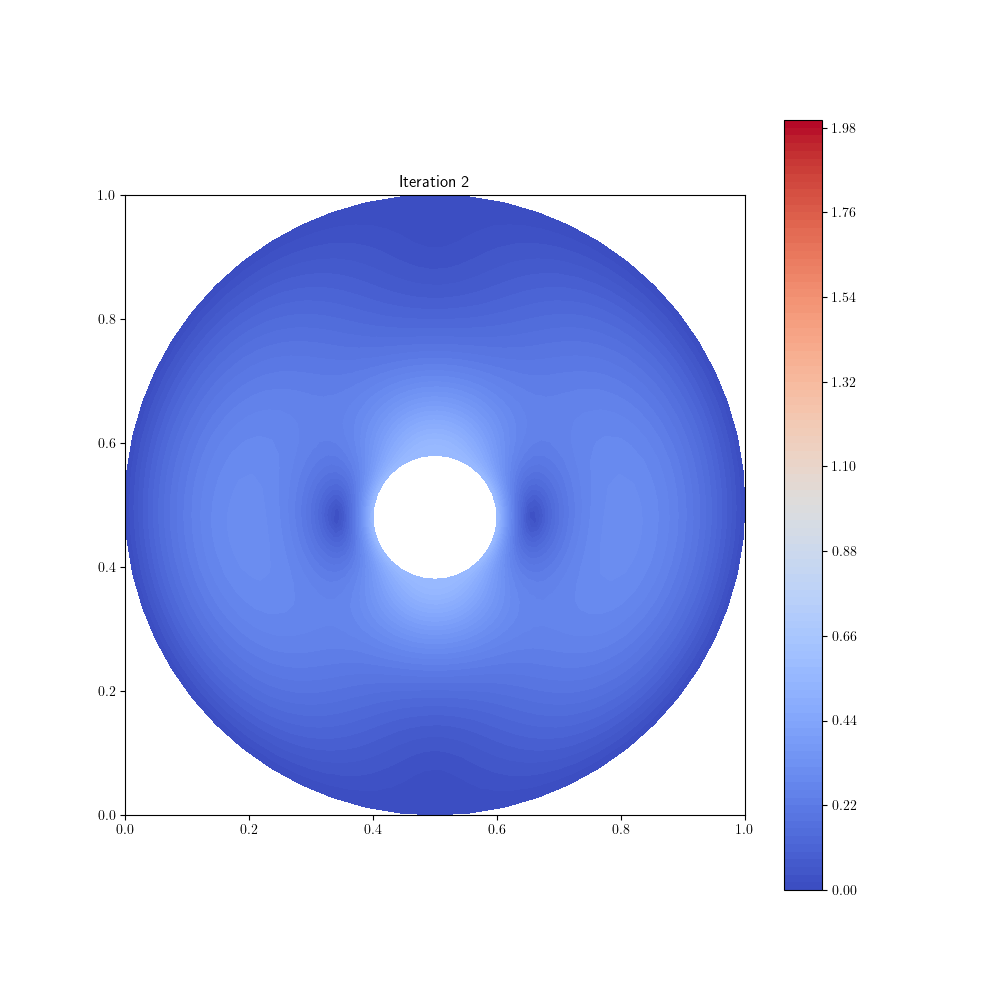}
        \caption{Iteration 2\\
        $t = 0.066$.}
    \end{subfigure}%
    ~
    \begin{subfigure}[t]{0.3\linewidth}
        \centering
        \includegraphics[width =\linewidth]{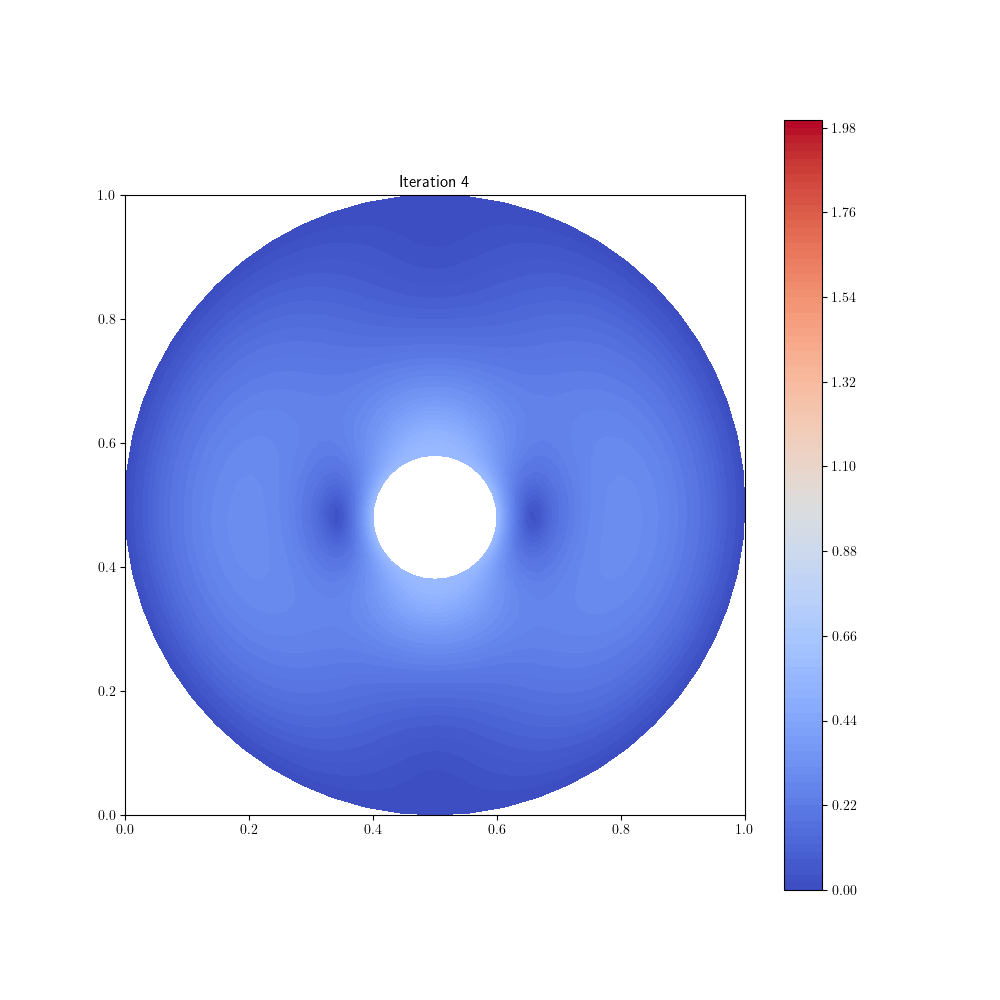}
        \caption{Iteration 4\\
        $t = 0.066$.}
    \end{subfigure}%
    \newline
    \begin{subfigure}[t]{0.3\linewidth}
        \centering
        \includegraphics[width =\linewidth]{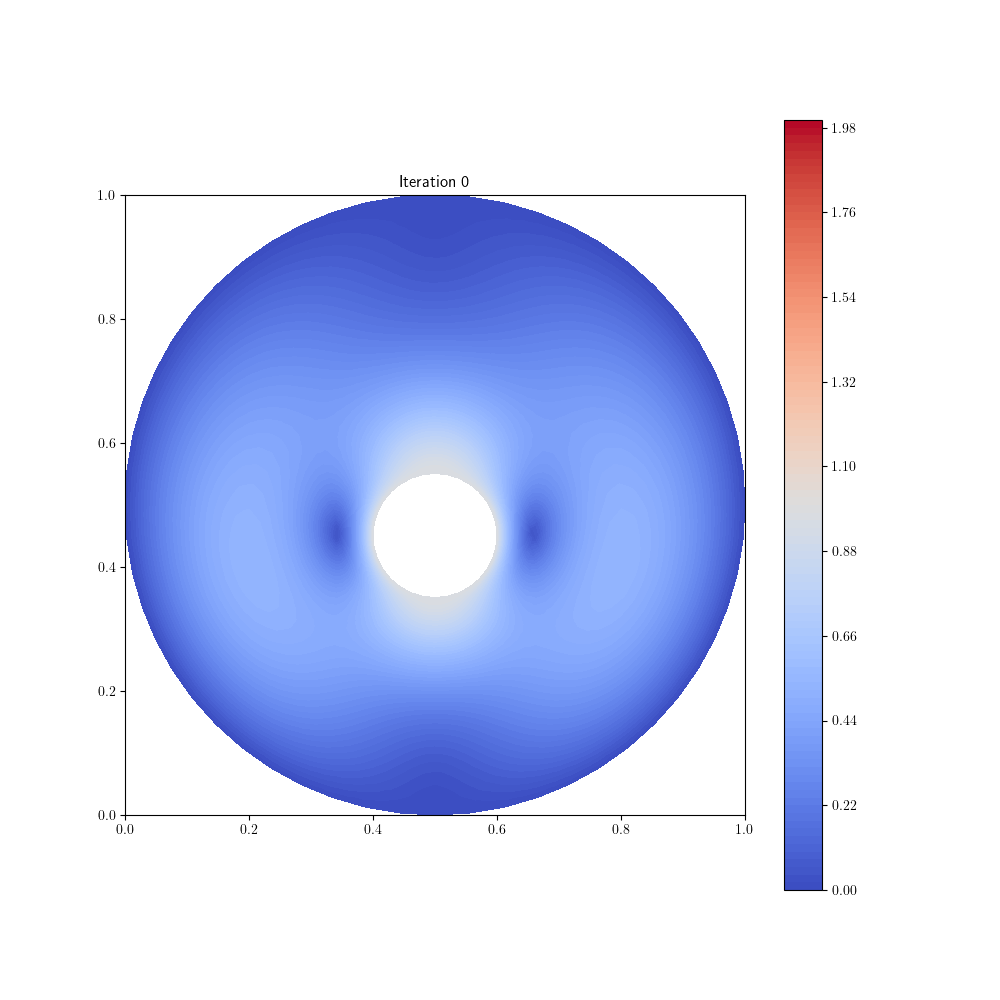}
        \caption{Iteration 0\\
        $t = 0.1$.}
    \end{subfigure}%
    ~
    \begin{subfigure}[t]{0.3\linewidth}
        \centering
        \includegraphics[width =\linewidth]{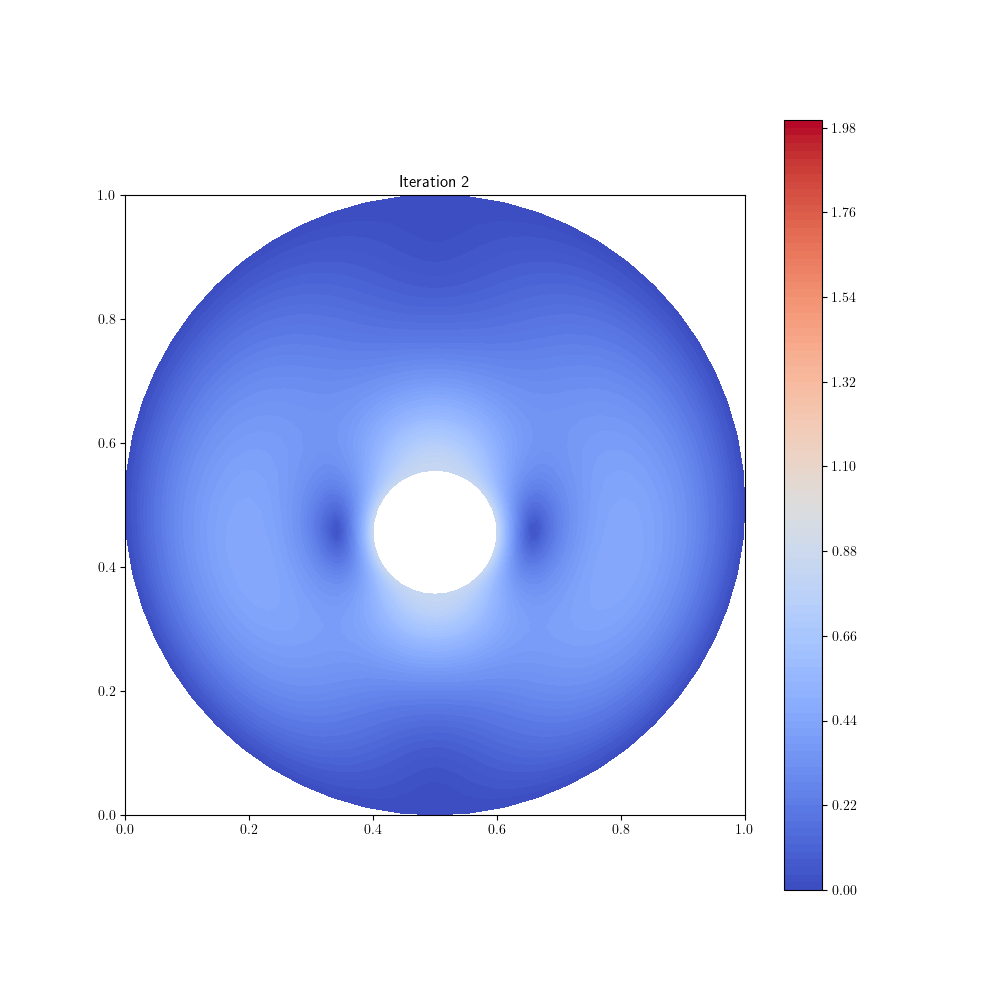}
        \caption{Iteration 2\\
        $t = 0.1$.}
    \end{subfigure}%
    ~
    \begin{subfigure}[t]{0.3\linewidth}
        \centering
        \includegraphics[width =\linewidth]{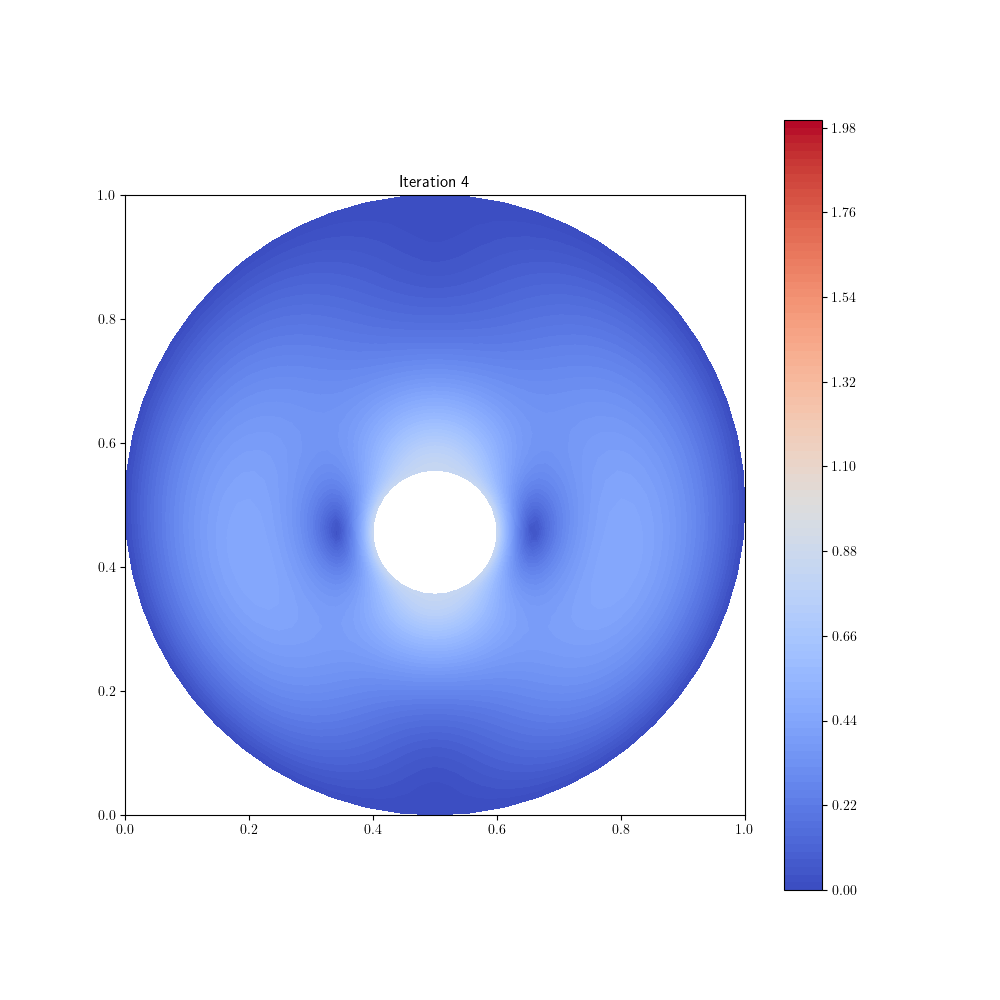}
        \caption{Iteration 4\\
        $t = 0.1$.}
    \end{subfigure}%
    \caption{Plots of the magnitude of the fluid velocity with a rigid body with density $\rho_B = \frac{200}{\pi}$.
    Row corresponds to $t = 0.033$, $t = 0.066$, and $t = 0.1$ respectively, and each column corresponds to the $0$th, $2$nd and $4$th iteration respectively.}
    \label{fig: fsi plots}
\end{figure}

This illustrates how our construction may be used in the design of a numerical method to solve \eqref{eqn: FSI fluid}, \eqref{eqn: FSI rigid body}, provided that the relative density is sufficiently small.
The rigorous numerical analysis is the topic for future work --- in particular it is likely one will encounter mesh degeneration which must be appropriately resolved for a reliable numerical method, for example by using introducing choosing the advective velocity as in~\cite{Elliott2016DeTurck}.
For the numerical experiments with a ball with small density ($\rho_B = \frac{10}{\pi}$) one observes that the iterations do indeed fail to converge, as illustrated in Table~\ref{table: fsi failure}, Table~\ref{table: fsi failure refined} and Figure~\ref{fig: fsi nonconvergence}.
This suggests that our smallness assumption on the relative density is in fact crucial for the convergence of our iterations.
As seen in the proof of Lemma~\ref{lemma: non degeneracy2} this issue seems to stem from the computation of the integral of the pressure over $\partial B(t)$.
It remains to be seen whether one can alter the iteration presented in this paper in such a way to remove the computation of this boundary integral, which should hopefully allow for arbitrary relative densities.
\begin{table}[ht]
    \centering
    \begin{tabular}{|c|c|c|c|}
    \hline
        Iteration & $\mbf{q}(T)$ & $\mbf{v}(T)$ & $\omega(T)$\\
        \hline
        0 & $(0.5,0.505)$ &   $(-6.158\cdot 10^{-4},4.379\cdot 10^{-1})$ & $3.225\cdot 10^{-3}$\\
        \hline
        1 & $(0.5,0.4854)$ & $(1.146\cdot 10^{-3},-8.709\cdot 10^{-1})$ & $-1.539\cdot 10^{-2}$\\
        \hline
        2 & $(0.5,0.5051)$ & $(-2.221\cdot 10^{-3},6.832\cdot 10^{-1})$ & $4.853\cdot 10^{-2}$\\
        \hline
        3 & $(0.5,0.4871)$ & $(2.423\cdot 10^{-3},-9.444\cdot 10^{-1})$ & $-1.161\cdot 10^{-1}$\\
        \hline
        4 & $(0.5,0.5023)$ & $(-3.963\cdot 10^{-3},6.012\cdot 10^{-1})$ & $2.388\cdot 10^{-1}$\\
        \hline
        5 & $(0.5,0.4902)$ & $(3.16\cdot 10^{-3},-7.541\cdot 10^{-1})$ & $-4.331\cdot 10^{-1}$\\
        \hline
        6 & $(0.5,0.4993)$ & $(-5.704\cdot 10^{-3},3.576\cdot 10^{-1})$ & $7.125\cdot 10^{-1}$\\
        \hline
        7 & $(0.5,0.4928)$ & $(4.418\cdot 10^{-3},-5.038\cdot 10^{-1})$ & $-1.077$\\
        \hline
        8 & $(0.5,0.4973)$ & $(-7.897\cdot 10^{-3},1.313\cdot 10^{-1})$ & $1.513$\\
        \hline
        9 & $(0.5,0.4943)$ & $(7.204\cdot 10^{-3},-3.168\cdot 10^{-1})$ & $-1.991$\\
        \hline
    \end{tabular}
    \caption{Table of values of the position, velocity, and angular velocity of the ball (with density $\rho_B = \frac{10}{\pi}$) at time $T = 0.05$ for each iteration.
    We note the apparent non-convergence in the $x$-component of the velocity, and in the angular velocity.}
    \label{table: fsi failure}
\end{table}

Lastly, choosing a finer mesh ($h = 4.879 \cdot 10^{-2}$) and smaller timestep ($\tau = 2.5 \cdot 10^{-5}$), one still seems to exhibit the same convergence (respectively non-convergence) behaviour for $\rho_B = \frac{200}{\pi}$ (respectively $\rho_B = \frac{10}{\pi}$).
This is illustrated in Table~\ref{table: fsi convergence refined} and Table~\ref{table: fsi failure refined}.
We similarly obtain the same behaviour when fixing $h$ and taking $\tau$ to be smaller, but we do not include these results here.

\begin{table}[ht]
    \centering
    \begin{tabular}{|c|c|c|c|}
    \hline
        Iteration & $\mbf{q}(T)$ & $\mbf{v}(T)$ & $\omega(T)$\\
        \hline
        0 & $(0.5,0.4567)$ &   $(-7.177\cdot 10^{-5},-8.149\cdot 10^{-1})$ & $-1.442\cdot 10^{-3}$\\
        \hline
        1 & $(0.5,0.4561)$ & $(-6.126\cdot 10^{-5},-8.362\cdot 10^{-1})$ & $-3.383\cdot 10^{-4}$\\
        \hline
        2 & $(0.5,0.4562)$ & $(-6.027\cdot 10^{-5},-8.342\cdot 10^{-1})$ & $-5.8\cdot 10^{-4}$\\
        \hline
        3 & $(0.5,0.4562)$ & $(-6.07\cdot 10^{-5},-8.344\cdot 10^{-1})$ & $-5.345\cdot 10^{-4}$\\
        \hline
        4 & $(0.5,0.4562)$ & $(-6.07\cdot 10^{-5},-8.344\cdot 10^{-1})$ & $-5.421\cdot 10^{-4}$\\
        \hline
    \end{tabular}
    \caption{Table of values of the position, velocity, and angular velocity of the ball (with density $\rho_B = \frac{200}{\pi}$) at time $T = 0.1$ for each iteration.
    Here $\tau = 2.5\cdot 10^{-5}$ and $h \approx 4.879\cdot 10^{-2}$.}
    \label{table: fsi convergence refined}
\end{table}

\begin{table}[ht]
    \centering
    \begin{tabular}{|c|c|c|c|}
    \hline
        Iteration & $\mbf{q}(T)$ & $\mbf{v}(T)$ & $\omega(T)$\\
        \hline
        0 & $(0.5,0.5046)$ &   $(-4.065\cdot 10^{-4},-3.978\cdot 10^{-1})$ & $4.324\cdot 10^{-1}$\\
        \hline
        1 & $(0.5,0.4857)$ & $(7.761\cdot 10^{-4},-8.586\cdot 10^{-1})$ & $1.998\cdot 10^{-2}$\\
        \hline
        2 & $(0.5,0.5047)$ & $(-1.138\cdot 10^{-3},6.603\cdot 10^{-1})$ & $-5.574\cdot 10^{-2}$\\
        \hline
        3 & $(0.5,0.4875)$ & $(1.558\cdot 10^{-3},-9.134\cdot 10^{-1})$ & $1.249\cdot 10^{-1}$\\
        \hline
        4 & $(0.5,0.5018)$ & $(-8.834\cdot 10^{-4},5.629\cdot 10^{-1})$ & $-2.467\cdot 10^{-1}$\\
        \hline
        5 & $(0.5,0.4907)$ & $(1.187\cdot 10^{-3},-7.144\cdot 10^{-1})$ & $4.284\cdot 10^{-1}$\\
        \hline
    \end{tabular}
    \caption{Table of values of the position, velocity, and angular velocity of the ball (with density $\rho_B = \frac{10}{\pi}$) at time $T = 0.05$ for each iteration.
    Here $\tau = 2.5\cdot 10^{-5}$ and $h \approx 4.879\cdot 10^{-2}$.}
    \label{table: fsi failure refined}
\end{table}

\begin{figure}[ht]
    \centering
    \begin{subfigure}[t]{0.3\linewidth}
        \centering
        \includegraphics[width =\linewidth]{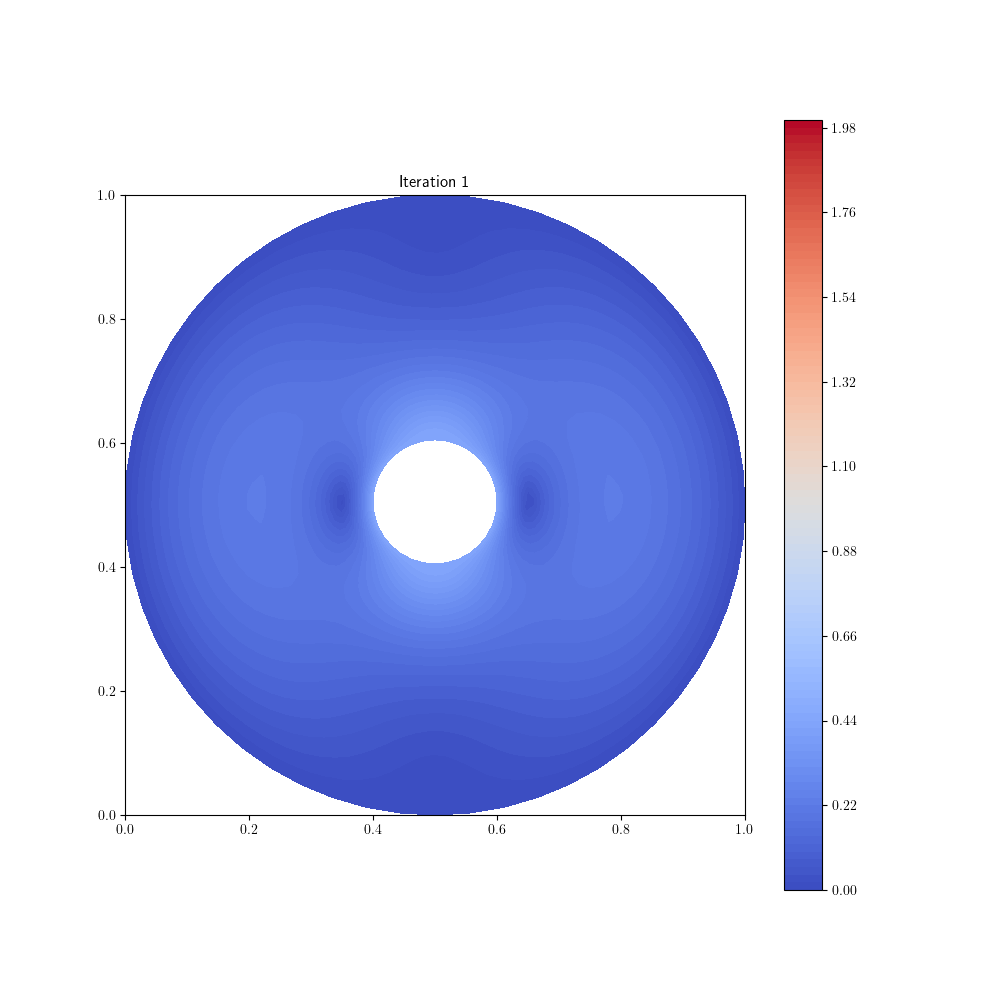}
        \caption{Iteration 1.}
    \end{subfigure}%
    ~
    \begin{subfigure}[t]{0.3\linewidth}
        \centering
        \includegraphics[width =\linewidth]{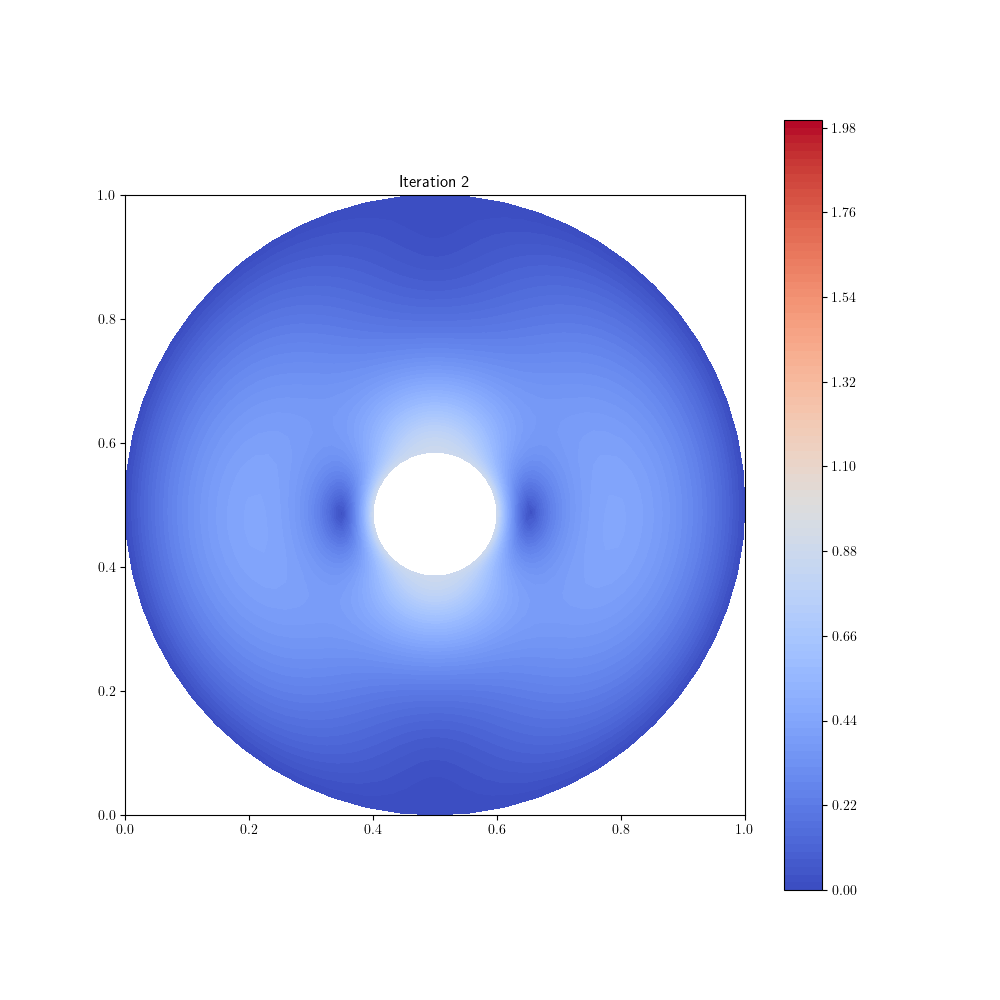}
        \caption{Iteration 2.}
    \end{subfigure}%
    ~
    \begin{subfigure}[t]{0.3\linewidth}
        \centering
        \includegraphics[width =\linewidth]{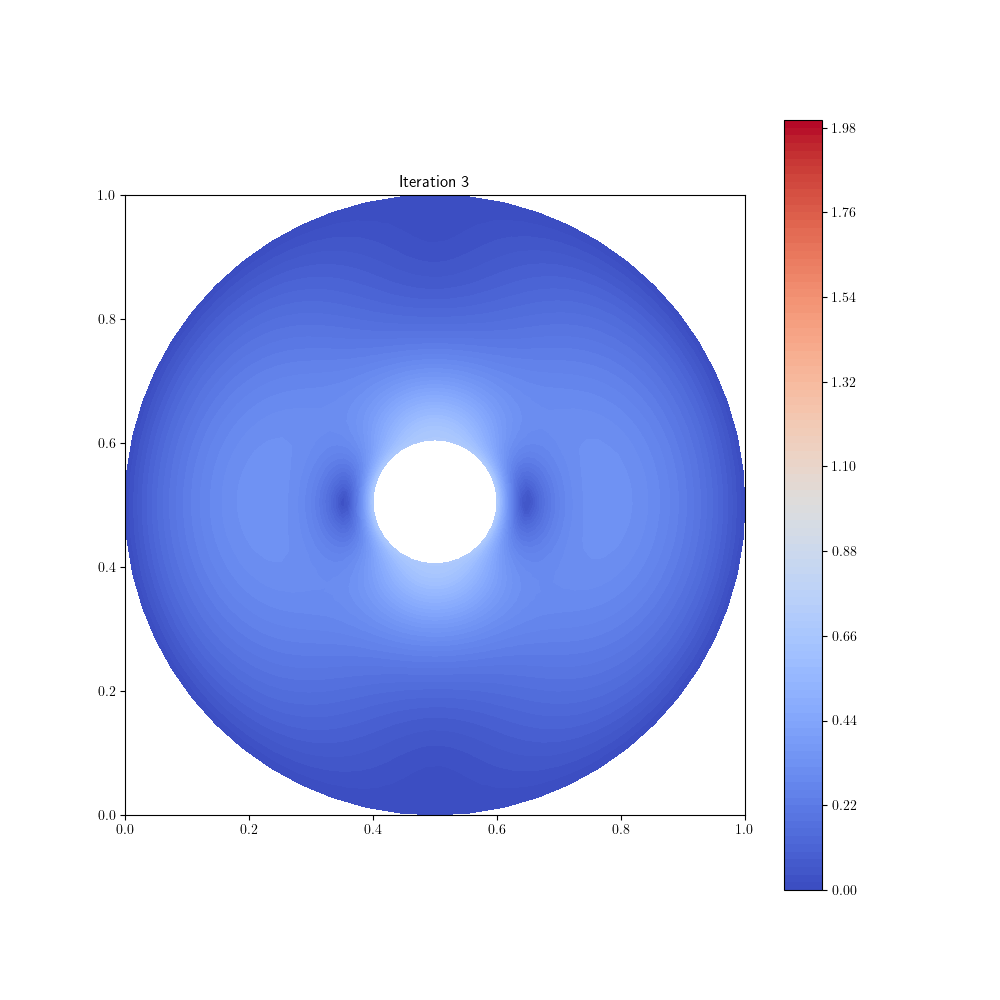}
        \caption{Iteration 3.}
    \end{subfigure}%
    \newline
    \begin{subfigure}[t]{0.3\linewidth}
        \centering
        \includegraphics[width =\linewidth]{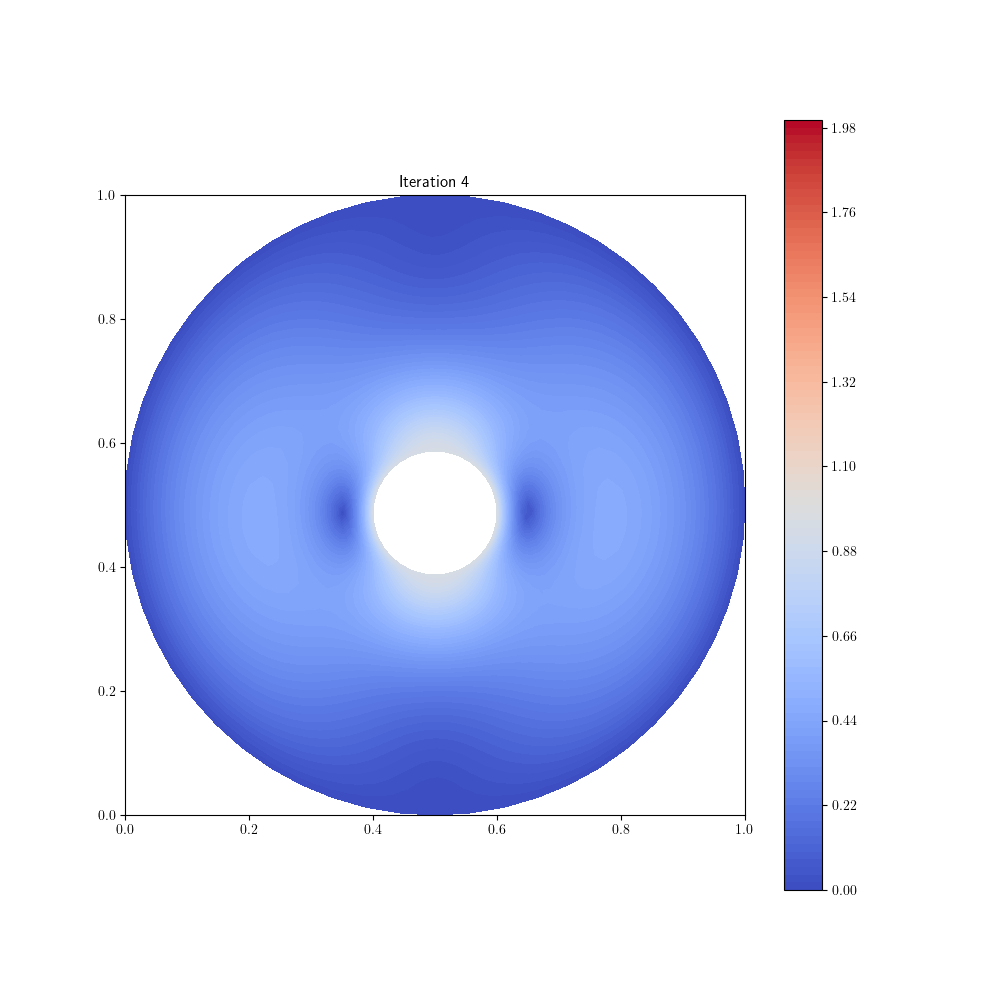}
        \caption{Iteration 4.}
    \end{subfigure}%
    ~
    \begin{subfigure}[t]{0.3\linewidth}
        \centering
        \includegraphics[width =\linewidth]{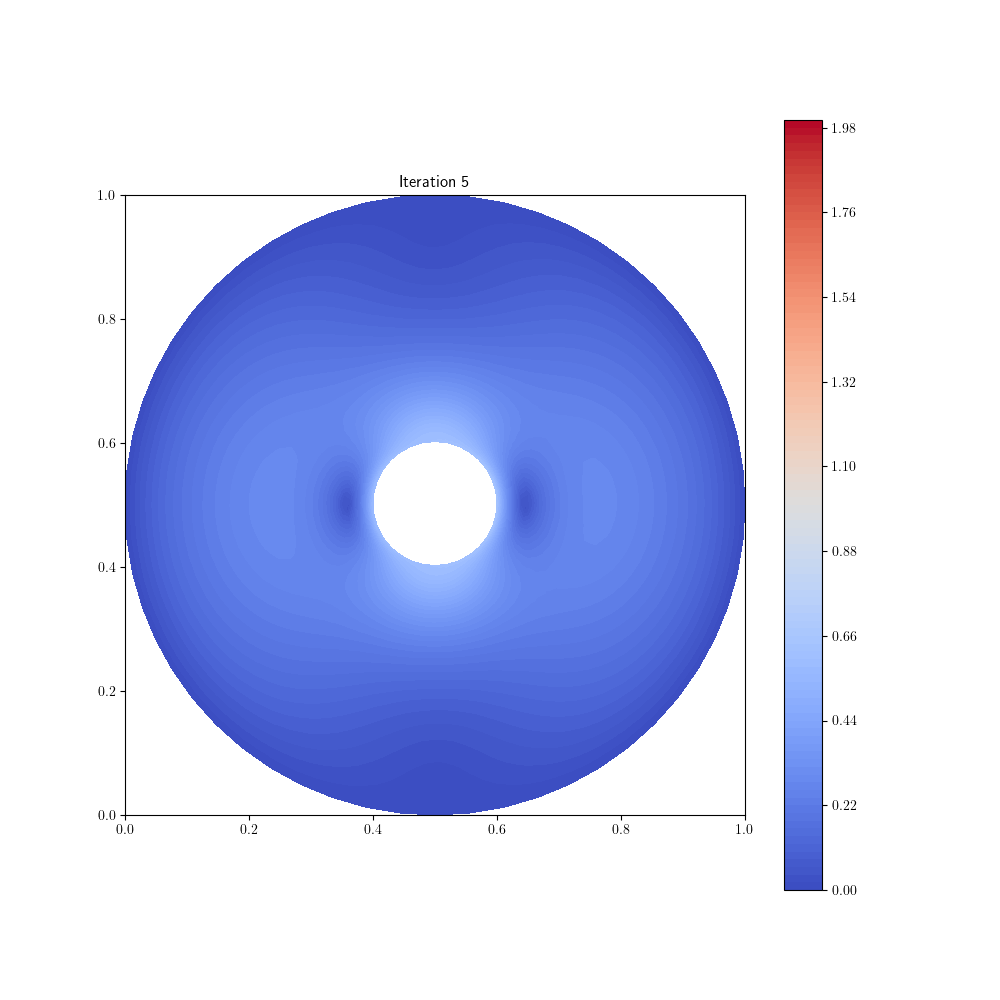}
        \caption{Iteration 5.}
    \end{subfigure}%
    ~
    \begin{subfigure}[t]{0.3\linewidth}
        \centering
        \includegraphics[width =\linewidth]{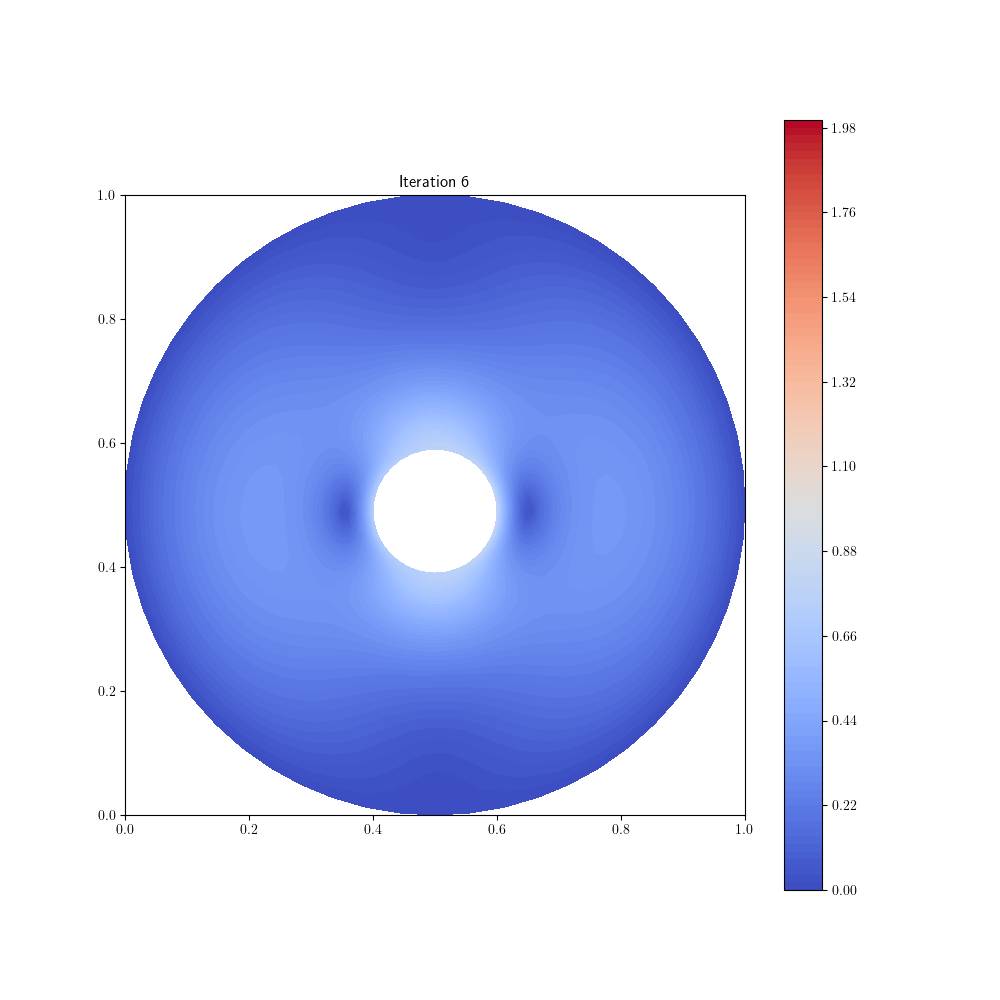}
        \caption{Iteration 6.}
    \end{subfigure}%
    \newline
    \begin{subfigure}[t]{0.3\linewidth}
        \centering
        \includegraphics[width =\linewidth]{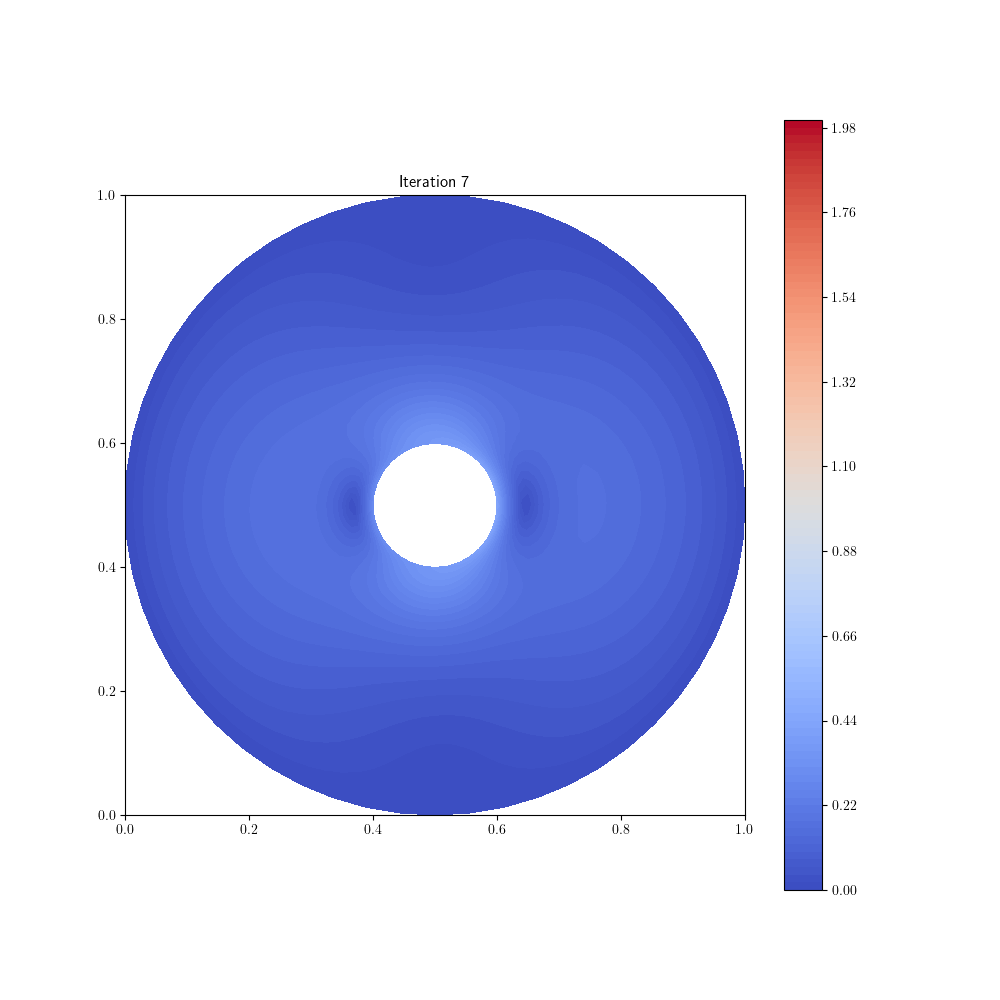}
        \caption{Iteration 7.}
    \end{subfigure}%
    ~
    \begin{subfigure}[t]{0.3\linewidth}
        \centering
        \includegraphics[width =\linewidth]{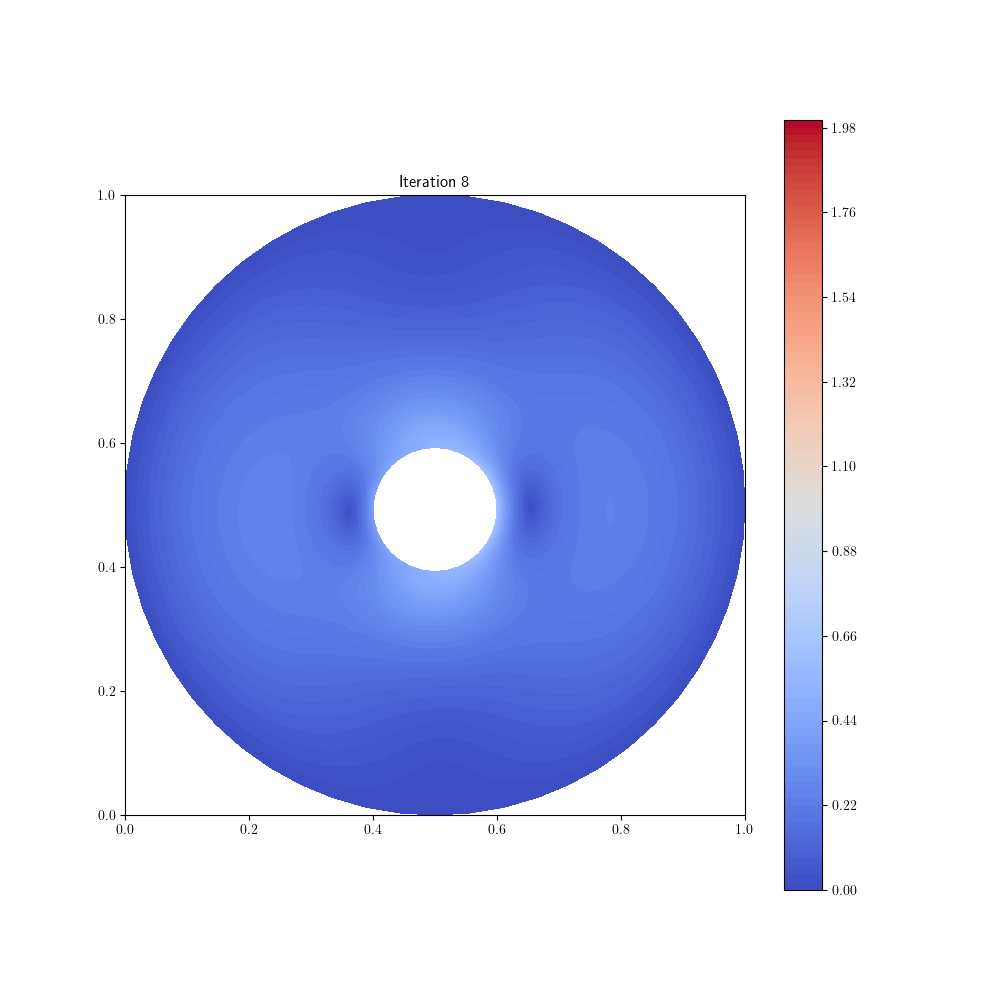}
        \caption{Iteration 8.}
    \end{subfigure}%
    ~
    \begin{subfigure}[t]{0.3\linewidth}
        \centering
        \includegraphics[width =\linewidth]{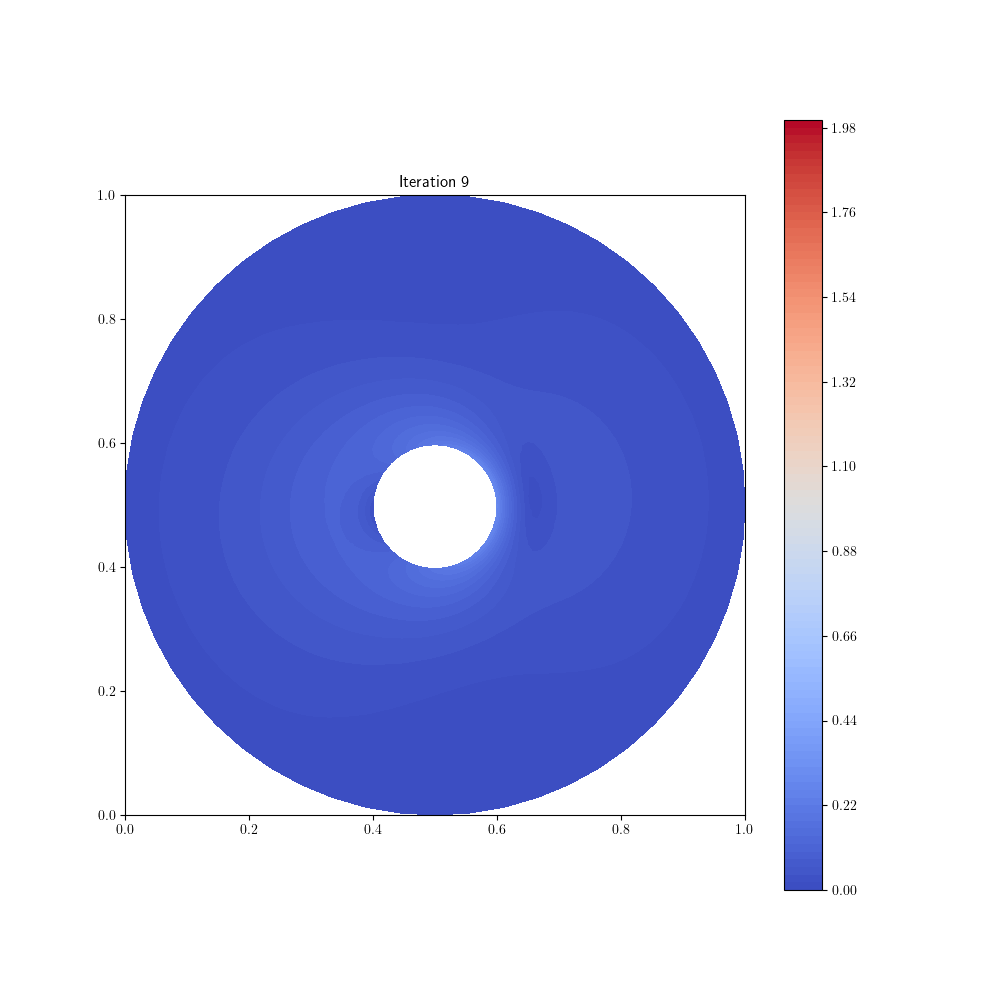}
        \caption{Iteration 9.}
    \end{subfigure}%
    \caption{Plots of the magnitude of the fluid velocity, at time $t = 0.05$, over 9 iterations, with a rigid body with density $\rho_B = \frac{10}{\pi}$.}
    \label{fig: fsi nonconvergence}
\end{figure}

\section*{Generalisations}
Our method readily extends to multiple rigid bodies with some minor, but tedious, modifications to our arguments.
Some more interesting topics would be whether this extends to the case of Navier-slip boundary conditions \eqref{eqn: slip BCs}, and to compressible fluids as in \cite{feireisl2003motion}.
Indeed the latter of these topics relies heavily on the short time existence of strong solutions to the compressible Navier--Stokes equations (see \cite{feireisl2004dynamics} or the recent monograph \cite{kreml2025book}) which would impact the proof of Lemma~\ref{lemma: non degeneracy1}.
Lastly, and the most challenging extension of our method would be to problems with a boundary that changes shape, as we have repeatedly exploited the fact that the boundary does not change shape.
Examples of such problems are fluid-elastic body interactions (such as those discussed in \cite{benevsova2023variational,vcanic2021moving}), free boundary problems (for example the problems discussed in \cite{EllOckBook82}), and problems coupling the geometric evolution of a surface to the solution of a PDE posed on the surface (see for example \cite{abels2023short}).

\subsection*{Acknowledgments}
Thomas Sales is supported by the UK Engineering and Physical Sciences Research Council (Grant number: EP/Z535138/1).

\appendix
\section{A calculation for Lemma~\ref{lemma: fsi energy estimates}}
\label{appendix: bihari use}

In this appendix we perform the calculation which was omitted at the end of the proof of Lemma \ref{lemma: fsi energy estimates}.
That is, we verify that \eqref{fsi energyestimate2} holds given that \eqref{smalltimecondition} holds.
We recall from the proof of Lemma \ref{lemma: fsi energy estimates} that we want to use Lemma \ref{biharilasalle} with $\lambda(s) = s + s^3$, where for a small, fixed $y_0 > 0$ one finds that 
\begin{gather*}
\Lambda(y) = \int_{y_0}^y \frac{1}{s + s^3} \, \mathrm{d}s = \frac{1}{2}\log\left(\frac{y^2}{1+y^2}\right) + \frac{1}{2}\log\left(\frac{1+y_0^2}{y_0^2}\right),\\
\Lambda^{-1}(z) = \left( \frac{y_0^2}{1 + y_0^2 - y_0^2\exp(2z))} \right)^{\frac{1}{2}} \exp(z),\\
\text{dom}(\Lambda^{-1}) = \left(-\infty, \frac{1}{2}\log\left(\frac{1+y_0^2}{y_0^2}\right)\right).
\end{gather*}
It is clear to see that the condition \eqref{smalltimecondition} is equivalent to 
\begin{multline*}
		\frac{1}{2} \log \left( \frac{\rho^2 \|\grad \mbf{u}_0\|_{\mbf{L}^2(\Omega(0))}^4 + C\left(\kappa \int_0^{\widetilde{T}} \|\mbf{B}\|_{\mbf{H}^1(\Omega(t))}^2 \right)^2}{4 + \rho^2 \|\grad \mbf{u}_0\|_{\mbf{L}^2(\Omega(0))}^4 + C\left(\kappa\int_0^{\widetilde{T}} \|\mbf{B}\|_{\mbf{H}^1(\Omega(t))}^2 \right)^2} \right) + \frac{1}{2}\log\left(\frac{1+y_0^2}{y_0^2}\right)\\
        +\int_0^{\widetilde{T}}C \left(\frac{1}{\kappa^2} + \|\Div\mbf{V}\|_{L^\infty(\Omega(t))} + \|\mbf{V}\|_{\mbf{L}^\infty(\Omega(t))}^2 +\|\widetilde{\mbf{u}}\|_{\mbf{L}^4(\Omega(t))}^8 +\|\widetilde{\mbf{u}}\|_{\mbf{W}^{1,4}(\Omega(t))}^2\right)\\
        \in \left(-\infty, \frac{1}{2}\log\left(\frac{1+y_0^2}{y_0^2}\right)\right),
\end{multline*}
and hence we may apply $\Lambda^{-1}$ to this term.
For ease of notation we now define 
\begin{gather*}
    \mathcal{A}_1 := \frac{\rho^2 \|\grad \mbf{u}_0\|_{\mbf{L}^2(\Omega(0))}^4 + C\left(\int_0^{\widetilde{T}} \|\mbf{B}\|_{\mbf{L}^2(\Omega(t))}^2 \right)^2}{4 + \rho^2 \|\grad \mbf{u}_0\|_{\mbf{L}^2(\Omega(0))}^4 + C\left(\int_0^{\widetilde{T}} \|\mbf{B}\|_{\mbf{L}^2(\Omega(t))}^2 \right)^2},\\
    \mathcal{A}_2 := \int_0^{\widetilde{T}}C \left( \frac{1}{\kappa^2} + \|\Div\mbf{V}\|_{L^\infty(\Omega(t))} + \|\mbf{V}\|_{\mbf{L}^\infty(\Omega(t))}^2 +\|\widetilde{\mbf{u}}\|_{\mbf{L}^4(\Omega(t))}^8 +\|\widetilde{\mbf{u}}\|_{\mbf{W}^{1,4}(\Omega(t))}^2\right).
\end{gather*}
Now using the above expression for $\Lambda^{-1}$ one finds that
\begin{multline*}
    \Lambda^{-1}\left(  \frac{1}{2} \log(\mathcal{A}_1) + \mathcal{A}_2 + \frac{1}{2} \log\left( \frac{1 + y_0^2}{y_0^2} \right) \right)\\
    = \left( \frac{y_0^2}{1 + y_0^2 - y_0^2 \left( \frac{1 + y_0^2}{y_0^2} \right)\mathcal{A}_1  \exp(2 \mathcal{A}_2)} \right)^{\frac{1}{2}} \sqrt{\mathcal{A}_1} \sqrt{\frac{1 + y_0^2}{y_0^2}} \exp(\mathcal{A}_2)\\
    = \frac{\sqrt{\mathcal{A}_1} \exp(\mathcal{A}_2)}{\sqrt{1 -\mathcal{A}_1 \exp(2 \mathcal{A}_2)}}.
\end{multline*}
Unpacking the expression for $\mathcal{A}_1$, and recalling $\mathcal{E}(\widetilde{T}) = \exp(2 \mathcal{A}_2)$ we see
\begin{equation*}
    \frac{\sqrt{\mathcal{A}_1} \exp(\mathcal{A}_2)}{\sqrt{1 -\mathcal{A}_1 \exp(2 \mathcal{A}_2)}} = \mathcal{D}(\widetilde{T}),
\end{equation*}
for $\mathcal{D}(\cdot)$ as defined in Lemma~\ref{lemma: fsi energy estimates}, from which one finds that \eqref{fsi energyestimate2} holds.

\section{The Stokes equations on an evolving domain}
\label{appendix:stokes}

Let $\mbf{v}(t)$ to be the unique solution to the following problem\footnote{Here we have set the viscosity to be $\mu = 1$ for notational simplicity.} on $\Omega(t)$:
\begin{align}
    \begin{rcases}
        -\Delta \mbf{v} + \grad p &= 0\\
    \Div \mbf{v} &= 0
    \end{rcases} \text{ on } \Omega(t),
    \label{stokes problem}
\end{align}
such that
\[ \mbf{v} = 0 \text{ on } \Gamma, \quad \mbf{v}= \mbf{q}'(t) + \boldsymbol{\omega}(t)\times(\mbf{x} - \mbf{q}(t)) \text{ on } \partial B(t). \]
Here we are considering $\Omega(t)$ to be determined by $\mbf{q}$ and $\boldsymbol{\omega}$, as in Section \ref{section: evolving NS}.
Existence of a weak solution follows from standard results on the Stokes equation, and moreover one finds that, for $r \in [1, \infty)$,
\begin{align}
    \| \mbf{v} \|_{\mbf{W}^{2,r}(\Omega(t))} + \|{p}\|_{W^{1,r}(\Omega(t))}\leq C \| \mbf{q}'(t) + \boldsymbol{\omega}(t)\times(\mbf{x} - \mbf{q}(t)) \|_{\mbf{W}^{2,r}(\Omega(t))}, \label{stokes regularity1}
\end{align}
for a constant $C$ which depends only on $r$ and the geometry of $\Omega(0)$, cf.~\cite{galdi2011introduction,temam2001navier}.
From this and the smoothness of $\mbf{q}$ and $\boldsymbol{\omega}$ we find that
\[ \mbf{v} \in L^\infty_{\mbf{W}^{2,r}}(\Phi), \quad p \in L^\infty_{W^{1,r}}(\Phi) \quad \forall r \in [1,\infty). \]
We will also require that the time derivative of $\mbf{v}$ to be sufficiently smooth, namely an element of $L^2_{\mbf{H}^1}(\Phi)$.
We verify this by pulling this equation back to $\Omega(0)$ and differentiating in time.

\subsection{Some estimates}

We introduce shorthand notation $\widehat{\mbf{v}} = \mbf{v} \circ \Phi(t)$ and $\widehat{p} = p \circ \Phi(t)$.
Then these solve the system

\begin{gather}
    \int_{\Omega(0)} J \mbb{D} \mbb{D}^T \grad \widehat{\mbf{v}}:\grad\boldsymbol{\eta} = \int_{\Omega(0)} J \widehat{p} \tr{\mbb{D}^T \grad \boldsymbol{\eta}}, \label{stokes pullback1}\\
    \int_{\Omega(0)} J q \tr{\mbb{D}^T \grad \widehat{\mbf{v}}} = 0, \label{stokes pullback2}
\end{gather}
for all $\boldsymbol{\eta} \in \mbf{H}^1_0(\Omega(0)), q \in L^2(\Omega(0))$, where $\mbb{D} = \grad \Phi, J = \det(\mbb{D})$.
We observe that
\[ \tr{\mbb{D}^T \grad \boldsymbol{\eta}} = \mbb{D}: \grad \boldsymbol{\eta}, \]
and will use this expression throughout.
It is useful to note that the bilinear form on the left of \eqref{stokes pullback1} is coercive, and the bilinear form in \eqref{stokes pullback2} satisfies an inf-sup condition.
This first statement is trivially true, but we briefly prove this latter statement. 

\begin{lemma}
    Let 
    \[ L_t^2(\Omega(0)) := \left\{ q \in L^2(\Omega(0)) \mid \int_{\Omega(t)} q \circ \Phi(t)^{-1} = 0 \right\}. \]
    This is a closed subspace of $L^2(\Omega(0))$, and moreover there exists a constant $C(\Phi(t))$ such that
    \begin{align}
        \inf_{q \in L_t^2(\Omega(0))} \sup_{\boldsymbol{\eta} \in \mbf{H}^1(\Omega(0))} \frac{\int_{\Omega(0)} J q \mbb{D}: \grad \boldsymbol{\eta}}{\|\boldsymbol{\eta}\|_{\mbf{H}^1(\Omega(0))} \|q \|_{L^2(\Omega(0))}} \geq C(\Phi(t)) > 0.
    \end{align}
\end{lemma}
\begin{proof}
    Firstly we notice that by pulling back the integral which defines $L_t^2(\Omega(0))$ we find
    \[ \int_{\Omega(t)} q \circ \Phi(t)^{-1} = \int_{\Omega(0)} J(t) q, \]
    and so we can view $L_t^2(\Omega(0))$ as the preimage of $0$ under the map $q \mapsto \int_{\Omega(0)} J(t) q$, and is hence a closed subspace of $L^2(\Omega(0))$.

    By considering the pushforward of this bilinear form we find that an equivalent quantity to consider is
    \[ \inf_{\widetilde{q} \in L_0^2(\Omega(t))} \sup_{\widetilde{\boldsymbol{\eta}} \in \mbf{H}^1(\Omega(t))} \frac{\int_{\Omega(t)} \widetilde{q} \Div \widetilde{\boldsymbol{\eta}}}{C_{L^2}(t)C_{\mbf{H}^1}(t)\|\widetilde{\boldsymbol{\eta}}\|_{\mbf{H}^1(\Omega(t))} \|\widetilde{q} \|_{L^2(\Omega(t))}},  \]
    where $C_{L^2}(t), C_{\mbf{H}^1}(t)$ are norm equivalence constants which appear in the compatibility of the pairs $(L^2(\Omega(t)), \Phi(t))$, $(\mbf{H}^1(\Omega(t)), \Phi(t))$ respectively, and
    \[ L^2_0(\Omega(t)) := \left\{ \widetilde{q} \in L^2(\Omega(t)) \mid \int_{\Omega(t))} \widetilde{q} = 0\right\}. \]
    It is well known that there is a constant $\widetilde{C} > 0$, which depends only on the geometry of $\Omega(0)$, such that
    \[ \inf_{\widetilde{q} \in L_0^2(\Omega(t))} \sup_{\widetilde{\boldsymbol{\eta}} \in \mbf{H}^1(\Omega(0))} \frac{\int_{\Omega(t)} \widetilde{q} \Div \widetilde{\boldsymbol{\eta}}}{\|\widetilde{\boldsymbol{\eta}}\|_{\mbf{H}^1(\Omega(t))} \|\widetilde{q} \|_{L^2(\Omega(t))}}  \geq \widetilde{C} > 0.\]
    Thus our result holds for $C(\Phi(t)) := \widetilde{C}C_{L^2}(t)C_{\mbf{H}^1}(t)$.
\end{proof}
With this inf-sup condition we can now show a bound for $\ddt{\widehat{\mbf{v}}}$.
\begin{lemma}
    $\mbf{v}$ solving \eqref{eqn: stokes eqn} for almost all $t \in [0,T]$ is such that $\mbf{v} \in H^1_{\mbf{H}^1} \cap L^\infty_{\mbf{W}^{2,r}}$ for all $r \in [1,\infty)$.
\end{lemma}
\begin{proof}
    We have already provided the $\mbf{W}^{2,r}$ estimate above, and so our focus is on the time derivative of $\mbf{v}$.
    One can show the existence of strong time derivatives (see \cite[Appendix A]{elliott2024navier} for a similar calculation), but we provide only a formal argument which yields bounds for this function.
    By differentiating \eqref{stokes pullback1} in time, for a sufficiently smooth test function, one finds
    \begin{align*}
        &\int_{\Omega(0)} \ddt{}(J \mbb{D} \mbb{D}^T) \grad \widehat{\mbf{v}}: \grad \boldsymbol{\eta} + \int_{\Omega(0)} J \mbb{D} \mbb{D}^T \grad \ddt{\widehat{\mbf{v}}}:\grad \boldsymbol{\eta} + \int_{\Omega(0)} J \mbb{D} \mbb{D}^T \grad\widehat{\mbf{v}}:\grad\ddt{\boldsymbol{\eta}}\\
        &= \int_{\Omega(0)} \ddt{J} \widehat{p} \tr{\mbb{D}^T \grad \boldsymbol{\eta}} + \int_{\Omega(0)} J \ddt{\widehat{p}} \mbb{D} :\grad \boldsymbol{\eta} + \int_{\Omega(0)} J \widehat{p} \ddt{\mbb{D}}:\grad \boldsymbol{\eta}+ \int_{\Omega(0)} J \widehat{p} \mbb{D}: \grad \ddt{\boldsymbol{\eta}},
    \end{align*}
    and we can use \eqref{stokes pullback1} to see that this simplifies to
    \begin{equation}
        \begin{aligned}
        &\int_{\Omega(0)} \ddt{}(J \mbb{D} \mbb{D}^T) \grad\widehat{\mbf{v}}:\grad \boldsymbol{\eta} + \int_{\Omega(0)} J \mbb{D} \mbb{D}^T \grad\ddt{\widehat{\mbf{v}}}: \grad \boldsymbol{\eta}\\
        &= \int_{\Omega(0)} \ddt{J} \widehat{p} \mbb{D}: \grad \boldsymbol{\eta} + \int_{\Omega(0)} J \ddt{\widehat{p}} \mbb{D}: \grad \boldsymbol{\eta} + \int_{\Omega(0)} J \widehat{p} \ddt{\mbb{D}} : \grad \boldsymbol{\eta}.
        \end{aligned}
        \label{stokes pullback3}
    \end{equation}
    Similarly we find that differentiating \eqref{stokes pullback2} yields
    \begin{align}
        \int_{\Omega(0)} \ddt{J} q \mbb{D}: \grad \widehat{\mbf{v}} + \int_{\Omega(0)} J q \ddt{\mbb{D}}: \grad \widehat{\mbf{v}} + \int_{\Omega(0)} J q \mbb{D}:\grad \ddt{\widehat{\mbf{v}}} = 0. \label{stokes pullback4}
    \end{align}
Thus we find that the pair $\left( \ddt{\widehat{\mbf{v}}}, \ddt{\widehat{p}} \right)$ solve the saddle point problem
\begin{gather}
    \begin{aligned}
    \int_{\Omega(0)} J \mbb{D} \mbb{D}^T \grad \ddt{\widehat{\mbf{v}}}:\grad \boldsymbol{\eta} - \int_{\Omega(0)} J \ddt{\widehat{p}} \mbb{D}: \grad \boldsymbol{\eta} 
        &= -\int_{\Omega(0)} \ddt{}(J \mbb{D} \mbb{D}^T) \grad \widehat{\mbf{v}}:\grad \boldsymbol{\eta}\\
        &+ \int_{\Omega(0)} \ddt{J} \widehat{p} \mbb{D}: \grad \boldsymbol{\eta} + \int_{\Omega(0)} J \widehat{p} \ddt{\mbb{D}}: \grad \boldsymbol{\eta},
    \end{aligned}\label{eqn: stokes pullback1}\\
    \int_{\Omega(0)} J q \mbb{D}: \grad \ddt{\widehat{\mbf{v}}} = -\int_{\Omega(0)} \ddt{J} q \mbb{D} :\grad \widehat{\mbf{v}} -\int_{\Omega(0)} J q \ddt{\mbb{D}} :\grad \widehat{\mbf{v}}, \label{eqn: stokes pullback2}
\end{gather}
for all $q \in L^2(\Omega(0)), \boldsymbol{\eta} \in \mbf{H}_0^1(\Omega(0))$.
We have written the system in this way so that the terms on the left-hand side involve the time derivatives, and the terms on the right-hand side can be seen as some ``forcing'' term in the system.
This system is subject to boundary conditions
\begin{align*} \ddt{\widehat{\mbf{v}}} = 0 \text{ on } \Gamma, \quad \ddt{\widehat{\mbf{v}}} = \mbf{q}''(t) + \boldsymbol{\omega}'(t)\times(\mbf{x} - \mbf{q}(t)) - \boldsymbol{\omega}(t) \times \mbf{q}'(t) \text{ on } \partial B(0).
\end{align*}
Now owing to the fact that these bilinear forms have coercivity and inf-sup properties respectively we can apply standard results for saddle point problems in Hilbert spaces, see for instance \cite[Chapter 4]{Boffi13Book}, to obtain a bound of the form
\begin{align}
    \left\|  \ddt{\widehat{\mbf{v}}} \right\|_{\mbf{H}^1(\Omega(0))} &\leq \frac{C}{\inf_{\Omega(0)} |J \mbb{D} \mbb{D}^T|} \left\| \ddt{}(J \mbb{D} \mbb{D}^T) \right\|_{\mbf{L}^\infty(\Omega(0))}\|\widehat{\mbf{v}}\|_{\mbf{H}^1(\Omega(0))} \notag \\
    &+ \frac{C}{\inf_{\Omega(0)} |J \mbb{D} \mbb{D}^T|} \left\| \ddt{J} \right\|_{L^\infty(\Omega(0))} \| \mbb{D} \|_{\mbf{L}^\infty(\Omega(0))} \|\widehat{p}\|_{\mbf{L}^2(\Omega(0))}\notag \\
    &+ \frac{C}{\inf_{\Omega(0)} |J \mbb{D} \mbb{D}^T|}\left\| J\right\|_{L^\infty(\Omega(0))} \left\| \ddt{\mbb{D}} \right\|_{\mbf{L}^\infty(\Omega(0))} \|\widehat{p}\|_{\mbf{L}^2(\Omega(0))} \notag \\
    &+ \frac{C \sup_{\Omega(0)}|J \mbb{D} \mbb{D}^T|}{C(\Phi(t)) \inf_{\Omega(0)} |J \mbb{D} \mbb{D}^T| }\left\|  \ddt{J} \right\|_{L^\infty(\Omega(0))} \| \mbb{D}\|_{\mbf{L}^\infty(\Omega(0))} \|\widehat{\mbf{v}}\|_{\mbf{H}^1(\Omega(0))} \label{stokes regularity2}\\
    &+ \frac{C \sup_{\Omega(0)}|J \mbb{D} \mbb{D}^T|}{C(\Phi(t)) \inf_{\Omega(0)} |J \mbb{D} \mbb{D}^T| }\left\|  J \right\|_{L^\infty(\Omega(0))} \left\| \ddt{\mbb{D}}\right\|_{\mbf{L}^\infty(\Omega(0))} \|\widehat{\mbf{v}}\|_{\mbf{H}^1(\Omega(0))} \notag \\
    &+ \frac{C}{\inf_{\Omega(0)} |J \mbb{D} \mbb{D}^T|} \left\| \mbf{q}''(t) + \boldsymbol{\omega}'(t)\times(\mbf{x} - \mbf{q}(t)) - \boldsymbol{\omega}(t) \times \mbf{q}'(t) \right\|_{\mbf{H}^1(\Omega(0))}\notag \\
    &+ \frac{C \sup_{\Omega(0)}|J \mbb{D} \mbb{D}^T|}{C(\Phi(t)) \inf_{\Omega(0)} |J \mbb{D} \mbb{D}^T| } \left\| \mbf{q}''(t) + \boldsymbol{\omega}'(t)\times(\mbf{x} - \mbf{q}(t)) - \boldsymbol{\omega}(t) \times \mbf{q}'(t) \right\|_{\mbf{H}^1(\Omega(0))}.\notag
\end{align}
We then pushforward onto $\Omega(t)$ to transform this into a bound for $\ddt{\mbf{v}}$, using the compatibility of the function space pairs.
Firstly we notice that, due to the evolving domain, taking the pushforward of $\frac{\partial \widehat{\mbf{v}}}{\partial t}$ does not yield $\frac{\partial \mbf{v}}{\partial t}$ but rather the material derivative, $\frac{\partial \mbf{v}}{\partial t} + (\mbf{V} \cdot \grad) \mbf{v}$.
All in all we obtain a bound of the form,
\begin{align}
\begin{aligned}
    \left\|  \ddt{\mbf{v}} \right\|_{\mbf{H}^1(\Omega(t))} &\leq \| \mbf{V}\cdot \grad \mbf{v}\|_{\mbf{H}^1(\Omega(t))}\\
    &+ \frac{C_{\mbf{H}^1}(t)^2}{\inf_{\Omega(0)} |J \mbb{D} \mbb{D}^T|} \left\| \ddt{}(J \mbb{D} \mbb{D}^T) \right\|_{\mbf{L}^\infty(\Omega(0))} N_1(\mbf{q}, \boldsymbol{\omega};t)\\
    &+ \frac{C_{\mbf{H}^1}(t)^2}{\inf_{\Omega(0)} |J \mbb{D} \mbb{D}^T|}\left\| \ddt{J} \right\|_{L^\infty(\Omega(0))} \| \mbb{D} \|_{\mbf{L}^\infty(\Omega(0))} N_1(\mbf{q}, \boldsymbol{\omega};t)\\
    &+ \frac{C_{\mbf{H}^1}(t)^2}{\inf_{\Omega(0)} |J \mbb{D} \mbb{D}^T|}\left\| J\right\|_{L^\infty(\Omega(0))} \left\| \ddt{\mbb{D}} \right\|_{\mbf{L}^\infty(\Omega(0))} N_1(\mbf{q}, \boldsymbol{\omega};t)\\
    &+ \frac{C_{\mbf{H}^1}(t)^2 \sup_{\Omega(0)}|J \mbb{D} \mbb{D}^T|}{C(\Phi(t)) \inf_{\Omega(0)} |J \mbb{D} \mbb{D}^T| }\left\|  \ddt{J} \right\|_{L^\infty(\Omega(0))} \| \mbb{D}\|_{\mbf{L}^\infty(\Omega(0))} N_1(\mbf{q}, \boldsymbol{\omega};t)\\
    &+ \frac{C_{\mbf{H}^1}(t)^2 \sup_{\Omega(0)}|J \mbb{D} \mbb{D}^T|}{C(\Phi(t)) \inf_{\Omega(0)} |J \mbb{D} \mbb{D}^T| }\left\|  J \right\|_{L^\infty(\Omega(0))} \left\| \ddt{\mbb{D}}\right\|_{\mbf{L}^\infty(\Omega(0))} N_1(\mbf{q}, \boldsymbol{\omega};t)\\
    &+ \left(\frac{C_{\mbf{H}^1}(t)^2}{\inf_{\Omega(0)} |J \mbb{D} \mbb{D}^T|} + \frac{C_{\mbf{H}^1}(t)^2 \sup_{\Omega(0)}|J \mbb{D} \mbb{D}^T|}{C(\Phi(t)) \inf_{\Omega(0)} |J \mbb{D} \mbb{D}^T| }\right) N_2(\mbf{q}, \boldsymbol{\omega};t),
    \end{aligned}
    \label{stokes regularity3}
\end{align}
where we have introduced shorthand notation
\begin{gather*}
    N_1(\mbf{q}, \boldsymbol{\omega};t) := \| \mbf{q}'(t) + \boldsymbol{\omega}(t)\times(\mbf{x} - \mbf{q}(t)) \|_{\mbf{H}^{1}(\Omega(t))},\\
    N_2(\mbf{q}, \boldsymbol{\omega};t) := \left\| \mbf{q}''(t) + \boldsymbol{\omega}'(t)\times(\mbf{x} - \mbf{q}(t)) - \boldsymbol{\omega}(t) \times \mbf{q}'(t) \right\|_{\mbf{H}^1(\Omega(t))},
\end{gather*}
for better legibility.
\end{proof}

\subsection{Estimates using the construction in \S \ref{subsection: construction}}

Lastly we end this appendix by turning \eqref{stokes regularity3} into something more conducive to our iterative argument by expressing this bound in terms of $\mbf{q}$ and $ \boldsymbol{\omega}$.
We assume that the map $\Phi(t): \Omega(0) \rightarrow \Omega(t)$ is given by the construction in \S \ref{subsection: construction}.
This means that one has bounds
\begin{gather*}
	\frac{1}{E(\mbf{q}, \boldsymbol{\omega};t)} \leq |\grad \Phi| \leq E(\mbf{q}, \boldsymbol{\omega};t), 
\end{gather*}
where we have introduced $E(\mbf{q}, \boldsymbol{\omega};t)$ given by
\[ E(\mbf{q}, \boldsymbol{\omega};t) :=  \exp \left( C \int_0^t \left\| \frac{\mathrm{d}\mbf{q}}{\mathrm{d}t} + \boldsymbol{\omega} \times (\mbf{x} - \mbf{q}(t)) \right\|_{\mbf{C}^{2,1}(\Omega(s))} \, \mathrm{d}s \right),\]
for more concise notation.
Similarly one finds that
\begin{align*}
    \left\|\ddt{\grad \Phi}\right\|_{\mbf{L}^\infty(\Omega(0))} \leq C \left\| \frac{\mathrm{d}\mbf{q}}{\mathrm{d}t} + \boldsymbol{\omega} \times (\mbf{x} - \mbf{q}(t)) \right\|_{\mbf{C}^{2,1}(\Omega(t))}E(\mbf{q}, \boldsymbol{\omega};t).
\end{align*}
This lets one control all of the terms involving $J, \mbb{D}$ and their time derivatives, but it remains to deal with the term $C(\Phi(t))$.
For this we require Lemma \ref{pullback lemma}, where by inspecting the proof one can verify that
\[ C(\Phi(t)) = C\|J\|_{\mbf{L}^\infty(\Omega(t))}(1 + \|\mbb{D}\|_{\mbf{L}^\infty(\Omega(t))}) \geq \frac{C}{E(\mbf{q}, \boldsymbol{\omega};t)},\]
and likewise
\[ C_{\mbf{H}^1}(t) \leq C E(\mbf{q}, \boldsymbol{\omega};t). \]
It is now straightforward to use these bounds in \eqref{stokes regularity3} to find that now
\begin{align}
\begin{aligned}
   \left\|  \ddt{\mbf{v}} \right\|_{\mbf{H}^1(\Omega(t))} &\leq C \left\| \frac{\mathrm{d}\mbf{q}}{\mathrm{d}t} + \boldsymbol{\omega} \times (\mbf{x} - \mbf{q}(t)) \right\|_{\mbf{C}^{2,1}(\Omega(t))} E(\mbf{q}, \boldsymbol{\omega};t)\\
   &+ C \left\| \frac{\mathrm{d}\mbf{q}}{\mathrm{d}t} + \boldsymbol{\omega} \times (\mbf{x} - \mbf{q}(t)) \right\|_{\mbf{C}^{2,1}(\Omega(t))}^2 E(\mbf{q}, \boldsymbol{\omega};t)\\
   &+ C\left\| \mbf{q}''(t) + \boldsymbol{\omega}'(t)\times(\mbf{x} - \mbf{q}(t)) - \boldsymbol{\omega}(t) \times \mbf{q}'(t) \right\|_{\mbf{H}^1(\Omega(t))} E(\mbf{q}, \boldsymbol{\omega};t),
   \end{aligned}
   \label{stokes regularity4}
\end{align}
for constants independent of $\mbf{q}, \boldsymbol{\omega}$.
This result will be used in the proof of Lemma~\ref{lemma: non degeneracy2}.

\subsection{A perturbation result}
\label{subsection: stokes perturbation}
We now consider the problem of bounding the difference of two solutions of \eqref{stokes problem} defined over two different domains --- this complements some of the missing calculations from Lemma \ref{lemma: fsi contraction}.
Given $(\mbf{q}^i, \boldsymbol{\omega}^i)$ defining a domain $\Omega^i(t)$, as above, we denote the solution of \eqref{stokes problem} by $\mbf{v}^i : \Omega^i(t) \rightarrow \mbb{R}^3$ and $p^i: \Omega^i(t) \rightarrow \mbb{R}$, for $i = 1,2$.
As in Section \ref{section:FSI iteration} we define maps $\Phi^{2,1}(t) : \Omega^2(t) \rightarrow \Omega^1(t)$ given by
\[ \Phi^{2,1}(t) := \Phi^1(t) \circ \Phi^2(t)^{-1},\]
and the inverse of $\Phi^{2,1}(t)$ as
\[ \Phi^{1,2}(t) := \Phi^2(t) \circ \Phi^1(t)^{-1} : \Omega^1(t) \rightarrow \Omega^2(t). \]
We consider the associated Piola transform to be defined as
\[\mathcal{P}_t^{1,2} \boldsymbol{\psi}(\mbf{x}) = {\det(\grad\Phi^{1,2}(t)(\mbf{x}))} (\grad\Phi^{2,1}(t) \boldsymbol{\psi} )\circ \Phi^{1,2}(t), \]
for a function $\boldsymbol{\psi}: \Omega^2(t) \rightarrow \mbb{R}^3$.

Our plan now is to pull $\mbf{v}^2, p^2$ back onto $\Omega^1(t)$ and bound the transformed functions
\[ \widehat{\mbf{v}^2} := \mathcal{P}^{1,2}_t \mbf{v}^2, \quad \widehat{p^2} := p^2 \circ \Phi^{1,2}(t). \]
We find that $\widehat{\mbf{v}^2}, \widehat{p^2}$ solve
\begin{gather*}
    \int_{\Omega^1(t)} \det(\grad \Phi^{1,2}) \grad \Phi^{1,2} (\grad \Phi^{1,2})^T \grad \widehat{\mbf{v}^2} \cdot \grad \boldsymbol{\eta} + \int_{\Omega^1(t)} \det(\grad \Phi^{1,2}) \widehat{p^2} \grad \Phi^{1,2} : \grad \boldsymbol{\eta} = 0,\\
    \int_{\Omega^1(t)} q \Div \widehat{\mbf{v}^2} = 0,
\end{gather*}
for all $\boldsymbol{\eta} \in \mbf{H}^1_0(\Omega^1(t)), q \in L^2(\Omega^1(t))$ and all $t \in [0,T]$, along with the corresponding boundary conditions.
We then observe that the differences $\mbf{v}^1 - \widehat{\mbf{v}^2}$ and $p^1 - \widehat{p^2}$ then solve
\begin{gather*}
\begin{aligned}
    \int_{\Omega^1(t)} \grad (\mbf{v}^1 - \widehat{\mbf{v}^2}) \cdot \grad \boldsymbol{\eta} + \int_{\Omega^1(t)} (p^1 -  \widehat{p^2}) \Div \boldsymbol{\eta} &= \int_{\Omega^1(t)} \left( \det(\grad \Phi^{1,2}) \grad \Phi^{1,2} (\grad \Phi^{1,2})^T - \mbb{I} \right) \grad \widehat{\mbf{v}^2} \cdot \grad \boldsymbol{\eta}\\
    &+ \int_{\Omega^1(t)} \widehat{p^2} \left(\det(\grad \Phi^{1,2})  \grad \Phi^{1,2} - \mbb{I} \right): \grad \boldsymbol{\eta},
    \end{aligned}\\
    \int_{\Omega^1(t)} q \Div (\mbf{v}^1 - \widehat{\mbf{v}^2}) = 0,
\end{gather*}
with boundary conditions
\begin{gather*}
    \mbf{v}^1 - \widehat{\mbf{v}^2} = 0, \text{ on } \Gamma,\\
    \mbf{v}^1 - \widehat{\mbf{v}^2} = \frac{\mathrm{d}(\mbf{q}^1-\mbf{q}^2)}{\mathrm{d}t} + (\boldsymbol{\omega}^1 - \boldsymbol{\omega}^2)\times \mbf{x} - \boldsymbol{\omega}^1 \times (\mbf{q}^1 - \mbf{q}^2) - (\boldsymbol{\omega}^1 - \boldsymbol{\omega}^2)\times \mbf{q}^2, \text{ on } \partial B^1(t).
\end{gather*}

The idea is now to use this PDE to obtain bounds for the differences $\mbf{v}^1 - \widehat{\mbf{v}^2}$ and $p^1 -  \widehat{p^2}$, which follows from elliptic regularity theory for the Stokes problem (see for instance \cite{galdi2011introduction}) provided we can bound the terms on the right-hand side.
This is straightforward, and follows by using Lemma \ref{lemma: gradient difference} as we did in the proof of Lemma \ref{lemma: fsi functional bounds}.
This allows one to obtain a bound of the form
\begin{multline}
    \|\mbf{v}^1 - \widehat{\mbf{v}^2}\|_{\mbf{H}^2(\Omega^1(t))} + \|p^1 - \widehat{p^2}\|_{H^1(\Omega^1(t))}\\
    \leq C\left( 1 + \|\widehat{p^2}\|_{H^1(\Omega^1(t))} + \| \widehat{\mbf{v}^2} \|_{\mbf{H}^2(\Omega^1(t))}  \right)N_3(\mbf{q}^1, \boldsymbol{\omega}^1,\mbf{q}^2, \boldsymbol{\omega}^2;t),
    \label{stokes regularity5}
\end{multline}
where
\begin{multline*}
N_3(\mbf{q}^1, \boldsymbol{\omega}^1,\mbf{q}^2, \boldsymbol{\omega}^2;t):=\\
 \left\|\frac{\mathrm{d}(\mbf{q}^1-\mbf{q}^2)}{\mathrm{d}t} + (\boldsymbol{\omega}^1 - \boldsymbol{\omega}^2)\times \mbf{x} - \boldsymbol{\omega}^1 \times (\mbf{q}^1 - \mbf{q}^2) - (\boldsymbol{\omega}^1 - \boldsymbol{\omega}^2)\times \mbf{q}^2 \right\|_{\mbf{C}^{2,1}(\Omega^1(t))}
\end{multline*}
the constant $C$ depends only on the geometry of $\Omega(0)$.
One may similarly obtain a bound for the difference of the time derivatives, $\ddt{}(\mbf{v}^1 - \widehat{\mbf{v}^2})$, by combining the methods used in this appendix.
We omit details on this calculation as it is incredibly long, and adds nothing of value to the presentation of this paper.

\bibliographystyle{acm}
\bibliography{boundarybib.bib}

\end{document}